\documentclass{memo-l}

\parskip=\smallskipamount

\usepackage{amsfonts,amssymb,amscd,amstext}
\usepackage{graphicx}
\usepackage[dvips]{epsfig}
\usepackage{pstricks,pst-plot}

\usepackage{enumerate}
\usepackage{mathrsfs}
\usepackage[up,bf]{caption}
\usepackage{color}

\newtheorem{theorem}{Theorem}[chapter]
\newtheorem{lemma}[theorem]{Lemma}
\newtheorem{corollary}[theorem]{Corollary}
\newtheorem{proposition}[theorem]{Proposition}
\newtheorem{claim}[theorem]{Claim}

\theoremstyle{definition}
\newtheorem{definition}[theorem]{Definition}
\newtheorem{example}[theorem]{Example}

\theoremstyle{remark}
\newtheorem{remark}[theorem]{Remark}

\numberwithin{section}{chapter}
\numberwithin{equation}{chapter}
\numberwithin{figure}{chapter}

\makeindex

\newcommand\Acal{\mathcal{A}}
\newcommand\Bcal{\mathcal{B}}

\newcommand\Mcal{\mathcal{M}}
\newcommand\Ncal{\mathcal{N}}
\newcommand\Ocal{\mathcal{O}}
\newcommand\Pcal{\mathcal{P}}

\newcommand\Rcal{\mathcal{R}}

\newcommand\bu{\mathbf{u}}
\newcommand\bv{\mathbf{v}}
\newcommand\bx{\mathbf{x}}

\newcommand\bw{\mathbf{w}}

\newcommand\Ascr{\mathscr{A}}

\newcommand\Cscr{\mathscr{C}}

\newcommand\Fscr{\mathscr{F}}

\newcommand\Oscr{\mathscr{O}}

\newcommand\Uscr{\mathscr{U}}
\newcommand\Vscr{\mathscr{V}}

\newcommand\UnNcal{\underline{\mathcal{N}}}

\newcommand\UnM{\underline{M}}
\newcommand\UnN{\underline{N}}
\newcommand\UnX{\underline{X}}
\newcommand\UnY{\underline{Y}}

\newcommand\B{\mathbb{B}}
\newcommand\C{\mathbb{C}}
\newcommand\D{\overline{\mathbb D}}
\newcommand\CP{\mathbb{CP}}
\renewcommand\D{\mathbb D}

\newcommand\K{\mathbb{K}}

\newcommand\N{\mathbb{N}}
\renewcommand\P{\mathbb{P}}
\newcommand\R{\mathbb{R}}

\renewcommand\S{\mathbb{S}}
\newcommand\T{\mathbb{T}}
\newcommand\Z{\mathbb{Z}}

\renewcommand\b{\mathbb{B}}
\renewcommand\c{\mathbb{C}}
\newcommand\cd{\overline{\mathbb D}}
\newcommand\cp{\mathbb{CP}}
\renewcommand\d{\mathbb D}

\newcommand\n{\mathbb{N}}
\renewcommand\r{\mathbb{R}}
\newcommand\s{\mathbb{S}}
\renewcommand\t{\mathbb{T}}
\newcommand\z{\mathbb{Z}}

\newcommand\igot{\mathfrak{i}}

\renewcommand\igot{\mathfrak{i}}

\newcommand\pgot{\mathfrak{p}}
\newcommand\qgot{\mathfrak{q}}

\newcommand\Agot{\mathfrak{A}}
\newcommand\Igot{\mathfrak{I}}

\renewcommand\imath{\igot}

\newcommand\dist{\mathrm{dist}}

\newcommand\wt{\widetilde}

\newcommand\di{\partial}
\newcommand\dibar{\overline\partial}
\newcommand\hra{\hookrightarrow}

\newcommand\Flux{\mathrm{Flux}}
\newcommand\GCMI{\mathrm{GCMI}}
\newcommand\GCMII{\mathrm{GCMI}_{\Igot}}
\newcommand\CMI{\mathrm{CMI}}

\newcommand\OI{\Oscr_{\Igot}}

\newcommand\OmI{\Omega_{\Igot}}

\newcommand\supp{\mathrm{supp}}

\newcommand\Hess{\mathrm{Hess}}

\newcommand\lra{\longrightarrow}
\newcommand\ssharp{\,\sharp\,}
\newcommand\Id{\mathrm{Id}}

\begin{document}

\frontmatter

\title{New complex analytic methods in the study of non-orientable minimal surfaces in $\R^n$}


\author{Antonio Alarc\'{o}n}
\address{Departamento de Geometr\'{\i}a y Topolog\'{\i}a e Instituto de Matem\'aticas (IEMath-GR), Universidad de Granada, Campus de Fuentenueva s/n, E--18071 Granada, Spain}
\curraddr{}
\email{alarcon@ugr.es}
\thanks{A. Alarc\'on is supported by the Ram\'on y Cajal program of the Spanish Ministry of Economy and Competitiveness, and partially supported by the MINECO/FEDER grant no. MTM2014-52368-P, Spain}

\author{Franc Forstneri\v c}
\address{Faculty of Mathematics and Physics, University of Ljubljana, Jadranska 19, SI--1000 Ljubljana, Slovenia, and
Institute of Mathematics, Physics and Mechanics, Jadranska 19, SI--1000 Ljubljana, Slovenia}
\curraddr{}
\email{franc.forstneric@fmf.uni-lj.si}
\thanks{F.\ Forstneri\v c is partially  supported  by the research program P1-0291 and the grants J1-5432 and J1-7256 
from ARRS, Republic of Slovenia}

\author{Francisco J.\ L\'opez}
\address{Departamento de Geometr\'{\i}a y Topolog\'{\i}a e Instituto de Matem\'aticas (IEMath-GR), Universidad de Granada, Campus de Fuentenueva s/n, E--18071 Granada, Spain}
\curraddr{}
\email{fjlopez@ugr.es}
\thanks{Francisco J.\ L\'opez is partially supported by the MINECO/FEDER grant no. MTM2014-52368-P, Spain}

\date{March 3, 2016; final version: May 21, 2017}



\dedicatory{}

\maketitle

\tableofcontents

\begin{abstract}
The aim of this work is to adapt the complex analytic methods originating in modern Oka theory to the study of non-orientable 
conformal minimal surfaces in $\r^n$ for any $n\ge 3$. These methods, which we develop essentially from the first principles, 
enable us to prove that the space of conformal minimal immersions of a given bordered non-orientable surface to $\r^n$ is a 
real analytic Banach manifold (see Theorem \ref{th:intro-structure}), obtain approximation results of Runge-Mergelyan type 
for conformal minimal immersions from non-orientable surfaces (see Theorem \ref{th:intro-Runge} and Corollary \ref{cor:intro-Runge}), 
and show general position theorems for non-orientable conformal minimal surfaces in $\r^n$ (see Theorem \ref{th:generalposition}). 
We also give the first known example of a properly embedded non-orientable minimal surface in $\r^4$; a M\"obius strip 
(see Example \ref{ex:Mobius}). 

All our new tools mentioned above apply to non-orientable minimal surfaces endowed with a fixed choice 
of a conformal structure. This enables us to obtain significant new applications to the global theory of non-orientable 
minimal surfaces. In particular, we construct proper non-orientable conformal minimal surfaces in $\r^n$ with any given conformal structure
(see Theorem \ref{th:applications1} {\rm (i)}), complete non-orientable minimal surfaces in $\r^n$ with arbitrary 
conformal type whose generalized Gauss map is nondegenerate and omits $n$ hyperplanes of $\cp^{n-1}$ in general position 
(see Theorem \ref{th:applications1} {\rm (iii)}), complete non-orientable minimal surfaces bounded by Jordan curves 
(see Theorem \ref{th:intro-Jordan}), and complete proper non-orientable minimal surfaces normalized by bordered surfaces
in $p$-convex domains of $\r^n$  (see Theorem \ref{th:intro3}).
\end{abstract}


\mainmatter


\chapter{Introduction}\label{ch:intro}

%
%

\section{A summary of the main results}
\label{sec:intro}

A minimal surface in $\R^n$ for $n\ge 3$ is a 2-dimensional smoothly embedded or immersed submanifold
which is locally area minimizing. The study of such surfaces was initiated by Euler and Lagrange 
in the 18th century. In 1776, Meusnier showed that a surface in $\R^n$ is minimal if and only if its 
mean curvature vector vanishes at every point. This observation ties the theory of minimal
surfaces with differential geometry. For background and history, 
we refer e.g.\ to the monographs  by Osserman \cite{Osserman-book}, Meeks and P\'erez 
\cite{MeeksPerez2012AMS}, and Colding and Minicozzi \cite{ColdingMinicozzi2011-book}.

Given an immersion $X\colon N\to \R^n$ from a smooth surface $N$, there is a unique conformal structure 
on $N$ which makes $X$ conformal (angle preserving); namely, the one determined by the pullback $X^* ds^2$ of the standard 
Euclidean metric $ds^2$ on $\R^n$. The immersion is said to be {\em complete} if $X^* ds^2$ is a complete metric on $N$.
A surface $N$ endowed with a conformal structure will be denoted by the  
corresponding caligraphic letter $\Ncal$. It is classical that a conformal immersion is minimal, in the sense 
that it parametrizes a minimal surface in $\R^n$, if and only if it is harmonic. 
If in addition $\Ncal$ is orientable, then it carries the structure of a Riemann surface;
this makes it possible to study harmonic maps $\Ncal\to \R^n$  as real parts of holomorphic 
maps $\Ncal\to \C^n$. It is therefore not surprising that the orientable case has been studied  extensively by 
exploiting complex analytic methods. 

Non-orientable minimal surfaces have also attracted considerable interest. Although the study of non-orientable minimal 
surfaces in $\r^3$ was initiated by Lie back in 1878 (see \cite{Lie1878MA}), the global theory, and in particular the use of 
complex analytic tools in its development, only began in the eighties of the last century with the seminal paper
of Meeks \cite{Meeks1981DMJ}; see also L\'opez \cite{Lopez1993,Lopez1996} and the references therein. 
Every surface is locally orientable, and hence we can study harmonic maps at least locally as real parts of
holomorphic maps. Recently, Alarc\'on  and L\'opez \cite{AlarconLopez2015GT} renewed the interest in 
this subject by proving the Mergelyan approximation theorem for non-orientable minimal surfaces  in $\R^3$. 

The aim of this work is to adapt the complex analytic methods originating  
in modern Oka theory (see \cite{Forstneric2011-book,Forstneric2013AFST} for the latter)
to the study of non-orientable conformal minimal surfaces in $\r^n$ for any $n\ge 3$. 
In the orientable case, these methods have been introduced in the recent works \cite{AlarconDrinovecForstnericLopez2015PLMS,AlarconForstneric2014IM,AlarconForstneric2015MA,
AlarconForstneric2015AS,AlarconForstnericLopez2016MZ}. 
As in \cite{AlarconLopez2015GT}, the main idea  is to pass to the oriented double cover which is a 
Riemann surface endowed with a fixed-point-free antiholomorphic involution; see Section \ref{sec:basics}.
For reasons that will become apparent shortly,  the non-orientable case is considerably more 
delicate both from the geometric and topological point of view, as well as from the purely technical side. 

The core of the work is Chapter \ref{ch:sprays} 
where we develop the relevant complex analytic methods essentially from the first principles.
These methods represent the main new technical contribution of the paper. The subsequent chapters 
are devoted to their implementation in the construction of non-orientable conformal minimal surfaces 
in $\R^n$ enjoying a variety of additional properties. 
Here is a brief summary of our main results; a more detailed explanation can be found
in Sections \ref{sec:intro-Runge}--\ref{sec:intro-proper} and in the main body of the work:
\begin{itemize}
\item  the space of conformal minimal immersions of a bordered non-orientable surface into $\R^n$
is a real analytic Banach manifold (see Theorem \ref{th:structure});
\vspace{1mm}
\item Runge and Mergelyan type approximation theorems hold for non-orientable conformal minimal surfaces 
(see Theorems \ref{th:Mergelyan} and  \ref{th:Mergelyan1});
\vspace{1mm}
\item general position results hold for non-orientable conformal minimal surfaces 
(see Theorems \ref{th:generalposition} and \ref{th:generalposition2});
\vspace{1mm}
\item construction of proper non-orientable conformal minimal surfaces in $\R^n$
with any given conformal structure (see Theorem \ref{th:Iproper});
\vspace{1mm}
\item existence and approximation results for complete non-orientable minimal surfaces in $\r^n$ 
with arbitrary conformal type and $n-2$ given coordinate functions (see Theorem \ref{th:fixed-derivative}
and Corollary \ref{co:fixed}); 
\vspace{1mm}
\item  construction of complete non-orientable minimal surfaces in $\r^n$ with arbitrary conformal type 
whose generalized Gauss map is nondegenerate and omits $n$ hyperplanes of $\cp^{n-1}$ in general position 
(see Corollary \ref{co:gaussmap});
\vspace{1mm}
\item existence of non-orientable complete minimal  surfaces with Jordan boundaries in $\R^n$ for any $n\ge 3$
(see Theorem \ref{th:Jordan});
\vspace{1mm} 
\item existence of complete proper non-orientable minimal surfaces normalized by bordered surfaces in $p$-convex domains 
(see Theorem \ref{th:convex}).
\end{itemize}

One of the major novelties of our approach, when compared with the existing works on the subject, is that 
our results hold for any given conformal structure on the source surface.
It is well known that problems typically become substantially easier
if one allows changes of the complex ($=$conformal)  structure. 
In the context of Oka theory, results of the latter type are referred to as {\em soft Oka theory}. 
To give a simple example, most continuous maps $\Ncal\to Y$ from an open non-simply connected 
Riemann surface $\Ncal$ to another complex manifold $Y$ are not homotopic 
to a holomorphic map; this already fails for maps between annuli, or for maps to $\C\setminus\{0,1\}$.
On the other hand, every continuous map admits a holomorphic
representative in the same homotopy class if one allows a homotopic deformation 
of the complex structure on the source surface $\Ncal$. 
(See \cite[Theorem 9.10.1, p.\ 446]{Forstneric2011-book} for a substantially
more general result in this direction.) 

A glance at the above list shows that most of our results pertain to the global theory of non-orientable 
minimal surfaces in $\r^n$. This part of the theory, which has been a very active focus of research during the 
last five decades (see e.g.\ the recent surveys \cite{Hoffman2005CMP,MeeksPerez2012AMS}),
concerns the questions how global assumptions such as Riemannian completeness, embeddedness, or properness 
(in $\R^n$, or in a given domain or class of domains in $\R^n$)
influence the topological, geometrical and conformal properties of minimal surfaces. 
Ours seems to be the first systematic investigation of this part of the theory in the case of non-orientable 
minimal surfaces in Euclidean spaces of arbitrary dimension $n\ge 3$.

%
%

\section{Basic notions of minimal surface theory} 
\label{sec:basics}

In order to describe the content of this work in more detail, we need to recall some of the basics 
of the theory of minimal surfaces; we refer to Chapter \ref{ch:prelim} for further details and to
Osserman \cite{Osserman-book} for a more complete treatment.

A non-orientable minimal surface in $\R^n$ is the image of a conformal minimal
immersion $\UnX\colon \UnNcal\to \R^n$ from a non-orientable surface, $\UnNcal$,
endowed with a conformal structure. Just as in the orientable case, a conformal immersion into $\R^n$ 
is minimal  if and only if it is harmonic. There is a 2-sheeted oriented covering $\iota\colon \Ncal \to\UnNcal$
of $\UnNcal$ by a  Riemann surface, $\Ncal$, endowed with a  fixed-point-free antiholomorphic involution $\Igot\colon \Ncal\to\Ncal$
(the deck transformation of the covering $\iota\colon \Ncal \to\UnNcal$) such that $\UnNcal=\Ncal/\Igot$. 
Any conformal minimal immersion $\UnX\colon \UnNcal\to\R^n$ arises from
a conformal minimal immersion $X\colon \Ncal\to \r^n$ which is {\em $\Igot$-invariant}, in the sense that
\begin{equation}\label{eq:XisIgotinvariant}
    	X\circ \Igot=X.
\end{equation}

The bridge between minimal surface theory in $\R^n$ and complex analysis in $\C^n$ is spanned by the
classical {\em Weierstrass representation formula} which we now recall; 
see Osserman \cite{Osserman-book} for the orientable framework, 
and Meeks \cite{Meeks1981DMJ} and Alarc\'on-L\'opez \cite{AlarconLopez2015GT}  for the non-orientable one 
in dimension $n=3$. Further details can be found in Section \ref{sec:CMI}.

Let $d=\di+\dibar$ be the decomposition of the exterior differential  on a Riemann surface $\Ncal$ 
as the sum of the $\C$-linear part $\di$ and the $\C$-antilinear part $\dibar$. Thus, holomorphic maps
$Z\colon \Ncal\to\C^n$ are those satisfying $\dibar Z=0$, while harmonic maps $X\colon\Ncal\to\R^n$
are characterized by the equation $\di\dibar X=-\dibar\di X=0$.
Given a conformal minimal immersion $X=(X_1,\ldots,X_n)\colon\Ncal \to\R^n$, the $\C^n$-valued $1$-form 
\[
	\Phi := \di X = (\phi_1,\ldots,\phi_n),\quad \text{with $\phi_j=\di X_j$ for $j=1,\ldots,n$,}
\] 
is holomorphic (since $X$ is harmonic), it does not vanish anywhere on $\Ncal$ (since
$X$ is an immersion), and it satisfies the nullity condition 
\begin{equation}\label{eq:null0}
	(\phi_1)^2 + (\phi_2)^2+\cdots + (\phi_n)^2=0
\end{equation}
characterizing the conformality of $X$. If the immersion $X$ is also $\Igot$-invariant (see \eqref{eq:XisIgotinvariant}),
then the $1$-form $\Phi=\di X$ is $\Igot$-invariant in the sense that
\begin{equation}\label{eq:I-invariantform}
   	 \Igot^*(\Phi)=\overline{\Phi}.
\end{equation}
Note that $dX=2\Re(\di X)=2\Re \Phi$, and hence $\Phi$ has vanishing real periods:
\begin{equation}\label{eq:vanishing-periods}
	\Re \int_\gamma \Phi =0\quad \text{for every closed path $\gamma$ in $\Ncal$}.
\end{equation}
Conversely, if $\Phi=(\phi_1,\ldots,\phi_n)$ is a nowhere vanishing holomorphic $1$-form
on $\Ncal$ satisfying conditions \eqref{eq:null0},  \eqref{eq:I-invariantform}, and \eqref{eq:vanishing-periods},
then the $\R^n$-valued $1$-form $\Re \Phi$ is exact on $\Ncal$ and satisfies $\Igot^* (\Re \Phi)=\Re\Phi$. Hence, 
for any fixed point $p_0\in\Ncal$ the integral
\begin{equation}\label{eq:integralofPhi}
	X(p) = X(p_0) + 2\int_{p_0}^p  \Re \Phi : \Ncal\lra \R^n,\quad p\in\Ncal
\end{equation}
is a well-defined $\Igot$-invariant conformal minimal immersion satisfying $\di X=\Phi$.

Let us fix a nowhere vanishing  $\Igot$-invariant  holomorphic $1$-form $\theta$ on $\Ncal$
(for the existence of such a $1$-form, see \cite[Corollary 6.5]{AlarconLopez2015GT}).  
Given $\Phi$ as above, the holomorphic map
\[
	f:= \Phi/\theta = (f_1,\ldots,f_n)\colon \Ncal\lra \c^n
\]
is {\em $\Igot$-invariant} in the sense that
\begin{equation}\label{eq:I-invariantmap}
   	 f\circ \Igot=\bar f,
\end{equation}
and it assumes values in the punctured {\em null quadric}
\begin{equation} \label{eq:null-quadric0}
	\Agot_*= \left\{z=(z_1,\ldots,z_n)\in \C^n: z_1^2+z_2^2+\cdots + z_n^2 =0\right\}\setminus \{0\}
\end{equation}
as seen from the nullity condition \eqref{eq:null0}. 

Conversely, an $\Igot$-invariant holomorphic map $f\colon \Ncal\to \Agot_*$ into the punctured null quadric $\Agot_*$ determines a 
holomorphic $1$-form $\Phi=f\theta$ satisfying conditions  \eqref{eq:null0} and \eqref{eq:I-invariantform}, 
but  not necessarily the period condition \eqref{eq:vanishing-periods}.

The above discussion shows that the Weierstrass representation reduces the study of conformal minimal 
immersions $\UnNcal \to\r^n$ from an open non-orientable conformal surface to the study of $\Igot$-invariant 
holomorphic maps $f\colon \Ncal\to \Agot_*$ from the $2$-sheeted oriented cover $(\Ncal,\Igot)$ of 
$\UnNcal=\Ncal/\Igot$ into the punctured null quadric  \eqref{eq:null-quadric0} such that the $1$-form $\Phi=f\theta$ 
satisfies the period vanishing conditions \eqref{eq:vanishing-periods}. (Here, $\theta$ can be any nowhere 
vanishing $\Igot$-invariant holomorphic $1$-form on the Riemann surface $\Ncal$.)
Indeed, the map $X\colon \Ncal\to\R^n$ given by  \eqref{eq:integralofPhi} is then an $\Igot$-invariant 
conformal minimal immersion, and hence it determines a conformal minimal immersion $\UnX\colon \UnNcal\to\R^n$. 
This is the main vantage point of our analysis.

We are now ready for a more precise description of our main results.

%
%

\section{Approximation and general position theorems} \label{sec:intro-Runge}

We begin with a Banach manifold structure theorem.

%
%

\begin{theorem}[Banach manifolds of conformal minimal immersions]
\label{th:intro-structure}
If $\underline \Ncal$ is a compact, smoothly bounded, non-orientable surface with a conformal structure, 
then for any pair of integers $n\ge 3$ and $r\ge 1$ the set of all conformal minimal immersions $\underline \Ncal\to\R^n$ 
of class $\Cscr^r(\UnNcal)$ is a real analytic Banach manifold. 
\end{theorem}

For a more precise result, see Theorem \ref{th:structure} and Corollary \ref{cor:structure}.
The analogous result in the orientable case is \cite[Theorem 3.1 (b)]{AlarconForstnericLopez2016MZ};
see also \cite[Theorem 3.2 (b)]{AlarconForstneric2014IM} for the corresponding result concerning 
null  curves and more general directed holomorphic immersion $\Ncal\to\C^n$.
In \cite{AlarconForstnericLopez2016MZ} we had to exclude {\em flat} conformal minimal immersions, 
i.e., those whose image lies in an affine $2$-plane in $\R^n$. This is unnecessary here 
since a non-orientable surface does not immerse into $\R^2$, and hence every immersion
into $\R^n$ for $n>2$ is nonflat.

It was shown in \cite{Forstneric2007AJM} that many natural spaces of holomorphic maps from 
a compact strongly pseudoconvex Stein domain $D$ to an arbitrary complex manifold $Y$ carry the structure of a
complex Banach manifold. This holds in particular for the space $\Ascr^r(D,Y)$ of holomorphic maps $\mathring D\to Y$ 
which are smooth of class $\Cscr^r(D)$ up to the boundary for some $r\in\Z_+=\{0,1,2,\ldots\}$.
The proof relies on the construction of a tubular Stein neighborhood of the graph of
a given map in the product manifold $D\times Y$, fibered over the base domain $D$. 
These results apply in particular to maps from bordered Riemann surfaces.

In the proof of Theorem \ref{th:intro-structure} (see Section \ref{sec:structure}) we proceed more directly.
Consider the oriented double cover $(\Ncal,\Igot)$ of $\UnNcal$. We begin by showing that 
the space $\Ascr_\Igot(\Ncal,\Agot_*)$  of all continuous $\Igot$-invariant maps $\Ncal \to \Agot_*$ 
that are holomorphic in the interior $\mathring\Ncal=\Ncal\setminus b\Ncal$ is a real analytic Banach submanifold of the 
real Banach space $\Ascr_\Igot(\Ncal,\C^n)$ of all maps $\Ncal \to \C^n$ of the same class.
To see this, we apply the classical ideal theory for the algebra $\Ascr(\Ncal)$ of continuous functions that 
are holomorphic in $\mathring\Ncal$, adapted to $\Igot$-invariant functions (see Lemma \ref{lem:maximal-ideals}),
to show that the defining equation $f_1^2+\ldots+ f_n^2=0$ for $\Ascr_\Igot(\Ncal,\Agot_*)$ has maximal rank 
whenever the component functions $f_1,\ldots,f_n$ have no common zeros; the conclusion then follows by
the implicit function theorem. Next, we use period dominating $\Igot$-invariant sprays 
(see Proposition \ref{prop:period-dominating-spray}) to show that the period vanishing equation 
\eqref{eq:vanishing-periods}
 is of maximal rank on $\Ascr_\Igot(\Ncal,\Agot_*)$,
and hence it defines a Banach submanifold of finite codimension. The same analysis
applies to maps of class $\Cscr^r$ for any $r\in\Z_+$.

%
%
%
%

Chapter \ref{ch:Runge} is devoted to the proof of Runge and Mergelyan type approximation theorems
for conformal minimal immersions of non-orientable surfaces into $\R^n$. 
A crucial point which makes the approximation theory possible is that the punctured null quadric 
$\Agot_*\subset\C^n$  \eqref{eq:null-quadric0} is an {\em Oka manifold}; see
\cite[Example 4.4]{AlarconForstneric2015MA}. More precisely, $\Agot_*$  is a homogeneous 
space of the complex Lie group $\C_*\oplus {\mathrm O}_n(\C)$, where ${\mathrm O}_n(\C)$ is the orthogonal
group over the complex number field $\C$. For the theory of Oka manifolds, see the monograph 
\cite{Forstneric2011-book} and the surveys \cite{ForstnericLarusson2011NY,Forstneric2013AFST}.

Here we shall content ourselves by stating the following simplified version of 
Theorems \ref{th:Mergelyan} and \ref{th:Mergelyan1}.
The latter results give approximation of generalized conformal minimal immersions 
on admissible $\Igot$-invariant Runge sets in $\Ncal$, including approximation on curves.

%
%
\begin{theorem}[Runge approximation theorem for $\Igot$-invariant conformal minimal immersions]
\label{th:intro-Runge}
Let $\Ncal$ be an open Riemann surface with a fixed-point-free antiholomorphic involution $\Igot \colon \Ncal\to\Ncal$, 
and let $D\Subset\Ncal$ be a relatively compact, smoothly bounded, $\Igot$-invariant domain (i.e., $\Igot(D)=D$) 
which is Runge in $\Ncal$. Then, any $\Igot$-invariant conformal minimal immersion $X=(X_1,\ldots,X_n)\colon \bar D\to\r^n$ 
$(n\ge 3)$ of class $\Cscr^1(\bar D)$ can be approximated in the $\Cscr^1(\bar D)$ topology by $\Igot$-invariant 
conformal minimal immersions  $\wt X=(\wt X_1,\ldots,\wt X_n)\colon \Ncal\to\r^n$. 

Furthermore, if we assume that the components $X_3,\ldots,X_n$ extend to harmonic functions on $\Ncal$,  
then we can choose $\wt X$ as above such that $\wt X_j=X_j$ for $j=3,\ldots,n$.
\end{theorem}

Recall that a compact set $K$ in an open Riemann surface $\Ncal$  is said to be {\em Runge} 
(or {\em holomorphically convex})  if any holomorphic function on a neighborhood of $K$ can be approximated, 
uniformly on $K$, by functions that are holomorphic on $\Ncal$. By classical theorems of 
Runge \cite{Runge1885AM} and Bishop \cite{Bishop1958PJM}, this holds if and only if the complement 
$\Ncal\setminus K$ has no connected component with compact closure. 

By passing to the quotient $\UnNcal=\Ncal/\Igot$, Theorem \ref{th:intro-Runge}
immediately gives the following corollary. For the statement concerning embeddings,
see Theorem \ref{th:generalposition}. 

%
%
\begin{corollary}[Runge approximation theorem for non-orientable conformal minimal surfaces]
\label{cor:intro-Runge}
Assume that $\UnNcal$  is a non-orientable surface  with a conformal structure and $D\Subset \UnNcal$ is a relatively compact 
domain with $\Cscr^1$ boundary such that $\UnNcal\setminus \bar D$ has  no relatively compact connected components.
Then, every conformal minimal immersion $\bar D\to \R^n$ $(n\ge 3$) of class $\Cscr^1(\bar D)$ can be 
approximated in the $\Cscr^1(\bar D)$ topology by conformal minimal immersions  
$\UnNcal\to\r^n$ (embeddings if $n\ge 5$). 
\end{corollary}

These approximation theorems are of crucial importance to the theory; indeed, all
subsequent results presented in Chapters \ref{ch:general-position} and \ref{ch:applications}
depend on them.

The case $n=3$ of Theorem \ref{th:intro-Runge} is due to  Alarc\'on-L\'opez \cite{AlarconLopez2015GT}. 
For orientable surfaces, see Alarc\'on-L\'opez \cite{AlarconLopez2012JDG} for $n=3$ and 
Alarc\'on-Fern\'andez-L\'opez \cite{AlarconFernandezLopez2013CVPDE} and Alarc\'on-Forstneri\v c-L\'opez  
\cite{AlarconForstnericLopez2016MZ} for arbitrary dimension. 

The second part of Theorem \ref{th:intro-Runge}, 
where we fix $n-2$ components of the approximating immersions, is very useful for the construction of 
non-orientable minimal surfaces with additional global properties such as completeness or properness; we shall return  
to this subject later in this introduction.

Theorem \ref{th:intro-Runge} is proved in a similar way as the corresponding result in the
orientable case (see \cite[Theorem 5.3]{AlarconForstnericLopez2016MZ}, and also
\cite[Theorem 7.2]{AlarconForstneric2014IM} for more general directed holomorphic immersions);
however, the details are considerably more involved in the non-orientable case. 
The basic step in the proof is provided by Proposition \ref{prop:noncritical} which gives approximation 
of the $\Igot$-invariant holomorphic map $f=2\di X/\theta \colon \bar D\to \Agot_*$ 
by a map with the same properties defined on the union 
of $\bar D$ with an $\Igot$-symmetric pair of bumps attached to $\bar D$. 
This depends on the method of gluing pairs of {\em $\Igot$-invariant sprays} 
on a {\em very special $\Igot$-invariant Cartan pair} in $\Ncal$;
see  Section \ref{sec:gluing}.  Finitely many applications of Proposition \ref{prop:noncritical} 
complete the proof of Theorem \ref{th:intro-Runge} in the noncritical case 
(i.e., when there is no change of topology and $M$ deformation retracts onto $\bar D$); 
see Corollary \ref{cor:noncritical}. The critical case
amounts to passing an $\Igot$-invariant pair of Morse critical points of a strongly 
subharmonic  $\Igot$-invariant exhaustion function $\rho\colon \Ncal \to \R$; see
the proof of Theorem \ref{th:Mergelyan0} in Section \ref{sec:Mergelyan0}.
The geometrically different situations regarding the change of topology of the sublevel set of 
$\rho$ demand a considerably finer analysis than in the orientable case.
The last part of  Theorem \ref{th:intro-Runge}, concerning approximation with fixed
$n-2$ coordinates, is proved in Section \ref{sec:Mergelyan1}; see the proof of 
Theorem \ref{th:Mergelyan1}.

%
%
%
%
In Chapter \ref{ch:general-position} we prove the following general position result
for conformal minimal immersions from non-orientable surfaces.

\begin{theorem}[General position theorem]
\label{th:generalposition}
Let $\UnNcal$ be an open non-orientable surface endowed with a conformal structure, and let
$\UnX\colon \UnNcal \to\r^n$ be a conformal minimal immersion for some $n\ge 3$.

If $n\ge 5$ then $\UnX$ can be approximated uniformly on compacts by conformal minimal embeddings 
$\wt \UnX\colon \UnNcal \hra \r^n$. If $n=4$ then $\UnX$ can be approximated uniformly on compacts 
by conformal minimal immersions $\wt \UnX$ with simple double points. 

The same results hold if $\UnNcal$ is a compact non-orientable conformal 
surface with smooth boundary $b\UnNcal\ne \emptyset$ and $X$ is of class $\Cscr^r(\UnNcal)$ for 
some $r\in\n=\{1,2,3,\ldots\}$; in such case, the approximation takes place in the $\Cscr^r(\UnNcal)$ topology.
Furthermore, for every $n\ge 3$ we can approximate $\UnX$ by conformal minimal immersions
$\wt \UnX\colon \UnNcal \to\r^n$ such that $\wt \UnX\colon b\UnNcal\hra\R^n$ is an embedding
of the boundary.
\end{theorem}

Theorem \ref{th:generalposition} is a corollary to Theorem  \ref{th:generalposition2} which 
gives a general position result for $\Igot$-invariant conformal minimal immersions 
from Riemann surfaces to $\R^n$. The proof also shows that a generic conformal
minimal immersion $\UnNcal\to\R^4$ has isolated normal crossings. However, in dimension 4 such double
points are stable under small deformations, and it is not clear whether one could remove them
even by big deformations. It is a major open problem whether every (orientable or non-orientable)
conformal surface admits a conformal minimal embedding into $\R^4$. 

The corresponding problem is well known and widely open also
for holomorphic embeddings of open Riemann surfaces \cite{BellNarasimhan1990EMS}
(which are a special case of conformal minimal embeddings):

{\em Does every open Riemann surface embed holomorphically into $\C^2$?} 

One of the most general recent results in this direction, due to  Forstneri\v c and Wold \cite{ForstnericWold2013APDE},
is that every planar domain (possibly infinitely connected) with at most finitely many isolated boundary points admits such an embedding.
For a discussion of this subject and references to other recent results, see \cite{AlarconForstnericLopez2016MZ}.

The analogue of Theorem \ref{th:generalposition} for orientable surfaces is 
\cite[Theorem 1.1]{AlarconForstnericLopez2016MZ}; here we adapt the proof to the non-orientable case. 
The main point is to prove a general position theorem for conformal minimal immersions 
from compact bordered surfaces (see \cite[Theorem 4.1]{AlarconForstnericLopez2016MZ} for the orientable case; 
for embeddings on the boundary, see also \cite[Theorem 4.5 (a)]{AlarconDrinovecForstnericLopez2015PLMS}).
This is accomplished by applying transversality methods along with period dominating sprays, a method developed in the
proof of Theorem \ref{th:intro-Runge}. 
The general case for an open surface is obtained by an induction with respect to a suitable exhaustion 
by compact subdomains and the Runge approximation theorem furnished by Theorem \ref{th:intro-Runge}. 


%
%

\section{Complete non-orientable minimal surfaces with Jordan boundaries}
\label{sec:intro-Jordan}

In this section, we focus on the existence of complete non-orientable minimal surfaces with 
Jordan boundaries and methods for their construction. 

Douglas \cite{Douglas1932TAMS} and Rad\'o \cite{Rado1930AM}  independently proved in 1931 that any Jordan curve in $\r^n$ 
for $n\ge 3$ bounds a minimal surface, thereby solving the classical {\em Plateau problem}. 
By the isoperimetric inequality, minimal surfaces in $\r^n$ spanning rectifiable Jordan curves cannot be complete. 
It has been a long standing open problem whether there exist
 (necessarily nonrectifiable) Jordan curves in $\R^n$ whose solution to the Plateau problem are
 complete minimal surfaces;  equivalently, whether there are complete minimal surfaces in $\R^n$
 bounded by Jordan curves. In the orientable case, the authors together with B. Drinovec
 Drnov{\v{s}}ek introduced another important complex analytic method to this problem, 
 namely the Riemann-Hilbert boundary value problem 
 adapted to the case of holomorphic null curves and conformal minimal surfaces, and thereby 
 proved that every bordered Riemann surface is the underlying
 conformal structure of a complete minimal surface bounded by Jordan curves 
 (see \cite[Theorem 1.1]{AlarconDrinovecForstnericLopez2015PLMS}; see also \cite{Alarcon2010TAMS,AlarconLopez2013IJM,AlarconLopez2014AGMS}
 for previous partial results in this direction). In this paper, we prove the following analogous existence result
 for non-orientable minimal surfaces.
 
%
%
%
%
\begin{theorem} [Non-orientable complete minimal surfaces bounded by Jordan curves] 
\label{th:intro-Jordan}
Let $\UnNcal$ be a compact non-orientable bordered surface with non\-empty boundary $b\UnNcal\ne\emptyset$,
endowed with a conformal structure, and let $n\ge3$ be an integer.
Every conformal minimal immersion $\UnX \colon \UnNcal\to\r^n$ can be approximated uniformly on $\UnNcal$ 
by continuous maps $\UnY\colon\UnNcal\to\r^n$ such that $\UnY|_{\mathring \UnNcal}\colon\mathring\UnNcal\to\r^n$ 
is a complete conformal minimal immersion and $\UnY|_{b\UnNcal}\colon b\UnNcal\to\r^n$ is a topological embedding. 
In particular, the boundary $\UnY(b\UnNcal) \subset \R^n$ consists of finitely many Jordan curves.
If $n\ge 5$ then there exist topological embeddings $\UnY\colon \UnNcal \hra \r^n$ with these properties.
\end{theorem}

Theorem \ref{th:intro-Jordan} is a corollary to Theorem \ref{th:Jordan} in Section \ref{sec:Jordan}
which furnishes a $\Igot$-invariant map $Y \colon \Ncal \to \r^n$ with the analogous properties 
on a 2-sheeted oriented covering $(\Ncal,\Igot)$ of $\UnNcal$.  

Theorem \ref{th:intro-Jordan} gives the first known examples of complete bounded non-orientable minimal surfaces 
with Jordan boundaries. It is related to the Calabi-Yau problem for minimal surfaces, dating back to Calabi's conjectures 
from 1966 \cite{Calabi1965P} (see also Yau \cite{Yau1982AMS,Yau2000AMS}), asking about 
the topological, geometric, and conformal properties of complete bounded minimal surfaces.
Earlier constructions of boun\-ded complete ({\em orientable or non-orientable}) minimal surfaces in $\R^n$ were based
on classical Runge's approximation theorem applied on labyrinths of compact sets in the given surface $\Ncal$,
at the cost of having to cut away pieces of $\Ncal$ in order to keep the image $X(\Ncal)\subset\R^n$ suitably bounded.
The original construction of this type, with $\Ncal$ the disk, is due to
Nadirashvili \cite{Nadirashvili1996IM}; the first non-orientable examples were given by L\'opez, Mart\'in, and Morales 
in \cite{LopezMartinMorales2006TAMS}.  A fairly complete set of references to subsequent works can be found in the papers
\cite{AlarconDrinovecForstnericLopez2015PLMS,AlarconForstneric2015MA,AlarconForstneric2015AS, 
AlarconLopez2013MA}.  

%
%
%
%

The proof of Theorem \ref{th:intro-Jordan}, and of the related Theorem \ref{th:Jordan}, 
uses approximate solutions to a Riemann-Hilbert type boundary value problem for 
$\Igot$-invariant conformal minimal immersions; see Theorem \ref{th:RH}.
This method allows one to stretch a piece of the image surface $X(\Ncal) \subset \R^n$ so as to increase the 
immersion-induced boundary distance in $\Ncal$ near a given arc in the boundary $b\Ncal$, 
while at the same time controlling the placement of the whole surface in $\R^n$.
In this way, one avoids having to cut away pieces of $\Ncal$ during the construction,
thereby obtaining results for any given conformal structure on $\Ncal$. As we show
in this paper, this technique can also be adapted to the $\Igot$-invariant situation
which is needed for the application to non-orientable bordered surfaces.

The term {\em Riemann-Hilbert problem} classically refers to a family of free boundary value problems
for holomorphic maps and, by analogy, for certain other classes of maps. 
For a brief historical background, especially in connection with the applications 
to the theory of null curves and conformal minimal immersions,
we refer to the survey by Alarc\'on and Forstneri\v c \cite{AlarconForstneric2015AS}.
Let us  explain this problem in the original complex analytic setting.

Assume that $\Ncal$ is a compact bordered Riemann surface and $Z\colon \Ncal\to\C^n$ is a 
holomorphic map. (It suffices that $Z$ be of class $\Ascr(\Ncal,\C^n)$; also, the target $\C^n$ 
may be replaced by an arbitrary complex manifold.) Let $I\subset b\Ncal$ be a boundary arc. 
Consider a continuous family of holomorphic disks $G_p\colon \cd\to\C^n$
$(p\in I)$ such that $G_p(0)=Z(p)$ for every $p\in I$.
Let $I'$ be a smaller arc contained in the relative interior of $I$, and let $U\subset \Ncal$ be a 
small open neighborhood of $I'$ with $U\cap b\Ncal\subset I$. 
Given this configuration, one can find a holomorphic map 
$\wt Z\colon \Ncal \to\C^n$ which is uniformly close to $Z$ on $\Ncal \setminus U$, 
$\wt Z(U)$ is close to $Z(U)\cup \bigcup_{p\in I} G_p(\cd)$, 
and the curve $\wt Z(I')$ is close to the cylinder $\bigcup_{p\in I'} G_p(b\d)$. 
All approximations can be made as close as desired.
A simple proof in the basic case when $\Ncal$ is the disk can be found in 
\cite[Proposition 2.1]{ForstnericGlobevnik2001MRL} and in 
\cite[Lemma 3.1]{DrinovecForstneric2012IUMJ}. The general case is obtained by gluing sprays as 
explained in the proof of Theorem  \ref{th:RH}.

This method is quite versatile and can be used for several purposes. One of them is the construction of proper
holomorphic maps from bordered Riemann surfaces. 
In this case, one chooses the disks $G_p$ such that for a given exhaustion function
$\rho$ on the  target manifold, the composition $\rho\circ G_p$ has a minimum at the center point $0\in\cd$
and it satisfies an estimate of the form $\rho(G_p(\zeta)) \ge \rho(G_p(0))+c|\zeta|^2$
for some positive constant $c>0$ independent of the point $p\in I'$.  This requires that $\rho$ has a suitable convexity property
(more precisely, its Levi form should have at least two positive eigenvalues at every point).
The effect of the Riemann-Hilbert modification is that we push the boundary points $p\in I'$ 
in a controlled way to higher level sets of $\rho$ without decreasing $\rho$ much anywhere else
along the image of the surface. This is how Drinovec Drnov\v sek and Forstneri\v c \cite{DrinovecForstneric2007DMJ}
constructed proper holomorphic maps from the interior of any given bordered Riemann surface 
to any complex manifold satisfying a suitable $q$-convexity condition.  For a generalization to higher
dimensional source domains, see \cite{DrinovecForstneric2010AJM} and the references therein.

The Riemann-Hilbert method can also be used to construct bounded complete
complex curves in $\C^n$ for any $n>1$. This is how Alarc\'on and Forstneri\v c 
\cite{AlarconForstneric2013MA} found a proper complete holomorphic immersion 
from the interior of any bordered Riemann surface to the unit ball of $\C^2$. In this case,
the holomorphic disks $G_p$ are chosen to be linear, complex orthogonal to the vector $Z(p)\in\C^n$
for each $p\in I$, and of uniform size $\epsilon>0$.  
In view of Pythagoras' theorem, the Riemann-Hilbert modification increases 
the immersion-induced boundary distance  from an interior point of $\Ncal$ 
to the arc $I'$ by approximately $\epsilon$, while at the same time the outer radius of the 
immersion increases by at most $C\epsilon^2$ for some constant $C>0$. 
A recursive use of this technique on short  boundary arcs, along with certain other complex analytic tools, 
leads to the mentioned theorem. 

Both these applications of the Riemann-Hilbert method  have recently been extended
to the constructions of proper and/or complete null holomorphic curves in $\C^n$
and conformal minimal surfaces in $\R^n$ for any $n\ge 3$.
The Riemann-Hilbert method was first introduced to the study of null curves in dimension $n=3$ 
by Alarc\'on and Forstneri\v c \cite{AlarconForstneric2015MA}. It was improved 
and extended to constructions of complete bounded null curves and conformal minimal immersions
in any dimension $n\ge 3$ in \cite{AlarconDrinovecForstnericLopez2015PLMS}.
In the same paper, and before that in \cite{DrinovecForstneric2015TAMS}, this method was also used 
to study minimal hulls and null hulls of compact sets in $\R^n$ and $\C^n$, respectively. 
The latter results are in the spirit of Poletsky's theory of analytic disk functionals; 
see \cite{DrinovecForstneric2012IUMJ} and the references therein.
Finally, in \cite{AlarconDrinovecForstnericLopez2015PLMS,AlarconDrinovecForstnericLopez2015AJM}
the Riemann-Hilbert method was applied to the construction of complete 
proper conformal minimal immersions from an arbitrary bordered Riemann surface 
to a certain class of domains in $\R^n$. 

In Section \ref{sec:RH} we extend the Riemann-Hilbert method to $\Igot$-invariant 
conformal minimal immersions $\Ncal\to\R^n$; see Theorem \ref{th:RH}. The proof 
follows the orientable case (see \cite[Theorem 3.6]{AlarconDrinovecForstnericLopez2015PLMS}), using
tools from Chapter \ref{ch:sprays} on approximation and gluing $\Igot$-invariant sprays.
The key to the proof of Theorem \ref{th:Jordan}  (which includes Theorem \ref{th:intro-Jordan} as a corollary)
is to use the Riemann-Hilbert method on $\Igot$-invariant pairs of short boundary arcs $I= I'\cup \Igot(I') \subset b\Ncal$
in order to increase the  boundary distance from a fixed interior point $p_0 \in\Ncal$ to $I$ by a 
prescribed amount by a deformation which is arbitrarily $\Cscr^0$ small on $\Ncal$. 
A recursive application of this  procedure yields a sequence of $\Igot$-invariant conformal
minimal immersions $Y_j\colon \Ncal \to\R^n$ $(j=1,2,\ldots)$ 
converging uniformly on $\Ncal$ to a continuous map $Y \colon \Ncal\to\R^n$ 
which is a complete conformal minimal immersion in the interior $\mathring \Ncal$
and is a topological embedding (modulo $\Igot$-invariance) of the boundary $b\Ncal$.
By passing down to $\UnNcal=\Ncal/\Igot$ we obtain a map $\UnY\colon \UnNcal\to\R^n$ 
satisfying Theorem \ref{th:intro-Jordan}.

%
%
\section{Proper non-orientable minimal surfaces in domains in $\R^n$} \label{sec:intro-proper}

We begin this section with the following theorem summarizing several results from Sections \ref{sec:Iproper} and \ref{sec:Icomplete}
on the existence of 
complete conformal minimal immersions of arbitrary non-orientable surfaces to $\R^n$.
The proofs rely upon the Runge-Mergelyan theorems
in this setting; see Theorems \ref{th:intro-Runge}, \ref{th:Mergelyan} and \ref{th:Mergelyan1}.

%
%

\begin{theorem}[Complete non-orientable conformal minimal surfaces in $\R^n$]
\label{th:applications1}
Let $\UnNcal$ be an open non-orientable surface endowed with a conformal structure.
The following assertions hold for any $n\ge 3$:
\begin{enumerate}[\rm (i)]
\item Every conformal minimal immersion $\UnNcal\to\r^n$ can be approximated 
uniformly on compacts by conformal minimal immersions $\UnNcal \to \r^n=\r^2\times\r^{n-2}$ properly 
projecting into $\r^2$ (hence proper in $\r^n$).  If $n\ge 5$, there exist approximating 
conformal minimal embeddings with this property.
\vspace{1mm}
\item Let $\underline K$ be a compact set in $\UnNcal$.  
Every conformal minimal immersion $\UnX=(\UnX_1,\ldots,\UnX_n)\colon\underline K\to\r^n$ such that $\UnX_j$ 
extends harmonically to $\UnNcal$ for $j=3,\ldots,n$ can be approximated uniformly on $\underline K$ by complete conformal 
minimal immersions $\wt \UnX=(\wt \UnX_1,\ldots,\wt \UnX_n)\colon\UnNcal\to\r^n$ with $\wt \UnX_j=\UnX_j$ for $j=3,\ldots,n$.
\vspace{1mm} 
\item $\UnNcal$ carries a complete conformal minimal immersion $\UnNcal\to\r^n$ whose generalized Gauss map 
$\UnNcal\to \cp^{n-1}$  is nondegenerate and 
omits $n$ hyperplanes of $\cp^{n-1}$ in general position.
\end{enumerate}
\end{theorem}

The case $n=3$ of Theorem \ref{th:applications1} is proved in \cite{AlarconLopez2015GT}. 
More general versions  are given by Theorem \ref{th:Iproper} in Section \ref{sec:Iproper} (for part (i)) and Corollaries \ref{co:fixed} 
and \ref{co:gaussmap} in Section \ref{sec:Icomplete} (for parts (ii) and (iii), respectively).  
Analogous results are already known in the orientable framework (see \cite{AlarconLopez2012JDG,AlarconFernandezLopez2012CMH,AlarconFernandezLopez2013CVPDE,AlarconForstnericLopez2016MZ}). 

We wish to point out that part {\rm (i)} of the theorem is related to an old  conjecture of
Sullivan about the conformal type of properly immersed minimal surfaces in $\r^3$ with finite topology,
 and also to a question of Schoen and Yau from 1985 about the conformal type of minimal surfaces in $\r^3$ 
properly projecting into a plane (see \cite[p.\ 18]{SchoenYau1997IP}). The number of exceptional 
hyperplanes in part {\rm (iii)} is sharp according to a result of Ahlfors  \cite{Ahlfors1941ASSF}. 
We refer to Chapter \ref{ch:applications} for further references and motivation. 

We also find the first known example of a properly embedded non-orientable minimal surface
in $\R^4$ (a minimal M\"obius strip); see Example \ref{ex:Mobius}. No such surfaces exist in $\R^3$
due to topological reasons. A properly {\em immersed} minimal M\"obius strip in $\R^3$ with finite total
curvature was found by Meeks \cite{Meeks1981DMJ} back in 1981.

The last topic that we shall touch upon concerns the existence of proper (and complete proper) 
conformal minimal immersions of non-orientable bordered surfaces to certain classes of domains in $\R^n$.  

We recall the notion of a (strongly) $p$-plurisubharmonic function on a domain in $\R^n$
and of a $p$-convex domain in $\R^n$. This class of functions and domains have been studied by 
Harvey and Lawson in \cite{HarveyLawson2013IUMJ}; see also 
\cite{DrinovecForstneric2015TAMS,HarveyLawson20092AJM,HarveyLawson2011ALM,HarveyLawson2012AM}.
For more details and further references, see Section \ref{sec:proper}.

Let $n\ge 3$ and $p\in \{1,2,\ldots,n\}$. A smooth real function $\rho$ on a domain $D\subset \R^n$
is said to be {\em (strongly) $p$-plurisubharmonic} if for every point $x\in D$, the sum of the smallest 
$p$ eigenvalues of its Hessian $\Hess_\rho(x)$ at $x$ is nonnegative (resp.\ positive). 
This holds if and only if the restriction of $\rho$ to any $p$-dimensional minimal submanifold 
is a subharmonic function on that submanifold. Note that $1$-plurisubharmonic functions are convex functions, 
while $n$-plurisubharmonic  functions are the usual subharmonic functions.

A domain  $D\subset\r^n$ is said to be {\em $p$-convex} if it admits a smooth strongly $p$-plurisubharmonic 
exhaustion function. A $1$-convex domain is simply a convex domain, while every domain in $\R^n$ is $n$-convex,
i.e., it admits a strongly subharmonic exhaustion function (see Greene and Wu \cite{GreeneWu1975} 
and Demailly \cite{Demailly1990}). A domain $D$ with $\Cscr^2$ boundary is $p$-convex if and only if 
\[
	\kappa_1(x)+\cdots + \kappa_p(x) \ge 0\ \ \text{for each point $x\in bD$}, 
\]
where $\kappa_1(x)\le \kappa_2(x)\le \cdots\le \kappa_{n-1}(x)$ denote the normal curvatures of the boundary
$bD$ at the point $x$ with respect to the inner normal. The domain is said to be {\em strongly $p$-convex}
for some $p\in \{1,\ldots,n-1\}$ if it is smoothly bounded (of class $\Cscr^2$) and we have the strict inequality
$\kappa_1(x)+\cdots + \kappa_p(x) > 0$ for each point $x\in bD$.
Smooth $(n-1)$-convex domains in $\r^n$ are also called {\em mean-convex}, while
$2$-convex domains are called {\em minimally convex} 
(cf.\ \cite{DrinovecForstneric2015TAMS,AlarconDrinovecForstnericLopez2015AJM}).

In Section \ref{sec:proper} we prove the following result which holds both for 
orientable and non-orientable surfaces; see Theorem  \ref{th:convex} for a more precise statement.

%
%

\begin{theorem}[Proper non-orientable minimal surfaces in $p$-convex domains] \label{th:intro3}
Let $\UnNcal$ be a compact bordered non-orientable surface.
If $D\subset\r^n$ is a $p$-convex domain, where $p=2$ if $n=3$ and $p=n-2$ if $n>3$, 
then every conformal minimal immersion $\UnX\colon\UnNcal\to\r^n$ satisfying $\UnX(\UnNcal)\subset D$ 
can be approximated uniformly on compacts in $\mathring\UnNcal$ by complete proper conformal minimal 
immersions $\mathring\UnNcal\to D$. If the domain $D$ is relatively compact, smoothly bounded and 
strongly $p$-convex, then the approximating immersions can be chosen to extend continuously to $\UnNcal$. 
\end{theorem}

The orientable case of Theorem \ref{th:intro3} in dimension $n=3$, with $p=2$, is given by
\cite[Theorem 1.1]{AlarconDrinovecForstnericLopez2015AJM} (see also
\cite[Theorem 1.9]{AlarconDrinovecForstnericLopez2015AJM} for boundary continuity of maps
into strongly $2$-convex domains), while for $n\ge 4$ and $p=n-2$ it is stated as 
\cite[Remark 3.8]{AlarconDrinovecForstnericLopez2015AJM}. The result is new
for non-orientable surfaces.

It may be of interest to recall that any domain in $\R^3$ whose boundary
is a properly embedded minimal surface is a minimally convex domain 
(see \cite[Corollary 1.3]{AlarconDrinovecForstnericLopez2015AJM});
hence, Theorem \ref{th:intro3} applies to such domains. It is also worthwhile pointing  out that, already in the 
orientable case, the class of $2$-convex domains in $\R^3$ is the biggest general class of 
domains in $\R^3$ for which this result holds; see 
\cite[Remark 1.11 and Examples 1.13, 1.14]{AlarconDrinovecForstnericLopez2015AJM}.

Theorem \ref{th:intro3}, and its counterpart in the orientable case, 
contributes to the problem of understanding which domains in Euclidean spaces admit 
complete properly immersed minimal surfaces, and how the geometry of the domain influences the conformal 
properties of such surfaces. We again refer to Chapter \ref{ch:applications} for the motivation
and further references.

In conclusion, we briefly describe the method of proof of Theorem \ref{th:intro3};
the details are provided in the proof of Theorem  \ref{th:convex}.  A  major ingredient is the 
Riemann-Hilbert boundary value problem described in the previous section. 

Let $\rho\colon D\to \R$ be a smooth strongly $p$-plurisubharmonic exhaustion function.
Also, let $(\Ncal,\Igot)$ be an oriented double-sheeted covering of $\UnNcal$ and 
$X\colon \Ncal\to D$ be the lifting of $\UnX$. We inductively construct a sequence
of $\Igot$-invariant conformal minimal immersions $X_j\colon \Ncal\to D$ $(j=1,2,\ldots)$
which approximate $X$ on a given compact subset of $\Ncal$ and 
converge uniformly on compacts in $\mathring\Ncal$ to a proper complete 
$\Igot$-invariant conformal minimal immersion $\wt X=\lim_{j\to\infty} X_j \colon \mathring\Ncal \to D$.
If $D$ is smoothly bounded and strongly $p$-convex, then the process is designed
so that the convergence takes place in $\Cscr^0(\Ncal,\R^n)$. 
The induced map $\wt \UnX\colon\mathring \UnNcal \to D$ then satisfies the conclusion
of Theorem \ref{th:intro3}.

The Riemann-Hilbert method is used in this recursive process in two different and 
mutually independent ways: 
\begin{itemize}
\item[\rm (a)] to push the boundary $X(b\Ncal)$ of the image surface closer to $bD$, and
\vspace{1mm}
\item[\rm (b)] to increase the boundary distance in $\Ncal$ induced by the immersion.
\end{itemize}

Both these applications of the Riemann-Hilbert method have been mentioned in the previous section. 
The procedure (b) can be carried out by deformations  that are $\Cscr^0$-small on $\Ncal$, 
and hence it does not infringe upon the procedure (a); in the inductive process we simply alternate them
in order to achieve both completeness and properness of the limit map.
This has already been exploited in the orientable case in the papers \cite{AlarconDrinovecForstnericLopez2015PLMS,AlarconDrinovecForstnericLopez2015AJM}.

Part (b) has been explained in connection with the proof of Theorem \ref{th:intro-Jordan}.

The inductive step in procedure (a) uses the Riemann-Hilbert method to deform the surface 
along a short boundary arc in the direction of a carefully selected family of conformal minimal disks 
in $D$ along which the function $\rho$ increases quadratically. The construction of such 
disks depends on the hypothesis that $\rho$ is strongly $p$-plurisubharmonic;
see \cite[Lemma 3.1]{AlarconDrinovecForstnericLopez2015AJM} for the case  $p=2$, $n=3$.
We actually find a corresponding family of null holomorphic disks in the tube
domain $D\times \imath\R^n\subset\C^n$ over $D$ by following the standard complex analytic 
method of choosing so called {\em Levi disks}. (Here, $\imath =\sqrt{-1}$.)
These are complex disks which are tangential 
to the level set of $\rho$ at the center point and such that 
the Levi form of $\rho$, considered as a function on $D\times \imath \R^n$ which is 
independent of the imaginary component, is positive along these disk. These conditions 
ensure that $\rho$ increases quadratically along the disk. A  minor complication arises
at any critical point of $\rho$, but it is well understood how to handle it by reducing to the noncritical case.

For more details, see Section \ref{sec:proper} and the paper \cite{AlarconDrinovecForstnericLopez2015AJM}.


\chapter{Preliminaries} 
\label{ch:prelim}

We shall denote the complex coordinates on $\C^n$ by $z=x+\imath y=(z_1,\ldots, z_n)$, with $z_j=x_j+\imath y_j$.   
We denote by $\D=\{z \in\C\colon |z|<1\}$ the unit disk in $\C$ and by 
\[
	\B^n=\{z\in \C^n\colon ||z||^2=\sum_{j=1}^n |z_j|^2 <1\} \subset \C^n
\]
the unit ball. Here, $\imath =\sqrt{-1}$.

Given a smooth surface $N$, we shall be considering smooth maps $X\colon N\to\R^n$
and $Z=X+\imath Y\colon N\to \C^n$ whose images are minimal surfaces in $\R^n$ or complex 
null curves in $\C^n$. Points of $N$ will be denote by $p,p_0$ and the like.  { By a compact domain $A\subset \Ncal$, 
we mean a connected compact subset which is a topological surface with boundary.}
Often $N$ will be endowed with a conformal structure or, in the orientable case, 
with a Riemann surface structure, and this will be emphasized by using instead 
the corresponding caligraphic letter $\Ncal$.


\section{Conformal structures on surfaces}
\label{sec:conformal}
Let $\UnN$ be a smooth connected non-orientable surface. There is a smooth connected orientable surface $N$ 
and a two-sheeted covering map $\iota \colon N\to \UnN$ whose group of deck transformations is generated
by a fixed-point-free orientation reversing involution $\Igot \colon N\to N$. 
The triple $(N,\Igot,\iota)$ is unique up to an isomorphism of covering spaces. 
Conversely, given a pair $(N,\Igot)$ as above, the quotient $\underline{N}=N/\Igot$ 
is a non-orientable surface whose orientable 2-sheeted covering
space is $N$, and the deck group is generated by the involution $\Igot$.

A {\em conformal structure} on  a surface $N$ is determined by a choice  of a Riemannian metric; 
two metrics $g,g'$ on $N$ determine the same conformal structure if and only if $g'=\lambda g$ for some 
function $\lambda\colon N\to (0,\infty)$.  We denote by $\Ncal=(N,[g])$ the surface 
with a choice of a conformal structure. If $\UnN$ is a non-orientable surface and  $\iota \colon N\to \UnN$
is as above, the pullback $\iota^*g$ of a Riemannian metric $g$ on $\UnN$ is a Riemannian metric on $N$
which, together with a choice of orientation, determines the structure of a Riemann surface 
$\Ncal$. In this complex structure, the involution $\Igot\colon\Ncal\to\Ncal$ 
is antiholomorphic and the quotient projection $\iota \colon \Ncal \to  \UnNcal$ is conformal.

Equivalently, a conformal structure on a surface $M$ is determined by a choice of an atlas
$\Uscr=\{(U_i,\phi_i)\}$, where $\{U_i\}$ is an open cover of $M$ and
$\phi_i\colon U_i\to U'_i\subset\C$ are homeomorphisms whose transition
functions $\phi_{i,j}=\phi_i\circ\phi_j^{-1}$ are holomorphic or antiholomorphic on each connected 
component of their domain.  If $M$ is orientable then, 
replacing $\phi_i$ by $\bar\phi_i$ if necessary, we obtain a complex  atlas.
If $M$ is non-orientable and $\iota\colon N\to M$ is a 2-sheeted oriented covering,
then a conformal atlas $\Uscr$ on $M$ pulls back to a conformal atlas $\Vscr$ on $N$.
Choosing an orientation on $N$ and orienting each chart in $\Vscr$ accordingly
gives a complex atlas, and $\Ncal=(N,\Vscr)$ is then a Riemann surface.
The quotient projection $\iota \colon \Ncal \to \Mcal$ is either holomorphic or antiholomorphic 
in every pair of charts from the respective atlases $\Vscr$, $\Uscr$.

In the sequel, objects related to a non-orientable surface $\UnNcal$ will be underlined, 
while the corresponding objects related to the oriented 2-sheeted covering $\Ncal$ will be 
denote by the same letter but without underlining.

%
%

\section{$\Igot$-invariant functions and $1$-forms. Spaces of functions and maps}
\label{sec:invariant} 

Let $\Ncal$ be an open Riemann surface. A differential $(1,0)$-form $\phi$ on $\Ncal$ is 
given in any local holomorphic coordinate $z=x+\imath y$ by $\phi(z) = a(z)dz$ for some 
function $a$; $\phi$ is holomorphic if and only if every coefficient function $a(z)$ obtained in this way 
is holomorphic. A $(0,1)$-form is locally of the form $\phi(z)=b(z) d\bar z$. 
The exterior differential splits as $d=\di+\dibar$  where 
\[
	\di f= \frac{\di f}{\di z} dz = \frac{1}{2}\left(\frac{\di f}{\di x}- \imath \frac{\di f}{\di y}\right) dz,
	\quad
	\dibar f= \frac{\di f}{\di \bar z} d\bar z = \frac{1}{2}\left(\frac{\di f}{\di x} +\imath \frac{\di f}{\di y}\right) d\bar z.
\]

We denote by $\Oscr(\Ncal)$ and $\Omega(\Ncal)$ the spaces of all holomorphic functions 
and 1-forms on $\Ncal$, respectively, endowed with the compact-open topology. If $K$ is a compact set in $\Ncal$, 
then $\Oscr(K)$ stands for the algebra of functions $f\colon U_f\to \C$ that are holomorphic on open 
neighborhoods $U_f\supset K$ of $K$ (depending on the function), in the sense 
of germs along $K$. If $K$ has boundary of class $\Cscr^1$ and $k\ge 0$ is an integer
then $\Ascr^k(K)$ denotes the space of $\Cscr^k$ functions $K\to \C$
that are holomorphic in the interior $\mathring K=K\setminus bK$. 
We shall write $\Ascr^0(K)=\Ascr(K)$. We denote by 
\begin{equation}\label{eq:supnorm}
	||f||_{0,K} = \sup\{ |f(p)|: p\in K\}
\end{equation}
the supremum norm of a function $f$ in any of the spaces $\Oscr(K)$ or $\Ascr^k(K)$.

Given a complex manifold $Z$, we denote the corresponding space of maps to $Z$ by $\Oscr(\Ncal,Z)$, 
$\Oscr(K,Z)$, and $\Ascr^k(K,Z)$. For $f=(f_1,\ldots, f_n)\in \Ascr(K,\C^n)$ we set
\[
	||f||^2_{0,K} = \sum_{i=1}^n ||f_i||^2_{0,K}.
\]
If $f\colon K\times V\to Z$ is a map  on the product $K\times V$ of a compact set $K\subset \Ncal$
and an open set $V \subset \C^m$ for some $m$, we shall say that $f$ is of class $\Ascr(K)$ 
(or $\Ascr(K\times V,Z)$ when we wish to emphasize the sets $V$ and $Z$)
if $f$ is continuous on $K\times V$ and holomorphic on $\mathring K\times V$.

%
%

\begin{definition}[$\Igot$-invariant maps and $1$-forms] \label{def:invariant}
Let $(\Ncal,\Igot)$ be a Riemann surface with a fixed-point-free antiholomorphic involution
$\Igot\colon \Ncal\to\Ncal$. 
\begin{itemize}
\item[\rm (a)] A map $X \colon \Ncal\to \R^n$ is {\em $\Igot$-invariant} if 
\[
	X \circ \Igot= X.
\]
\item[\rm (b)]  A holomorphic map $f\colon \Ncal\to \C^n$ is {\em $\Igot$-invariant} if 
\[ 
	f\circ \Igot =\bar f,
\] 
and is {\em $\Igot$-skew-invariant} if
$ 
	f\circ \Igot =- \bar f.
$ 
\vspace{1mm}
\item[\rm (c)]
A holomorphic (or meromorphic) $1$-form $\phi$ on $\Ncal$  is  {\em $\Igot$-invariant} if
\[ 
	\Igot^* \phi = \bar \phi.
\] 
\end{itemize}
\end{definition}

Note that a function $f=u+\imath v\in\Oscr(\Ncal)$  is $\Igot$-invariant if and only if $u,v\colon \Ncal\to\R$ are 
conjugate harmonic functions satisfying
\[
	u\circ \Igot = u,\quad v\circ \Igot=-v.
\]
Thus, the real part $u=\Re f$ is $\Igot$-invariant while the imaginary part  $v=\Im f$ is $\Igot$-skew-invariant. 
We introduce the following notation:
\[
	\OI(\Ncal)= \{f\in \Oscr(\Ncal) :   f\circ\Igot=  \bar f\}.
\]
Note that $\imath \OI(\Ncal)=\{\imath f \colon f\in\OI(\Ncal)\} =\{f\in \Oscr(\Ncal):   f\circ\Igot= - \bar f\}$ and that  $ \OI(\Ncal)$ 
and $\imath \OI(\Ncal)$ are real Fr\'echet algebras such that
\begin{equation}\label{eq:direct sum}
	\OI(\Ncal)\cap \imath \OI(\Ncal)=\{0\} \quad 
	\text{and} \quad
	\Oscr(\Ncal) = \OI(\Ncal) \oplus \imath \OI(\Ncal),
\end{equation}
with projections $\Pi_\Igot \colon \Oscr(\Ncal) \to \OI(\Ncal)$,  $\Pi_{\bar \Igot} \colon \Oscr(\Ncal) \to \imath \OI(\Ncal)$ 
given by
\begin{equation}\label{eq:projections}
	\Pi_\Igot(f)           = \frac{1}{2} \left(f + \overline{f\circ \Igot}\, \right),\quad
	\Pi_{\bar \Igot}(f) = \frac{1}{2} \left(f - \overline{f\circ \Igot}\, \right).
\end{equation}
These projections satisfy $\ker \Pi_\Igot = \imath \OI(\Ncal)$ and $\ker \Pi_{\bar \Igot}=\OI(\Ncal)$.

By $\Omega_\Igot(\Ncal)$ we denote the space of  all $\Igot$-invariant holomorphic $1$-forms
on $\Ncal$ with the compact-open topology. Note that $\Omega_\Igot(\Ncal)$ is an $\OI(\Ncal)$-module.
The following lemma summarizes connections between $\Igot$-invariant functions and $1$-forms.

%
%
\begin{lemma}\label{lem:invariant}
Let $(\Ncal,\Igot)$ be a connected open Riemann surface with a fixed-point-free antiholomorphic involution. 
If $u\colon\Ncal\to \R$ is a $\Igot$-invariant harmonic function (i.e., $u\circ \Igot =u$), 
then $\di u\in \Omega_\Igot(\Ncal)$ is a $\Igot$-invariant holomorphic $1$-form:
$\Igot^*(\di u)=\overline {\di u}$. Conversely, assume that $\phi\in \Omega_\Igot(\Ncal)$.
Pick a point $p_0\in \Ncal$ and a number $a \in \R$. Then the following hold: 
\begin{itemize}
\item [\rm (a)]
If the periods of $\Re \phi$ vanish, then the function 
$u(p) = a+ \int_{p_0}^p \Re \phi$ $(p\in\Ncal)$ is harmonic, $\Igot$-invariant, 
and satisfies $2\di u = \phi$. 
\vspace{1mm}
\item[\rm (b)]
If the periods of $\phi$ vanish, then the function  
\[
	f(p)=a - \frac{\imath}{2} \int_{p_0}^{ \Igot(p_0)} \Im \phi  +\int_{p_0}^p \phi,
	\quad   p\in \Ncal,
\]
is holomorphic and $\Igot$-invariant $(f\in \OI(\Ncal))$ and $df=\di f=\phi$.	
\end{itemize}
\end{lemma}

\begin{proof}
Assume that $u$ is a $\Igot$-invariant real harmonic function. Then
\[
	\Igot^* \di u + \Igot^* \dibar u = \Igot^* du = d(u\circ\Igot) = du = \di u + \dibar u =
	\di u +  \overline{\di u}.
\]
Since $\Igot$ is antiholomorphic, the 1-form $\Igot^* \di u$ is of type $(0,1)$ while
$\Igot^* \dibar u$ is of type $(1,0)$. It follows that $\Igot^* \di u = \overline{\di u}$ as claimed. 

Assume that $\phi\in \Omega_\Igot(\Ncal)$.
Fix a point $p\in \Ncal$ and choose a path $\gamma\colon [0,1]\to\Ncal$ connecting $p=\gamma(0)$ 
to $\Igot(p)=\gamma(1)$. The path $\Igot_*\gamma$ then connects $\Igot(p)$ back to $p$,
and $\lambda=\gamma\cup \Igot_*\gamma$  is a loop based at $p$. We have that
\[
	   \int_\lambda \phi = \int_\gamma \phi + \int_{\Igot_* \gamma} \phi
	   = \int_\gamma \phi + \int_\gamma \Igot^*  \phi
	   = \int_\gamma \phi + \int_\gamma \bar  \phi 
	   = 2 \int_\gamma \Re \phi.
\]
If the real periods of $\phi$ vanish then $2\int_\gamma \Re\phi= \int_\lambda \phi=0$ and
hence  $u=\int \Re \phi$ satisfies $u(\Igot(p))= u(p)+\int_\gamma \Re\phi = u(p)$, so 
it is $\Igot$-invariant. The integral $f=\int \phi = u+\imath v$ is well defined and holomorphic
on any simply connected  domain in $\Ncal$, and the Cauchy-Riemann equations yield 
$\phi=df=\di f = 2\di u$. This proves (a).

Assume now that $\phi$ has vanishing complex periods and let $f$ be as in part (b). 
Then, $\Re f=u$ is the $\Igot$-invariant harmonic function in part (a).
Furthermore, the function $g=\overline{f\circ \Igot}$ is holomorphic and 
satisfies $\Re g=\Re f \circ \Igot = \Re f$. A calculation shows that 
$\Im f(p_0)=\Im g(p_0)$;  hence,  $f=g = \overline{f\circ\Igot}$ by the identity principle.
\end{proof}

Let $u$ be a smooth real function on $\Ncal$. Its {\em conjugate differential} is defined by 
\begin{equation}\label{eq:conjugatedif}
	d^c u= \imath(\dibar u - \di u) = 2 \Im (\di u).  
\end{equation}
In any local holomorphic coordinate $z=x+\imath y$ on $\Ncal$ we have $d^c u = -\frac{\di u}{\di y} dx+\frac{\di u}{\di x}dy$. 
Note that   
\[
	2\di u = du + \imath d^c u\quad \text{and}\quad  
	dd^c u= 2\imath\,  \di\dibar u = \Delta u\, \cdotp  dx\wedge dy.  
\]
Hence, $u$ is harmonic if and only if $d^c u$ is a closed $1$-form,  and in this case 
we have $d^c u=dv$ for any local harmonic conjugate $v$ of $u$. 

The {\em flux} of a harmonic function $u$ is the homomorphism 
$\Flux_u \colon H_1(\Ncal;\Z)\to\R$ on the first homology group of $\Ncal$ given by
\begin{equation}
\label{eq:Flux}
	 \Flux_u (\gamma) =  \int_\gamma d^c u = 2\int_\gamma \Im(\di u), \quad [\gamma] \in H_1(\Ncal;\Z).
\end{equation}
A harmonic function $u$ admits a harmonic conjugate on $\Ncal$ if and only if 
$\Flux_u=0$. If $\phi$ is a holomorphic $\Igot$-invariant $1$-form on $\Ncal$ and $\gamma$ is a smooth 
arc or a closed curve, then
\begin{equation}\label{eq:invariance}
	\int_{\Igot_*\gamma} \phi = \int_\gamma \Igot^* \phi = \int_\gamma \bar\phi. 
\end{equation}
Applying this formula to $\phi=\di u$, where $u$ is a $\Igot$-invariant harmonic function on $\Ncal$, we get
for any closed curve $\gamma$ in $\Ncal$ that
\begin{equation}
\label{eq:FluxJ*}
	\Flux_u(\Igot_*\gamma) =  2 \int_{\Igot_*\gamma} \Im(\di u) = 
	- 2\int_{\gamma}  \Im(\di u) =  -\Flux_u(\gamma).
\end{equation}

We shall need the following result of Alarc\'on and L\'opez; the standard version 
without $\Igot$-invariance is due to Gunning and Narasimhan \cite{GunningNarasimhan1967MA}.

\begin{proposition}[Corollary 6.5 in \cite{AlarconLopez2015GT}]
\label{prop:Cor6.5}
Every open Riemann surface $\Ncal$ carrying an anti\-holomorphic involution $\Igot\colon \Ncal\to \Ncal$ 
without fixed points admits a $\Igot$-invariant holomorphic function $f\in \OI(\Ncal)$
without critical points. In particular, the differential $df=\di f$ of such a function is a nowhere vanishing
$\Igot$-invariant holomorphic $1$-form on $\Ncal$.
\end{proposition}


\section{Homology basis and period map}
\label{sec:homology}

Given a non-orientable open surface  $\UnN$ with finite topology, we shall give an explicit 
geometric description of its two-sheeted orientable covering $(N,\Igot)$,
and we will find a homology basis $\Bcal$ for $H_1(\UnN;\Z)$ whose support $|\Bcal|$ satisfies 
$\Igot(|\Bcal|)=|\Bcal|$ and  such that $N\setminus |\Bcal|$ 
has no relatively compact connected components in $N$. 
For background on topology of surfaces, see e.g.\ \cite[Chapter 9]{Hirsch1994-book}.

Every  closed  non-orientable surface $\UnM$ is the connected sum $\UnM= \P^2\ssharp {\cdots}\ssharp \P^2$ of $g\ge 1$ 
copies of the real projective plane $\P^2$;  the number $g$ is the genus of $\UnM$.
(This is the maximal number of pairwise disjoint closed curves in $\UnM$ which 
reverse the orientation.) Furthermore, $\K=\P^2 \ssharp \P^2$ is the Klein bottle, and for any compact non-orientable surface 
$\UnM$ we have that $\UnM\ssharp \K=\UnM\ssharp \T$ where $\T$ is the torus.  
This gives the following dichotomy according to whether the genus $g$ is even or odd:
\begin{enumerate}[(I)]
\item $g=1+2k \ge 1$ is odd. In this case, $\UnM =\P^2\ssharp \overbrace{\t \ssharp \cdots \ssharp \t}^{k}$
and $k=0$ corresponds to the projective plane $\P^2$. 
\item $g=2+2k\ge 2$ is even. In this case,
$\UnM= \K \ssharp \overbrace{\t \ssharp \cdots \ssharp \t }^{k}$ and
$k=0$ corresponds to the Klein bottle $\K=\P^2\ssharp \P^2$.
\end{enumerate}

Let $\iota \colon M \to \UnM$ be the $2$-sheeted covering of $\UnM$ by a compact orientable surface $M$
and let $\Igot\colon M\to M$ be the deck transformation of $\iota$. 
Then $M$ is an orientable surface of genus $g-1$; the sphere $\s^2$ if $g=1$ and a 
connected sum of $g-1$ copies of the torus $\T$ if $g>1$.
We construct an explicit geometric model for $(M,\Igot)$ in $\r^3$.
Recall that $\s^2$ is the unit sphere in $\R^3$ centered at the origin. Let $\tau \colon \R^3\to\R^3$
denote the involution $\tau(x)=-x$.

%
%
{\bf Case (I):} $\UnM=\P^2\ssharp \overbrace{\t \ssharp \cdots \ssharp \t}^{k}$.
We take $M$  to be an embedded surface %
\[
	\big(\t_1^- \ssharp \cdots\ssharp \t_k^-\big)\ssharp \s^2\ssharp 
	\big(\t_1^+ \ssharp \cdots\ssharp \t_k^+\big)
\]	
of genus $g-1=2k$ in $\r^3$ which is invariant by the symmetry with respect to the origin  (i.e., $\tau(M)=M$), 
where $\t_j^-$, $\t_j^+$ are embedded tori in $\r^3$ with $ \tau(\t_j^-)=\t_j^+$ for all $j$.
(See Figure \ref{fig:I-basis1}.)
\begin{figure}[ht]
    \begin{center}
    \scalebox{0.22}{\includegraphics{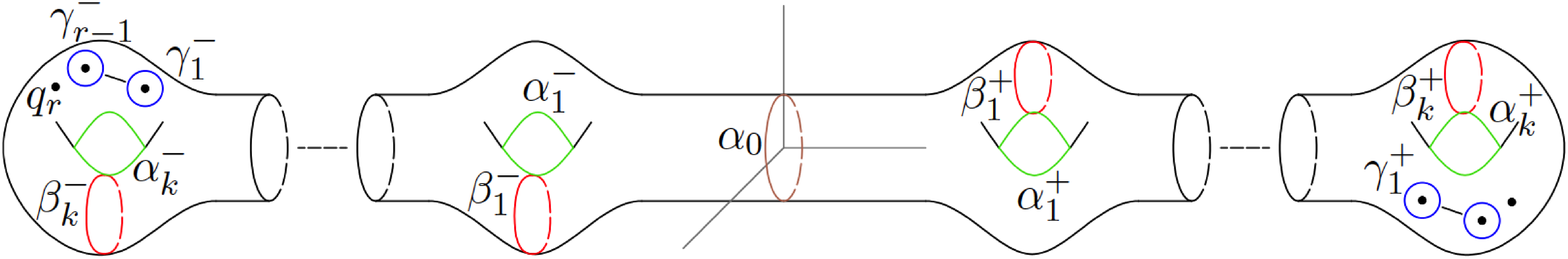}}
        \end{center}
\caption{$\Igot$-basis in case {\rm (I)}.}\label{fig:I-basis1}
\end{figure}
We also take $\Igot = \tau|_M\colon M\to M$.
If $k=0$, the model is the round sphere $\S^2$ with the orientation reversing antipodal map $\Igot$. 
We can split $M=M^-\cup C \cup M^+$, where $C\subset \s^2$ is a closed $\Igot$-invariant cylinder 
and $M^\pm$ are the closures of the two components of $M\setminus C$, 
both homeomorphic to the connected sum of $k$ tori minus an open disk. 
Obviously, $\Igot(M^-)=M^+$ and $M^-\cap M^+=\emptyset$.

%
%
{\bf Case (II):}  $\UnM= \K \ssharp \overbrace{\t \ssharp \cdots \ssharp \t  }^{k}$ where $\K$ is the
Klein bottle. The corresponding model for $(M,\Igot)$ is the following. 
Let $\T_0\subset \r^3$ be the standard torus of revolution centered at the origin and invariant under 
the antipodal map $\tau$.  Let $M$ be an embedded  $\tau$-invariant surface 
\[
	M= \big(\t_1^- \ssharp \cdots\ssharp \t_k^-\big)\ssharp \t_0\ssharp \big(\t_1^+ \ssharp \cdots\ssharp \t_k^+\big)
	\subset \R^3,
\]
where the tori $\T^\pm_j$ are as above, and set $\Igot= \tau|_{M}$. 
(See Figure \ref{fig:I-basis2}.)
\begin{figure}[ht]
    \begin{center}
    \scalebox{0.22}{\includegraphics{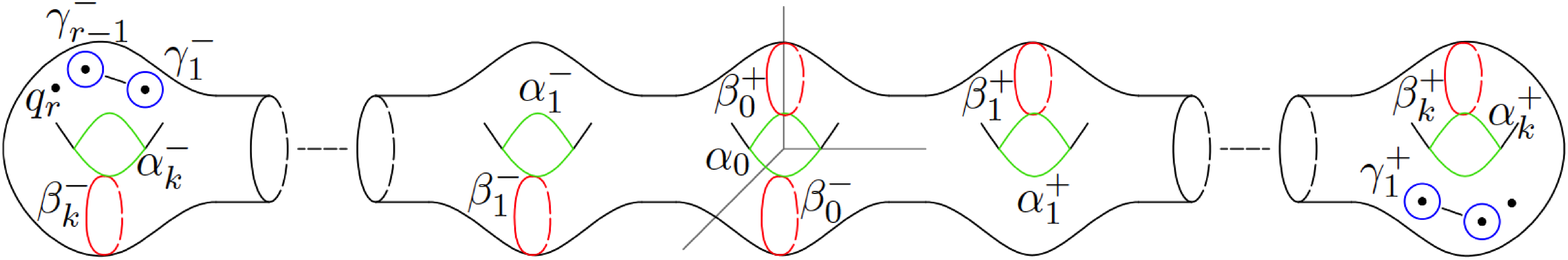}}
        \end{center}
\caption{$\Igot$-basis in case {\rm (II)}.}\label{fig:I-basis2}
\end{figure}
If $k=0$, the model is the torus $\T_0$ with the involution $\Igot=\tau|_{\T_0}$.
Write $M=M^-\cup K \cup M^+$, where $K\subset \t_0\subset \r^3$ is a $\Igot$-invariant torus minus two 
disjoint open disks, and $M^\pm$ are the closures of the two components of $M\setminus K$, 
both homeomorphic to the connected sum of $k$ tori minus an open disk. 
Obviously, $\Igot(M^-)=M^+$ and $M^-\cap M^+=\emptyset$.

\vspace{0.2cm}

If we remove $r\ge 1$ distinct points $p_1,\ldots,p_r$ from $\UnM$ and set 
\[
	\UnN=\UnM\setminus \{p_1,\ldots,p_r\},
\] 
then the geometric description of the oriented double covering $\iota \colon N\to \UnN$ is analogous to the previous case. 
Let $\{q_j,\Igot(q_j)\}=\iota^{-1}(p_j)$ $(j=1,\ldots,r)$. Set 
\[
	N=M \setminus\{q_1,\ldots,q_r,\Igot(q_1),\ldots,\Igot(q_r)\}
\]
and restrict the involution $\Igot$ to $N$.  Without loss of generality we can suppose that $q_1,\ldots,q_r\in M^-$ 
and $\Igot(q_1),\ldots,\Igot(q_r)\in M^+$. Set 
\begin{eqnarray*}
	N^- & = & N\cap M^- =
	M^-\setminus\{q_1,\ldots,q_r \},
\\  
	N^+ & = & N\cap M^+ =
	M^-\setminus\{\Igot(q_1),\ldots,\Igot(q_r)\}.
\end{eqnarray*}
A similar description holds if we remove from $\UnM$ pairwise disjoint disks (instead of points)
to get a compact non-orientable surface with smooth boundary.

\vspace{0.2cm}

{\bf Description of the homology basis.}
Consider for each torus $\t_j^-$ in either of the above descriptions its natural homology basis 
$\{\alpha_j^-,\beta_j^-\}$ of $H_1(\t_j^-;\z)$ (two closed embedded curves meeting at one point).   
Then, $\{\alpha_j^+,\beta_j^+\}=\{\Igot(\alpha_j^-),\Igot(\beta_j^-)\}$ is a basis of $H_1(\t_j^+;\z)$
for every $j=1,\ldots,k$.  Without loss of generality we can suppose that the following conditions hold:
\[
	\bigcup_{j=1}^k (\alpha_j^-\cup \beta_j^-)\subset N^-, \ \  \bigcup_{j=1}^k (\alpha_j^+\cup \beta_j^+)\subset N^+,
	\ \  (\alpha_j^-\cup \beta_j^-)\cap (\alpha_i^-\cup \beta_i^-)=\emptyset \ \ \text{for}\ i\neq j.
\]
Choose pairwise disjoint small circles $\gamma_j^-$ around the point $q_j$ $(j=1,\ldots,r-1)$  
lying in $N^-\setminus \cup_{j=1}^k (\alpha_j^-\cup \beta_j^-)$ and let $\gamma_j^+=\Igot(\gamma^-_j)\subset N^+$.   

\vspace{1mm}

{\bf Case (I)}: Let $\alpha_0$ denote the circle generating the homology of the cylinder $C$ 
such that $\Igot(\alpha_0)=\alpha_0$,  and notice that homologically $(\Igot|_{C})_*(\alpha_0)=\alpha_0$ 
in $H_1(C;\z)$ as well. (See Figure \ref{fig:I-basis1}.) The desired homology basis of $H_1(N;\Z)$ is
\begin{equation}\label{eq:basis1}
	\mathcal B=\{\alpha_0\} \cup \bigcup_{j=1}^k\{\alpha_j^\pm,\beta_j^\pm \} \cup 
	\bigcup_{j=1}^{r-1} \{\gamma_j^\pm\}.
\end{equation}

\vspace{1mm}

{\bf Case (II)}: Let $\{\alpha_0,\beta_0\}$ be embedded curves in  $K\subset \T_0$ generating $H_1(\t_0;\z)$ 
and satisfying 
\[
	\Igot(\alpha_0)=\alpha_0, \quad \Igot(\beta_0)\cap \beta_0 =\emptyset,
\]
and homologically in $H_1(\t_0;\z)$: 
\[
	(\Igot|_{\t_0})_*(\alpha_0)= \alpha_0, \quad ( \Igot|_{\t_0})_*(\beta_0)=-\beta_0.
\]
(See Figure \ref{fig:I-basis2}.)  For consistency of notation we write $\beta^+_0=\beta_0$ and $\beta^-_0=\Igot(\beta_0)$.
The desired homology basis of $H_1(N;\z)$ is now
\begin{equation}\label{eq:basis2}
	\mathcal B= \bigl\{\alpha_0, \beta^+_0,  \beta^-_0 \bigr\}  
	\cup  \bigcup_{j=1}^k\{\alpha_j^\pm,\beta_j^\pm \} 
	\cup  \bigcup_{j=1}^{r-1} \{\gamma_j^\pm\}.
\end{equation}
Note  that the curve $\beta^-_0$ is homologically independent from the other curves in $\Bcal$ due to the presence of 
the holes $\{q_r,\Igot(q_r)\}$, but small curves around these two holes  are homologically dependent on the curves in $\Bcal$.

A basis $\Bcal$ as in \eqref{eq:basis1} and \eqref{eq:basis2} will be called an {\em $\Igot$-basis of $H_1(N;\z)$}.

\begin{remark} \label{rem:basis}
If $S$ is an $\Igot$-invariant compact subset of $N$ which is a strong deformation retract of $N$, 
we may find an $\Igot$-basis $\Bcal$ of $H_1(N;\z)$  with support $|\Bcal|\subset S$. Obviously, $\Bcal$ is a basis  of 
$H_1(S;\z)$.  We will say that $\Bcal$ is an $\Igot$-basis of $S$.
\end{remark}

\vspace{1mm}

%
%

{\bf A period map.} 
Assume now that the surface $\UnN$ carries a conformal structure, 
so we denote it by $\UnNcal$. Its $2$-sheeted orientable covering, $\Ncal$, is then a Riemann surface
and the deck map $\Igot\colon \Ncal \to\UnNcal$  is an antiholomorphic involution. 
Since $\Ncal$ has finite topological type, a theorem of Stout \cite[Theorem 8.1]{Stout1965} shows that it 
is obtained by removing finitely many pairwise disjoint points and closed disks from a compact Riemann surface.  
However, from the topological point of view we need not distinguish between removing disks and points,
and the condition that the homology basis $\Bcal$ \eqref{eq:basis1}, \eqref{eq:basis2} 
for $H_1(\Ncal;\Z)$ be Runge in $\Ncal$ is independent of the choice of the conformal structure.

Write $\Bcal=\Bcal^+\cup \Bcal^-$, where
$\Bcal^+=\{\delta_1,\ldots,\delta_l\}$ and $\Bcal^-=\{\Igot(\delta_2),\ldots,\Igot(\delta_l)\}$
denote the following collections of curves (the curves $\beta^\pm_0$ appear only in Case (II)):
\begin{eqnarray}
\label{eq:B0}
	\Bcal^+ &= &  \{\delta_1=\alpha_0\} \cup 
	\bigl\{\alpha^+_1,\ldots,\alpha^+_m;\,\beta^+_0, \beta^+_1,\ldots,\beta^+_m;\ 
	\gamma^+_1,\ldots,\gamma^+_{r-1}\bigr\}, 
	\\
	\label{eq:B-}
	\Bcal^- &=& \bigl\{\alpha^-_1,\ldots,\alpha^-_m;\  \beta^-_0,\beta^-_1,\ldots,\beta^-_m;\ 
	\gamma^-_1,\ldots,\gamma^-_{r-1}\bigr\}.
\end{eqnarray}
We have set $\delta_1=\alpha_0$; the labeling of the remaining curves is unimportant. 
Note that $\Bcal^+\cap\Bcal^-=\emptyset$.
Fix an $\Igot$-invariant holomorphic $1$-form $\theta$ without zeros on $\Ncal$  
(cf.\ Proposition \ref{prop:Cor6.5}).  With $\Bcal^+=\{\delta_1,\ldots,\delta_l\}$ as above, let 
\[
	\Pcal=(\Pcal_1,\ldots,\Pcal_l)\colon\Oscr(\Ncal)\to\C^l
\] 
denote the period map whose $j$-th component is given by
\begin{equation}\label{eq:PB_0}
	\Pcal_j(f)= \int_{\delta_j} f\theta, \qquad f\in\Oscr(\Ncal),\ \ j=1,\ldots,l.
\end{equation}
Similarly, we define $\Pcal(\phi)=\bigl(\int_{\delta_j}\phi\bigr)_{j=1,\ldots,l}$ for a holomorphic $1$-form $\phi$.
Note that $\Pcal$ is a {\em partial period map} which does not control all periods of 
holomorphic $1$-forms. However, on $\Igot$-invariant $1$-forms the symmetry conditions \eqref{eq:invariance}
imply that $\Pcal$ completely determines the periods as shown by the following lemma.

\begin{lemma}\label{lem:partial-period}
Let $\phi$ be a $\Igot$-invariant holomorphic $1$-form on $\Ncal$. Then:
\begin{itemize} 
\item[\rm (a)] $\phi$ is exact if  and only if $\Pcal(\phi)=0$. 
\vspace{1mm}
\item[\rm (b)] The real part $\Re\phi$  is exact if and only if $\Re \Pcal(\phi)=0$.
\end{itemize}
\end{lemma}

\begin{proof}
By \eqref{eq:invariance} we have that
\[
	\int_{\Igot_*\delta_j}\phi =   \int_{\delta_j}\overline \phi, \quad\  j=1,\ldots,l.
\]
Therefore, $\Pcal(\phi)=0$ implies that $\phi$ has vanishing periods over all curves in the homology basis 
$\Bcal=\Bcal^+ \cup\Bcal^-$ given by \eqref{eq:basis1} or \eqref{eq:basis2}, and hence is exact.  
The converse is obvious.
Likewise, $\Re\Pcal(\phi)=0$ implies that $\Re\phi$ has vanishing periods over all curves in $\Bcal$
and hence  is exact.  However, the imaginary periods $\Im\Pcal(\phi)$ (the flux of $\phi$)
can be arbitrary subject to the conditions \eqref{eq:FluxJ*}: 
$\int_{\Igot_*\delta_j}\Im \phi =  - \int_{\delta_j}\Im \phi$ for $j=1,\ldots,l$.
In particular, we have $\int_{\delta_1}\Im \phi=0$ since $\Igot_*\delta_1=\delta_1$.
\end{proof}


\section{Conformal minimal immersions of non-orientable surfaces}
\label{sec:CMI}
We recall the notion of {\em Weierstrass representation} of a conformal minimal immersion.
For the orientable case we refer to Osserman \cite{Osserman-book};
the non-orientable case has been treated, for instance, in \cite{Meeks1981DMJ,AlarconLopez2015GT}. 

Let $n\ge 3$ be an integer and $\Ncal$ be an open Riemann surface or a compact bordered 
Riemann surface. An immersion 
$X\colon \Ncal\to\R^n$ is said to be {\em conformal} if it preserves angles. In a local holomorphic coordinate
$z=x+\imath y$ on $\Ncal$ this is equivalent to asking that the partial derivatives $X_x=\di X/\di x$ and 
$X_y=\di X/\di y$ of $X$ are orthogonal to each other and have the same length:
\begin{equation}\label{eq:Xconf}
	X_x\cdotp X_y = 0,\quad |X_x|=|X_y|>0.
\end{equation}
(We always use the standard Euclidean metric on $\R^n$.) Let us write 
\begin{equation}\label{eq:diX}
	\di X=\Phi=(\phi_1,\ldots,\phi_n)\quad \text{with $\phi_j=\di X_j$ for $j=1,\ldots,n$.} 
\end{equation}
From $2\di X = (X_x - \imath X_y)dz$ we obtain
\[
	(\phi_1)^2 + (\phi_2)^2+\cdots + (\phi_n)^2 = \frac{1}{4} 
	\left( |X_x|^2-|X_y|^2 - 2\imath  X_x\cdotp X_y \right) (dz)^2.
\]
A comparison with \eqref{eq:Xconf} shows that $X$ is conformal if and only if
\begin{equation}\label{eq:null}
	(\phi_1)^2 + (\phi_2)^2+\cdots + (\phi_n)^2=0.
\end{equation}

An immersion $X\colon \Ncal\to\R^n$ is {\em minimal} if its mean curvature vector vanishes identically. 
By Meusnier, this holds if and only if the image of $X$ is locally area minimizing.
It is classical (see e.g.\ \cite[Lemma 4.2, p.\ 29]{Osserman-book}) that a smooth conformal immersion 
$X\colon \Ncal\to\R^n$ is minimal if and only if it is harmonic. Furthermore, from
\[
	\triangle X = 4\frac{\di^2}{\di z \di\bar z} X = 4 \frac{\di}{\di\bar z} \left(\di X/\di z\right)
\]
we see that a smooth conformal immersion $X$ is minimal if and only if the $(1,0)$-differential
$\di X=\Phi$ \eqref{eq:diX} is a {\em holomorphic} $1$-form satisfying the nullity condition \eqref{eq:null}.
Note that $dX=2\Re(\di X)=2\Re \Phi$ and hence the real periods of $\Phi$ vanish:
\begin{equation}\label{eq:vanishing-period}
	\Re \int_\gamma \Phi =0\quad \text{for every closed path $\gamma$ in $\Ncal$}.
\end{equation}
The periods of  the imaginary part $2\Im \Phi=d^c X$ (the {\em conjugate differential} of $X$,
see \eqref{eq:conjugatedif}) define the {\em flux homomorphism}  (cf.\ \eqref{eq:Flux} for the scalar case): 
\[
	\Flux_X \colon H_1(\Ncal;\Z)\to \R^n,\quad \  \Flux_X(\gamma) = 2 \int_\gamma \Im \Phi.
\]
If $\Flux_X=0$ then the map $Z=X+\imath Y \colon \Ncal\to \C^n$, defined by
$Z(p) = \int^p \Phi$ $(p\in\Ncal)$, is a {\em holomorphic null immersion} with  
$ 
	\Phi= dZ = \di Z=2\, \di X = 2\imath\, \di Y.
$
In this case, $X$ is said to be {\em flux-vanishing} and the associated family of conformal minimal immersions 
$X^t:=\Re(e^{\imath t}Z)\colon\Ncal\to\r^n$, $t\in[0,2\pi)$, is well defined.

Denote by $\pi\colon\c^n\setminus\{0\} \to\cp^{n-1}$ the canonical projection onto the projective space,
so $\pi(z_1,\ldots,z_n)=[z_1:\cdots:z_n]$ are homogeneous coordinates on $\CP^{n-1}$.
Since the $1$-form $\di X=\Phi=(\phi_1,\ldots,\phi_n)$ \eqref{eq:diX} does not vanish anywhere on $\Ncal$ 
(see the last equation in \eqref{eq:Xconf}), it determines a holomorphic map 
\begin{equation}\label{eq:Gauss-map}
	\Ncal\lra\cp^{n-1}, \quad p\longmapsto [\phi_1(p):\phi_2(p):\cdots:\phi_n(p)]
\end{equation}
of Kodaira type, called the {\em generalized Gauss map} of $X$. By \eqref{eq:null}, this map assumes 
values in the projection of $\Agot_*=\Agot\setminus\{0\}$ in $\cp^{n-1}$ (a smooth quadratic complex hypersurface
in $\cp^{n-1}$), where $\Agot$ is the {\em null quadric}
\begin{equation} \label{eq:null-quadric}
	\Agot= \{z=(z_1,\ldots,z_n)\in \C^n: z_1^2+z_2^2+\cdots + z_n^2 =0\}.
\end{equation} 

Conversely, given a nowhere vanishing holomorphic $1$-form $\Phi=(\phi_1,\ldots,\phi_n)$
on $\Ncal$ satisfying the nullity condition \eqref{eq:null} and with vanishing real periods
\eqref{eq:vanishing-period}, the map $X\colon \Ncal\to\R^n$ defined by
\begin{equation}
\label{eq:X}
	X(p)= 2 \int^p \Re(\Phi)\in \R^n,\quad p\in\Ncal
\end{equation}
is a conformal minimal immersion with $\di X=\Phi$. 

Let $\theta$ be a nowhere vanishing holomorphic $1$-form on $\Ncal$ (see \cite{GunningNarasimhan1967MA}).
Clearly, a holomorphic 1-form $\Phi=(\phi_1,\ldots,\phi_n)$ satisfies the 
nullity condition \eqref{eq:null} if and only if the map $f=\Phi/\theta\colon \Ncal\to \C^n$ 
assumes values in the null quadric $\Agot$ \eqref{eq:null-quadric}.
This reduces the construction of conformal minimal immersions $\Ncal\to \R^n$ to the 
construction of holomorphic maps $f\colon \Ncal\to \Agot_*$ such that
the holomorphic $1$-form $f\theta$ has vanishing real periods over all closed curves in $\Ncal$.

%
%
If $n=3$ and $\phi_1\not\equiv\imath\phi_2$, the triple $\Phi=(\phi_1,\phi_2,\phi_3)$ can be written in the form
\begin{equation}\label{eq:gphi3}
     \Phi=\left( \frac12\Big( \frac1{g}-g\Big) \,,\, \frac{\imath}2\Big( \frac1{g}+g\Big) \,,\, 1\right)\phi_3, \quad     
     g=\frac{\phi_3}{\phi_1-\imath\phi_2}.
\end{equation}
Here, $g$ is a meromorphic function on $\Ncal$ whose zeros and poles coincide with the zeros of $\phi_3$, 
with the same order. Moreover, $g$ agrees with the stereographic projection of the Gauss map $\Ncal\to\s^2$ of $X$ 
(see \cite{Osserman-book}); this is why $g$ is called the  {\em complex Gauss map} of $X$. The pair $(g,\phi_3)$ 
is known as the {\em Weierstrass data} of $X$.
Writing $\phi_3=f_3\theta$ where $\theta$ is as above, we get the following formula, analogous to \eqref{eq:gphi3}, 
for the holomorphic map $f=\Phi/\theta\colon \Ncal\to \Agot_*\subset\c^3$:
\begin{equation}\label{eq:Weierstrass}
     f=\left( \frac12\Big( \frac1{g}-g\Big) \,,\, \frac{\imath}2\Big( \frac1{g}+g\Big) \,,\, 1\right)f_3.
\end{equation}
This classical {\em Weierstrass representation formula}
has been the main tool in the construction of minimal surfaces in $\r^3$ over the past half-century.
It has allowed to give explicit examples by pointing out Riemann surfaces, pairs $(g,f_3)$, and 
nowhere vanishing holomorphic $1$-forms $\theta$ such that the resulting vectorial $1$-form $f\theta$, 
where $f$ is given by \eqref{eq:Weierstrass}, 
has vanishing real periods  (see e.g.\ \cite{Osserman-book,ChenGackstatter1982MA,Costa1984BSBM,HoffmanMeeks1990AM} 
for examples of complete conformal minimal surfaces of finite total curvature). 
This representation formula has also been fundamental in the global theory of minimal surfaces and was
exploited in the proof of general existence results of complete minimal surfaces in $\r^3$. L\'opez and Ros observed in 
\cite{LopezRos1991JDG} that if $(g,f_3\theta)$ 
are the Weierstrass data of a conformal minimal immersion on a  Riemann surface 
$\Ncal$ and $h\colon\Ncal\to\c\setminus\{0\}$ is a holomorphic function, then the pair $(hg,f_3\theta)$ also furnishes a 
conformal minimal immersion provided that the period vanishing condition is satisfied.  
Combining this simple transformation with Runge's approximation theorem \cite{Runge1885AM} has been the key 
in the construction of complete bounded minimal surfaces (see \cite{JorgeXavier1980AM,Nadirashvili1996IM} and the discussion of the 
Calabi-Yau problem in Sections \ref{sec:intro-Jordan} and \ref{sec:Jordan}), and of properly immersed minimal surfaces in $\r^3$ with 
hyperbolic conformal type (see \cite{Morales2003GAFA,AlarconLopez2012JDG} and the discussion of the Sullivan conjecture and the 
Schoen-Yau problem in Sections \ref{sec:Icomplete} and \ref{sec:Iproper}). 
An analogue of the formula \eqref{eq:Weierstrass} is not available in higher dimensions, 
and hence constructing holomorphic maps $\Ncal\to\Agot_*\subset\c^n$ for $n>3$ requires a different approach
presented in this paper.

To treat the problem on a non-orientable surface $\UnNcal$, let $\iota\colon \Ncal\to \UnNcal$ 
be a $2$-sheeted oriented covering by a Riemann surface $\Ncal$ with a fixed-point-free antiholomorphic 
involution $\Igot\colon \Ncal\to\Ncal$ (see Section \ref{sec:conformal}).
For every conformal minimal immersion $\UnX \colon \UnNcal \to \R^n$,
the composition $X=\UnX\circ \iota \colon \Ncal \to\R^n$ is a $\Igot$-invariant conformal minimal
immersion; conversely, a $\Igot$-invariant conformal minimal immersion $X\colon \Ncal\to\R^n$
passes down to a conformal minimal immersion  $\UnX=(\UnX_1,\ldots,\UnX_n) \colon \UnNcal \to \R^n$.

By Lemma \ref{lem:invariant}, every $\Igot$-invariant conformal minimal immersion
$X\colon\Ncal \to\R^n$ is of the form \eqref{eq:X}
where $\Phi=(\phi_1,\ldots,\phi_n)$ is a holomorphic $1$-form on $\Ncal$ that is $\Igot$-invariant
($\Igot^* \Phi=\overline \Phi$),  it satisfies the nullity condition \eqref{eq:null}, and 
has vanishing real periods \eqref{eq:vanishing-period}. Furthermore, any such $\Phi$ is of the form 
$\Phi=f\theta$ where $\theta$ is a nowhere vanishing $\Igot$-invariant 1-form furnished by Proposition \ref{prop:Cor6.5} 
and $f=(f_1,\ldots,f_n)\colon \Ncal\to \Agot_*$ is an $\Igot$-invariant holomorphic map 
(satisfying $f\circ \Igot =\bar f$) such that $\int_\gamma \Re (f\theta)=0$ for every closed loop 
$\gamma$ in $\Ncal$.

If $n=3$ then $\phi_1\not\equiv\imath\phi_2$ holds (otherwise $\Phi\equiv 0$ by the $\Igot$-invariancy of $\Phi$), 
and the meromorphic Gauss map $g$ given by \eqref{eq:gphi3} satisfies the symmetry condition
\begin{equation}\label{eq:gIg}
      g\circ\Igot=-\frac1{\bar g}.
\end{equation}
Similarly to the orientable case, any pair $(g,f_3)$ where $g$ is a meromorphic function on $\Ncal$ satisfying \eqref{eq:gIg} and 
$f_3\colon \Ncal\to\c$ is an $\Igot$-invariant holomorphic function such that the zeros of $f_3$ coincide with the zeros and poles 
of $g$, with the same order, furnishes an $\Igot$-invariant holomorphic function $f\colon \Ncal\to \Agot_*\subset\c^3$ by the 
Weierstrass formula \eqref{eq:Weierstrass}. Most examples of non-orientable minimal surfaces in $\r^3$ given in the last five 
decades rely on these formulas; see for instance \cite{Meeks1981DMJ,Kusner1987BAMS,Lopez1993} for some explicit 
complete examples of finite total curvature, and 
\cite{Lopez1988PAMS,LopezMartinMorales2006TAMS,FerrerMartinMeeks2012AM,AlarconLopez2015GT} for constructions 
exploiting the L\'opez-Ros transformation and Runge's theorem. 

Constructing $\Igot$-invariant holomorphic functions $\Ncal\to\Agot_*\subset\c^n$ for $n>3$ is a much more involved task.

Let us illustrate these notions on the following  well known example due to Meeks \cite[Theorem 2]{Meeks1981DMJ}.
See also Example \ref{ex:Mobius} in Section \ref{sec:Iproper} which gives a properly embedded minimal M\"obius strip in $\R^4$.

%
%
%
%
\begin{example}[Meeks's minimal M\"obius strip in $\R^3$] \label{ex:MeeksMobius}  
Let $\Ncal =\C_*:=\C\setminus\{0\}$, and let $\Igot\colon  \c_* \to\c_*$ be the fixed-point-free 
antiholomorphic involution given by 
\[
    	\Igot(\zeta)=-\frac1{\bar\zeta},\quad\  \zeta\in\c_*.
\]
Let $\theta$ be the $\Igot$-invariant nowhere vanishing holomorphic $1$-form on $\C_*$ given by
\[
   	\theta=\imath \, \frac{d\zeta}{\zeta}.
\] 
Consider also the functions
\[
	g(\zeta)=\zeta^2\,\dfrac{\zeta+1}{\zeta-1}, 
	\quad f_3(\zeta)= 2\, \dfrac{\zeta^2-1}{\zeta},\quad \zeta\in\C_*.
\]
Note that $f_3$ is holomorphic and $\Igot$-invariant, whereas $g$ is meromorphic and satisfies the symmetry condition \eqref{eq:gIg}. 
Since the zeros of $f_3$ coincide with the zeros and poles of $g$ with the same order, the pair $(g,f_3)$ determines an 
$\Igot$-invariant holomorphic map $f\colon\c_*\to\Agot_*\subset\c^3$ by the Weierstrass formula \eqref{eq:Weierstrass}. 
Moreover, it is easily seen that the holomorphic $1$-form $\Phi=f\theta$ is exact on $\c_*$ and hence determines 
a flux-vanishing $\Igot$-invariant conformal minimal immersion $X\colon\c_*\to\r^3$ by the expression \eqref{eq:X}. 
Therefore, $X$ induces a conformal minimal immersion $\UnX \colon \C_*/\Igot \to\R^3$. 
This is Meeks's complete immersed minimal M\"obius strip in $\R^3$ with finite total curvature $-6\pi$. 
See Figure \ref{fig:strip3}.
\begin{figure}[ht]
    \begin{center}
    \scalebox{0.4}{\includegraphics{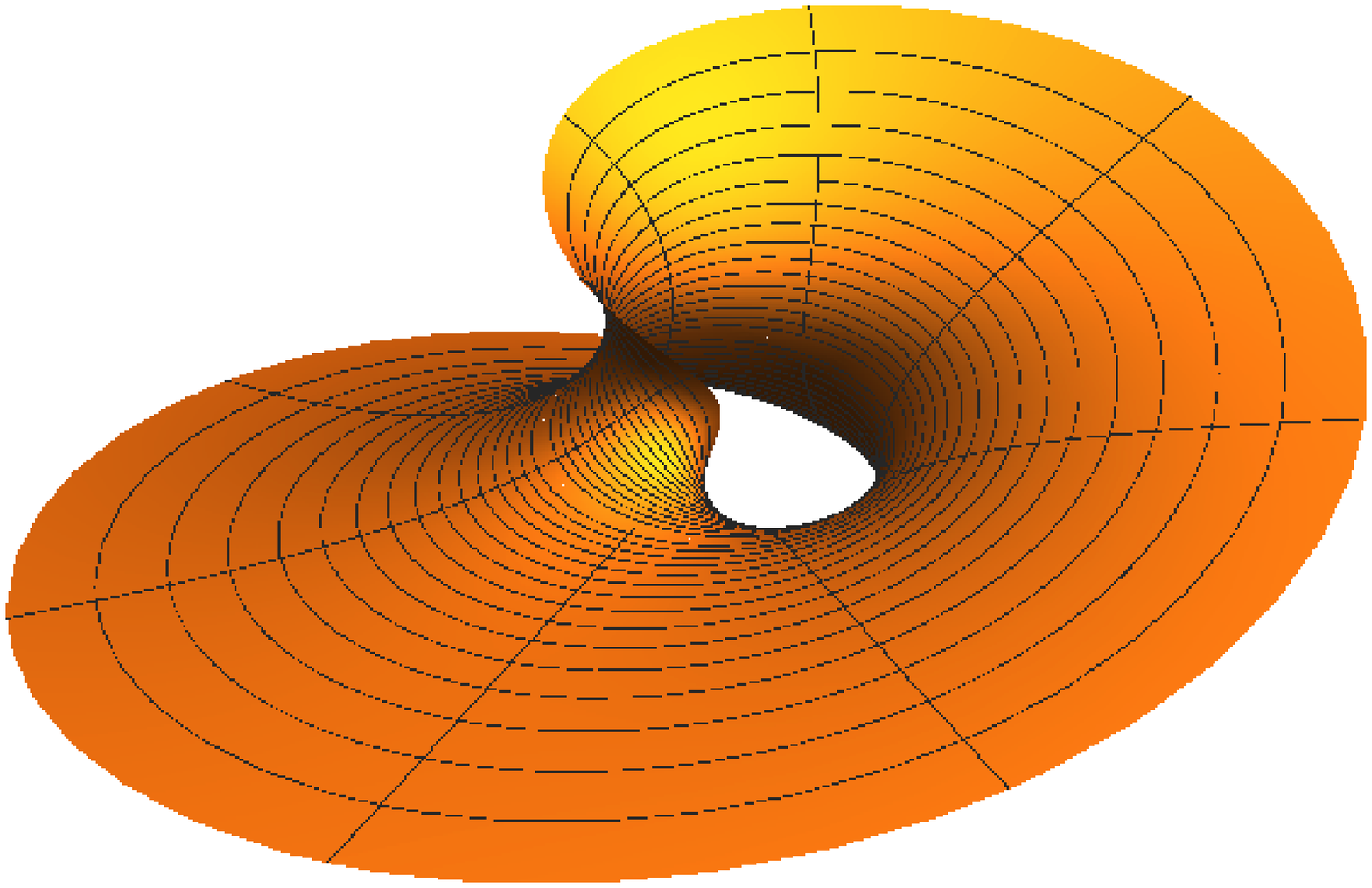}}
        \end{center}
\caption{Meeks's minimal M\"obius strip.}\label{fig:strip3}
\end{figure}
\qed\end{example}

Assume now that the Riemann surface $\Ncal$ has finite topology. Let $\Bcal$ be a homology basis
given by \eqref{eq:basis1},  \eqref{eq:basis2}, and let $\Bcal^+=\{\delta_1,\ldots,\delta_l\} \subset \Bcal$ 
be given by \eqref{eq:B0}. Fix a nowhere vanishing $1$-form $\theta\in \OmI(\Ncal)$.  
Given these choices, let
\[
	\Pcal=(\Pcal_1,\ldots, \Pcal_l)\colon \Oscr(\Ncal,\C^n)\to (\C^n)^l
\]
denote the associated {\em period map} 
\begin{equation}
\label{eq:periodmap}
	\Pcal(f) = \left( \int_{\delta_j} f\theta \right)_{j=1,\ldots,l} \in (\C^n)^l,\quad\  f\in \Oscr(\Ncal,\C^n).
\end{equation}
This is a vector-valued version of the map \eqref{eq:PB_0}. Lemma \ref{lem:partial-period} 
holds in this case as well by applying it componentwise.

%
%

\section{Notation}\label{sec:notation}
We conclude this chapter by introducing some notation.
 
If $\Ncal$ is an open Riemann surface and $n\ge 3$ is an integer, we shall denote by 
\[
	\CMI^n(\Ncal)
\]
the space of all conformal minimal immersions $\Ncal\to\R^n$ with the compact-open topology. 
If $\Ncal$ carries a fixed-point-free antiholomorphic involution, then 
\begin{equation}\label{eq:CMI}
	\CMI^n_\Igot(\Ncal) = \{X\in \CMI^n(\Ncal): X\circ \Igot =X\}
\end{equation}
denotes the space of all $\Igot$-invariant conformal minimal immersions $\Ncal\to\r^n$.

Given a locally closed complex submanifold $Z$ of $\C^n$ that is invariant 
under conjugation $z\mapsto \bar z$ (for example, the punctured null quadric $\Agot_*$ given by
\eqref{eq:null-quadric}), a compact $\Igot$-invariant subset $K\subset \Ncal$ and an integer $k\ge 0$, 
we set 
\begin{eqnarray*}
	\OI(\Ncal,Z) &=& \{f\in \Oscr(\Ncal,Z) : f\circ\Igot=\bar f\}, \\
	\Oscr_{\Igot}(K,Z) &=&  \{f\in \Oscr(K,Z) : f\circ\Igot=\bar f\}, \\
	\Ascr^k_{\Igot}(K,Z) &=&  \{f\in \Ascr^k(K,Z): f\circ\Igot=\bar f\}. 
\end{eqnarray*}
When $Z=\C$, we simply write $\OI(\Ncal,\C)=\OI(\Ncal)$, etc.

The complex Banach space $\Ascr^k(K,\C^n)$ decomposes as the direct sum
\[
	\Ascr^k(K,\C^n) = \Ascr_\Igot^k(K,\C^n)  \oplus  \imath\Ascr_\Igot^k(K,\C^n)
\]
of the real Banach subspaces $\Ascr_\Igot^k(K,\C^n)$ and 
\[
	\imath\Ascr_\Igot^k(K,\C^n) = \{\imath f\colon f\in \Ascr_\Igot^k(K,\c^n)\} = \{f\in \Ascr^k(K,\c^n) : f\circ\Igot=-\bar f\},
\]
with the projections on the respective factors given by \eqref{eq:projections}.

Choose a nowhere vanishing holomorphic $1$-form $\theta$ on $\Ncal$.
The discussion in Section \ref{sec:CMI} shows that $X\in \CMI^{n}(\Ncal)$ if and only if
\[
	f:=2\di X/\theta \in \Oscr(\Ncal,\Agot_*)\quad \text{and}\quad 
	\int_\gamma \Re(f\theta)=0\ \  \text{for all}\ \ \gamma\in H_1(\Ncal;\Z).
\]
Similarly, we have $X\in \CMI^{n}_\Igot(\Ncal)$ if and only if
\[
	f:=2\di X/\theta \in \Oscr_\Igot(\Ncal,\Agot_*)\quad \text{and}\quad 
	\int_\gamma \Re(f\theta)=0\ \  \text{for all}\ \ \gamma\in H_1(\Ncal;\Z).
\]
In both cases, we obtain $X$ from $f$ by integration, $X(p)=\int^p \Re(f\theta)$ $(p\in \Ncal)$.

The notation 
\begin{equation}\label{eq:triple} 
	(\Ncal,\Igot,\theta)
\end{equation} 
shall be used for a  {\em standard triple} consisting of a connected open Riemann surface $\Ncal$ endowed 
with a fixed-point-free antiholomorphic involution $\Igot\colon \Ncal\to\Ncal$ and an $\Igot$-invariant  
nowhere vanishing holomorphic $1$-form $\theta$ on $\Ncal$ (cf.\ Proposition \ref{prop:Cor6.5}). 
We shall also write 
\begin{equation}\label{eq:pair} 
(\Ncal,\Igot)
\end{equation}
and call it a {\em standard pair} when there is no need to emphasize the choice of $\theta$.


\chapter{Gluing $\Igot$-invariant sprays and applications}  \label{ch:sprays}

In this chapter, we develop the complex analytic methods that will be used in the paper.
They originate in modern Oka theory (cf.\ \cite[Chapter 5]{Forstneric2011-book}) 
and have already been exploited in the construction of conformal 
minimal immersions from Riemann surfaces; see \cite{AlarconDrinovecForstnericLopez2015PLMS,AlarconForstneric2014IM,AlarconForstneric2015MA,
AlarconForstnericLopez2016MZ}.
Here we adapt them to Riemann surfaces endowed with a fixed-point-free antiholomorphic involution $\Igot$. 

The main new technical results include Proposition \ref{prop:gluing}  on gluing pairs of $\Igot$-invariant sprays,
and Proposition \ref{prop:period-dominating-spray} on the existence of $\Igot$-invariant period dominating sprays.
In Section \ref{sec:structure} we use Proposition \ref{prop:period-dominating-spray} to 
show that the set of all conformal minimal immersions from a
compact bordered non-orientable surface is a real analytic Banach manifold.
These techniques are applied in Section \ref{sec:basic} to obtain basic approximation results for  
$\Igot$-invariant conformal minimal immersions which are in turn used in the 
proof of Mergelyan approximation theorems  for conformal minimal immersions from non-orientable surfaces
into $\R^n$ (cf.\ Theorems \ref{th:Mergelyan} and \ref{th:Mergelyan1}). 
The gluing techniques are also used in Section  \ref{sec:RH} to find approximate
solutions of certain Riemann-Hilbert type boundary value problems for $\Igot$-invariant 
conformal minimal immersions, generalizing what has been done in the orientable case
in \cite{AlarconDrinovecForstnericLopez2015PLMS,AlarconForstneric2015MA}. 
The Riemann-Hilbert method is an essential tool in the construction of complete non-orientable minimal surfaces
in Chapter \ref{ch:applications}.

For any $n\in\N$ we denote by $\Igot_0$ the 
anti\-holo\-morphic involution on $\C^n$ given by complex conjugation: $\Igot_0(z)=\bar z$. 
Note that $\Oscr_{\Igot_0}(\C^n)=\{f\in\Oscr(\C^n): f\circ \Igot_0 = \bar f\}$ 
is the algebra of all entire functions $f(z)=\sum c_\alpha z^\alpha$ with real coefficients $c_\alpha\in\R$.


\section{$\Igot$-invariant sprays}

\begin{definition}\label{def:spray}
Let $(\Ncal,\Igot)$ be a standard pair \eqref{eq:pair}, and let $Z$ be a locally closed $\Igot_0$-invariant 
complex submanifold of a complex Euclidean space $\C^n$. 
A spray of maps $\Ncal\to Z$ of class $\Oscr_{\Igot,\Igot_0}(\Ncal)$ is a holomorphic map
$F\colon \Ncal\times U\to Z$, where $U\subset \C^m$ is an open $\Igot_0$-invariant neighborhood 
of the origin in $\C^m$ for some $m\in \N$, which is $(\Igot,\Igot_0)$-invariant in the following sense:
\begin{equation}\label{eq:invariant}
	F(\Igot(p),\bar \zeta) = \overline{F(p,\zeta)},\quad\ p\in \Ncal,\ \zeta\in U.
\end {equation}
We shall also write $F\in \Oscr_{\Igot,\Igot_0}(\Ncal\times U, Z)$.
\end{definition}

%
%
\begin{lemma}\label{lem:composition}
If $F\colon \Ncal\times U\to Z$ is a $(\Igot,\Igot_0)$-invariant holomorphic spray and 
$\alpha\colon \Ncal \to U$ is a $\Igot$-invariant holomorphic map, then 
$\Ncal \ni p\mapsto F(p,\alpha(p))\in Z$ is a $\Igot$-invariant holomorphic map.
In particular, $F(\cdotp,\zeta)\colon \Ncal\to Z$ is $\Igot$-invariant when 
$\zeta\in\R^m\cap U$ is real.
\end{lemma}

\begin{proof}
This follows from	
$F(\Igot(p),\alpha(\Igot(p))) = F(\Igot(p),\overline{\alpha(p)}) = \overline{F(p,\alpha(p))}$.
\end{proof}

%
%

\begin{example}\label{ex:invariant}
1.  If $f\colon \Ncal\to \C^n$ is a $\Igot$-invariant holomorphic map (i.e.,\ $f\circ \Igot=\bar f$)
and $g\colon \C^m\to\C^n$ is a holomorphic map sending $\R^m$ (the real subspace of $\C^m$)
into $\R^n$, then the holomorphic map $F\colon \Ncal\times\C^m \to\C^n$ given by 
$F(p,\zeta)=f(p)+g(\zeta)$ is $(\Igot,\Igot_0)$-invariant.

2. Assume that $V=\sum_{j=1}^n a_j(z)\frac{\di}{\di z_j}$ is a holomorphic vector field on $\C^n$
whose coefficients $a_j(z)$ assume real values on $\R^n\subset \C^n$.
The flow $\phi_t(z)$ of $V$ then preserves the real subspace $\R^n$ for all
$t\in \R$, and hence it satisfies the condition
\begin{equation}\label{eq:invariant-flow}
	\phi_{\bar t}(\bar z) = \overline{\phi_t(z)} 
\end{equation}
for all points $(z,t)$ in the fundamental domain of $V$.  In particular, if $V$ is complete then
$\{\phi_t\}_{t\in\R} \subset \Ocal_{\Igot_0}(\C^n,\C^n)$ is a $1$-parameter group 
consisting of $\Igot_0$-invariant holomorphic
automorphisms of $\C^n$. Given an $\Igot$-invariant holomorphic map 
$f\colon \Ncal\to\C^n$, the map $(p,t)\mapsto \phi_t(f(p))$ is $(\Igot,\Igot_0)$-invariant
on a neighborhood of $\Ncal\times \{0\}$ in $\Ncal\times \C$ by Lemma \ref{lem:composition}.
\qed \end{example}

Recall that a subset $A\subset \Ncal$ is said to be $\Igot$-invariant if $\Igot(A)=A$. 

%
%

\begin{definition} 
\label{def:invariant-spray}
Let $(\Ncal,\Igot)$ be a standard pair \eqref{eq:pair}, let $A$ be a compact $\Igot$-invariant domain in  $\Ncal$,  
and let $Z \subset \C^n$ be a locally closed $\Igot_0$-invariant complex submanifold of $\C^n$. 
A {\em spray of maps of class $\Ascr_{\Igot,\Igot_0}(A)$ with values in $Z$}  is a continuous $(\Igot,\Igot_0)$-invariant
map $F\colon A\times U\to Z$ (see \eqref{eq:invariant}) that is holomorphic on $\mathring A \times U$,
where $U\subset\C^m$ is an open $\Igot_0$-invariant neighborhood of the origin in some $\C^m$.
We shall also write $F\in \Ascr_{\Igot,\Igot_0}(A\times U, Z)$.
The map $f=F(\cdotp,0)\colon A\to Z$ is called the {\em core} of $F$. 
The spray $F$ is {\em dominating} over a subset $C\subset A$ if its {\em vertical differential}
\begin{equation}\label{eq:domination}
	\di_\zeta\big|_{\zeta=0} F(p,\zeta)  \colon \C^m \longrightarrow T_{F(p,0)} Z 
\end{equation}
is surjective for every point $p\in C$; $F$ is dominating if this holds on $C=A$. 
\end{definition}

Holomorphic sprays were introduced into Oka theory by Gromov \cite{Gromov1989JAMS},
although special cases have already been used by Grauert in his proof of the 
Oka-Grauert principle \cite{Grauert1957MA,Grauert1958MA}. See \cite{Forstneric2011-book}
for a more comprehensive account of the use of sprays in complex analysis.

The following type of sprays will be of main importance to us.

\begin{example}[Sprays defined by vector fields]
Assume that $Z$ is a locally closed $\Igot_0$-invariant  complex submanifold of $\C^n$. 
Let $V_1,\ldots, V_m$ be entire vector fields  on $\C^n$ 
which are tangent to $Z$ and whose coefficients are real on $\R^n$.
The flow $\phi^j_t(z)$ of $V_j$ then satisfies condition  \eqref{eq:invariant-flow}. 
Given a compact $\Igot$-invariant domain $A\subset \Ncal$
and a $\Igot$-invariant map $f\colon A\to Z$ of class $\Ascr(A)$, there is an $r>0$ such that
the map $F\colon A \times r\B^m\to \c^n$ given by
\begin{equation}\label{eq:flow-spray}
	F(p,t_1,\ldots,t_m)=\phi^1_{t_1}\circ \phi^2_{t_2}\circ\cdots\circ \phi^m_{t_m} (f(p)),
	\quad p\in A,
\end{equation}
is well defined for all $t=(t_1,\ldots,t_m)\in r\B^m$; we can take $r=+\infty$ if the vector 
fields $V_j$ are complete. Cleary, $F$ is a $(\Igot,\Igot_0)$-invariant spray of class $\Ascr_{\Igot,\Igot_0}(A)$
with values in $Z$ and with the core $f$. (Compare with Example \ref{ex:invariant}.) From 
\[
	\frac{\di}{\di t_j}\bigg|_{t=0} F(p,t)= V_j(f(p))\quad \text{for $j=1,\ldots,m$}
\] 
we see that the spray $F$ is dominating on $A$  if and only if the vectors 
$V_1(z),\ldots,V_m(z)$ span the tangent space $T_z Z$ at every point $z=f(p)$, $p\in A$.

More generally, if $h_1,\ldots, h_m\in \OI(A)$ then the map $F\colon A \times \C^m\to \C^n$, 
\begin{equation}\label{eq:flow-spray2}
	F(p,t_1,\ldots,t_m) =
	\phi^1_{t_1h_1(p)}\circ \phi^2_{t_2 h_2(p)}\circ\cdots\circ \phi^m_{t_m h_m(p)} (f(p))
\end{equation}
is a $(\Igot,\Igot_0)$-invariant spray with values in $Z$ satisfying
\[
	\frac{\di}{\di t_j}\bigg|_{t=0} F(p,t)=   h_j(p)\, V_j(f(p))\quad \text{for $j=1,\ldots,m$}.
\] 
Hence, $F$ is dominating at every point $p\in A$ where $h_j(p)\ne 0$ for all 
$j=1,\ldots,m$ and the vectors $V_1(f(p)),\ldots,V_m(f(p))$ span $T_{f(p)} Z$.

The most relevant case for us is $Z=\Agot_*=\Agot\setminus\{0\}$, where $\Agot\subset \c^n$ is the null quadric
\eqref{eq:null-quadric}. The linear holomorphic vector fields 
\begin{equation}
\label{eq:Vjk}
	V_{j,k}=z_j\frac{\di}{\di z_k} - z_k\frac{\di}{\di z_j},\quad 1\le j\ne k\le n,
\end{equation}
are $\C$-complete, their coefficients are real on $\R^n$,  they are tangent to $\Agot_*$, 
and they span the tangent space $T_z\Agot_*$ at each point $z\in\Agot_*$.  
\qed\end{example}

%
%

\section{Gluing $\Igot$-invariant sprays on $\Igot$-invariant Cartan pairs}
\label{sec:gluing}

%
%
\begin{definition}[$\Igot$-invariant Cartan pairs] \label{def:Cartan}
Let $(\Ncal,\Igot)$ be a standard pair \eqref{eq:pair}. A pair $(A,B)$ of compact sets in $\Ncal$ is said to be 
a {\em $\Igot$-invariant Cartan pair} if the following conditions hold:
\begin{itemize}
\item[\rm (a)] the sets $A,B, A\cap B$ and $A\cup B$ are $\Igot$-invariant with $\Cscr^1$ boundaries, and 
\vspace{1mm}
\item[\rm (b)]
$\overline{A\setminus B} \cap \overline{B\setminus A}=\emptyset$  (the separation property). 
\end{itemize}
A $\Igot$-invariant Cartan pair $(A,B)$ is {\em special} if $B=B'\cup \Igot(B')$, where
$B'$ is a compact set with $\Cscr^1$ boundary in $\Ncal$ and $B'\cap \Igot(B')=\emptyset$.

A special Cartan pair $(A,B)$ is {\em very special} if the sets $B'$ and $A\cap B'$ are disks
(hence, $\Igot(B')$ and $A\cap \Igot(B')$ are also disks).
See Figure \ref{fig:Cartan}.
\end{definition}
\begin{figure}[ht]
    \begin{center}
    \scalebox{0.16}{\includegraphics{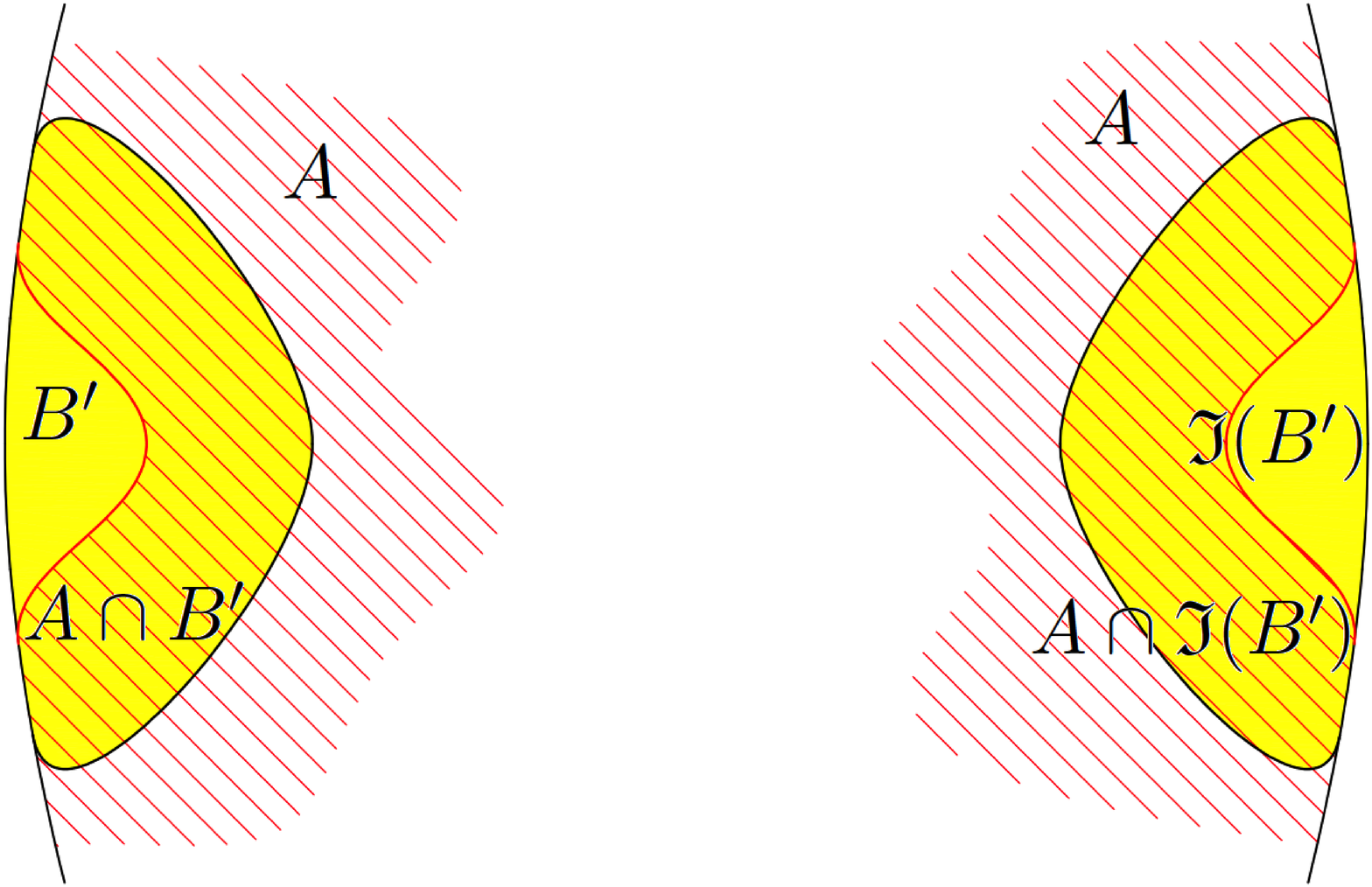}}
        \end{center}
\caption{A very special Cartan pair.}\label{fig:Cartan}
\end{figure}

The following is one of our main technical tools in the construction
of $\Igot$-invariant conformal minimal immersions. The corresponding gluing lemma in several variables
(without $\Igot$-invariance) is given by \cite[Proposition 5.9.2, p.\ 216]{Forstneric2011-book}.

%
%
\begin{proposition} [Gluing pairs of $\Igot$-invariant sprays] 
\label{prop:gluing}
Assume that $(\Ncal,\Igot)$ is a standard pair \eqref{eq:pair}, $(A,B)$ is a special $\Igot$-invariant Cartan pair in $\Ncal$,
$Z$ is a locally closed $\Igot_0$-invariant complex submanifold of $\C^n$ (where $\Igot_0(z)=\bar z$ is the conjugation 
on $\C^n$), $r>0$ is a real number, and $F\colon A\times r\B^m \to Z$ is a spray  of class $\Ascr_{\Igot,\Igot_0}(A)$ 
which is dominating over the set $C=A\cap B$.  Then, there is a number $r'\in (0,r)$ depending on $F$ with the following property.
For every $\epsilon>0$ there exists $\delta>0$ such that, given a spray $G\colon B\times r\B^m\to Z$ of class 
$\Ascr_{\Igot,\Igot_0}(B)$ satisfying $||F-G||_{0,C\times r\B^m} <\delta$, there exists a spray 
$H\colon (A\cup B) \times r'\B^m \to Z$ of class $\Ascr_{\Igot,\Igot_0}(A\cup B)$ 
satisfying $||H-F||_{0,A \times r'\B^m} <\epsilon$.
\end{proposition}

The main step in the proof of Proposition \ref{prop:gluing} is the following splitting lemma.
See \cite[Proposition 5.8.1, p.\ 211]{Forstneric2011-book} for a more general geometric setting.
The new aspect here is $\Igot$-invariance.
 
%
%
\begin{lemma}[A splitting lemma for $\Igot$-invariant maps] 
\label{lem:splitting}
Assume that $(\Ncal,\Igot)$ and $(A,B)$ are as in Proposition \ref{prop:gluing}, and let $C=A\cap B$. 
Given numbers $\epsilon>0$ and $0<r_2<r_1$, there exists a number $\delta>0$
with the following property. Let $\gamma\colon C\times r_1\B^m \to C\times \C^m$
be a map  of the form 
\begin{equation}\label{eq:transition}
	\gamma(p,\zeta) = \bigl(p,\zeta+c(p,\zeta)\bigr),\quad p\in C,\ \zeta\in r_1\B^m,
\end{equation}
where $c\in \Ascr_{\Igot,\Igot_0}(C\times r_1\B^m,\C^m)$ satisfies $||c||_{0,C\times r_1\B^m}<\delta$. 
Then there exist maps
\[
	\alpha\colon A\times r_2\B^m\lra A\times \C^m,\quad 
	\beta \colon B\times r_2\B^m\lra B\times \C^m
\]
of the form 
\begin{equation}\label{eq:alpha-pzeta}
	\alpha(p,\zeta)=\left(p,\zeta+a(p,\zeta)\right), \quad \beta(p,\zeta)=\left(p,\zeta+b(p,\zeta)\right),
\end{equation} 
where the maps $a\in  \Ascr_{\Igot,\Igot_0}(A\times r_2\B^m,\C^m)$ and $b\in  \Ascr_{\Igot,\Igot_0}(B\times r_2\B^m,\C^m)$ 
are uniformly $\epsilon$-close to $0$ on their respective domains, such that 
\begin{equation}\label{eq:split}
	\gamma \circ \alpha =\beta\quad \text{holds on}\ C\times r_2\B^m.
\end{equation}
\end{lemma}

\begin{proof} 
The proof is similar to that of \cite[Proposition 5.8.1]{Forstneric2011-book}. 
We begin by considering the following Cousin-I problem:
given a $(\Igot,\Igot_0)$-invariant map $c\colon C\times r_1\B^m \to \C^m$
of class $\Ascr(C\times r_1\B^m)$, we claim that it splits as a difference
\begin{equation}\label{eq:CousinI}
	c=b-a, 
\end{equation}
where $a\colon A\times r_1\B^m \to \C^m$ and $b\colon B\times r_1\B^m \to \C^m$
are $(\Igot,\Igot_0)$-invariant maps of class $\Ascr(A\times r_1\B^m)$ and $\Ascr(B\times r_1\B^m)$, 
respectively, satisfying the estimates
\begin{equation}\label{eq:CIestimate}
	||a||_{0,A\times r_1\B^m} \le M ||c||_{0,C\times r_1\B^m}, \quad
	||b||_{0,B\times r_1\B^m} \le M ||c||_{0,C\times r_1\B^m}
\end{equation}
for some constant $M>0$ depending only on the pair $(A,B)$. In particular, $a$ and $b$ are 
uniformly $\epsilon$-close to $0$ on their respective domains provided that 
$||c||_{0,C\times r_1\B^m}$ is sufficiently small. To find such splitting, we first solve the  
standard Cousin-I problem to find maps $a'$ and $b'$ as above 
(not necessarily $(\Igot,\Igot_0)$-invariant) satisfying $c=b'-a'$  on $C\times r_1\B^m$ and the 
estimates  \eqref{eq:CIestimate}. This is classical and uses a bounded solution operator
for the $\dibar$-equation on $A\cup B$; the variable $\zeta\in r_1\B^m$ plays
the role of a parameter. Since $c$ is $(\Igot,\Igot_0)$-invariant and hence satisfies 
\[
	c(p,\zeta) =\frac{1}{2}\left(c(p,\zeta) + \overline{ c(\Igot(p),\bar\zeta)}\right),
	\quad p\in C,\ \zeta\in r_1\B^m,
\]
it follows that the $(\Igot,\Igot_0)$-invariant $\C^m$-valued maps
\[
	a(p,\zeta) =\frac{1}{2}\left(a'(p,\zeta) + \overline{ a'(\Igot(p),\bar\zeta)}\right),
	\quad
	b(p,\zeta) =\frac{1}{2}\left(b'(p,\zeta) + \overline{ b'(\Igot(p),\bar\zeta)}\right),
\]
defined for $p\in A$ and $p\in B$, respectively, 
satisfy \eqref{eq:CousinI} and the estimates \eqref{eq:CIestimate}.

By using this Cousin-I problem, the proof  is then completed  
by an iterative process exactly as in \cite[proof of Theorem 3.2]{DrinovecForstneric2007DMJ}.
(One can also apply the implicit function theorem approach, cf.\ 
\cite[proof of Proposition 5.8.1]{Forstneric2011-book}.) It is easily seen (cf.\ Lemma \ref{lem:composition}) 
that the composition $\alpha_1 \circ \alpha_2$ of two maps of type \eqref{eq:alpha-pzeta} is still of the same type; 
in particular, its second component  is $(\Igot,\Igot_0)$-invariant if this holds for $\alpha_1$ and $\alpha_2$. 
Hence, the maps $\alpha$ and $\beta$ obtained in the limit have $(\Igot,\Igot_0)$-invariant second components.
\end{proof}

\begin{proof}[Proof of Proposition \ref{prop:gluing}]
We claim that there is a number $r_1\in (0,r)$, depending only on $F$, such that the following holds.
Assuming that $||F-G||_{0,C\times r\B^m}$  is small enough, there is a map
$\gamma\colon C\times r_1\B^m \to C\times r\B^m$ of class $\Ascr(C\times r_1\B^m)$ and of the form
\eqref{eq:transition}, with $(\Igot,\Igot_0)$-invariant second component, satisfying an estimate 
\begin{equation}\label{eq:est-gamma}
	||\gamma-\Id||_{0,C\times r_1 \B^m} \le const \, \cdotp  ||F-G||_{0,C\times r \B^m}
\end{equation}
(with a constant independent of $G$) and the identity 
\begin{equation}\label{eq:FGgamma}
	F=G\circ \gamma\quad \text{on $C\times r_1\B^m$}. 
\end{equation}
Without the $(\Igot,\Igot_0)$-invariance condition, such {\em transition map} $\gamma$ is furnished 
by \cite[Lemma 5.9.3, p.\ 216]{Forstneric2011-book}. (The proof relies on the implicit function theorem,
using the domination property of the spray $F$; this is why the number $r_1$ may be chosen to depend only on $F$.)
In the case at hand, the cited lemma  furnishes a map 
\[
	\gamma(p,\zeta)=(p,\zeta+c(p,\zeta)),\quad  p\in C'=A\cap B',\  \zeta\in r_1\B^m
\]
satisfying \eqref{eq:FGgamma} on $C'\times r_1\B^m$.
We extend $c$ (and hence $\gamma$) to $\Igot(C')=A\cap \Igot(B')$ by setting
\[
	c(p,\zeta) = \overline{c(\Igot(p),\bar\zeta)},\quad p\in \Igot(C'),\ \zeta\in r_1\B^m.
\]
This implies that the resulting map $c\colon C\times r_1\B^m\to \C^m$ is $(\Igot,\Igot_0)$-invariant.
Since $F$ and $G$ are $(\Igot,\Igot_0)$-invariant, $\gamma$ satisfies condition \eqref{eq:FGgamma} 
also on $\Igot(C')\times r_1\B^m$. 

Pick a number $r_2$ with $0<r_2<r_1$; for example, we may take $r_2=r_1/2$.
Assuming that $||F-G||_{0,C\times r \B^m}$ is small enough, the map $\gamma$ is so close to the 
identity map on $C\times r_1\B^m$ (see \eqref{eq:est-gamma}) that Lemma \ref{lem:splitting} applies. 
Let $\alpha$ and $\beta$ be maps of the form \eqref{eq:alpha-pzeta} satisfying
condition \eqref{eq:split}.  It follows that 
\begin{equation}\label{eq:FaGb}
	F \circ \alpha = G \circ \gamma\circ \alpha = G\circ \beta\quad\text{on}\ C\times r_2\B^n.
\end{equation}
Hence, the sprays $F \circ \alpha$ and $G\circ \beta$ amalgamate into
a spray $H\colon (A\cup B)\times r_2\B^m\to Z$  of class $\Ascr_{\Igot,\Igot_0}(A\cup B)$. Since $\alpha$ is close to 
the identity on $A\times r_2\B^m$ and the composition operator
$\alpha\mapsto F\circ \alpha$ is continuous for $\alpha$ near the identity, 
the map $H|_{A\times r_2\B^m}= F\circ \alpha$ is close to $F$ on $A\times r_2\B^m$.
Hence, Proposition \ref{prop:gluing} holds with $r'=r_2$. Note that, by the construction,
$r'$ depends only $F$.
\end{proof}

%
%
\begin{remark}[On the estimates] \label{rem:estimates}
Note that the sup-norms of the maps $\alpha$ and $\beta$ in \eqref{eq:FaGb}
depend only on the norm of $\gamma$, and hence on $||F-G||_{0,C\times r \B^m}$.
This gives (not necessarily linear) estimates of $||H-F||_{0,A\times r_2\B^m}$ and 
$||H-G||_{0,B\times r_2\B^m}$ in terms of $||F-G||_{0,C\times r \B^m}$
and the sup norms of $F$ and $G$ on their respective domains.
In applications, the first spray $F$ is fixed and hence we can estimate 
$||H-F||_{0,A\times r_2\B^m}$ solely in terms of $||F-G||_{0,C\times r \B^m}$.
However,  the second spray $G$ is typically obtained by Runge approximation of $F$ over $A\cap B$,
and hence one cannot estimate $||H-G||_{0,B\times r_2\B^m}$ solely by $||F-G||_{0,C\times r \B^m}$.
An exception with better estimates will occur in the proof of  Lemma \ref{lem:step1},
due to a special situation which allows approximation of $F$ by $G$ on a bigger set.

We wish to point out that the splitting and gluing results presented above also work 
for sprays of class $\Ascr^k$ for any $k\in\Z_+$, with the same proofs.
\qed\end{remark}

%
%

\section{$\Igot$-invariant period dominating sprays}
\label{sec:period-dominating}

In this section, we construct $(\Igot,\Igot_0)$-invariant period dominating sprays with values
in the punctured null quadric $\Agot_*$  (see Proposition \ref{prop:period-dominating-spray}). 

The following notion will play an important role; cf.\ \cite[Definition 2.2]{AlarconForstnericLopez2016MZ}.

\begin{definition}
\label{def:nonflat}
Let $A$ be a compact connected smoothly bounded domain in an open Riemann surface $\Ncal$. 
A holomorphic map $f\colon A \to \Agot_*$ is {\em flat} if the image $f(A)$ is contained in 
a complex ray $\C\nu$ $(\nu \in \Agot_*)$ of the null cone $\Agot$ \eqref{eq:null-quadric};
otherwise $f$ is {\em nonflat}.
\end{definition}

By \cite[Lemma 2.3]{AlarconForstnericLopez2016MZ}, a holomorphic map $f\colon A \to \Agot_*$ is nonflat
if and only if the tangent spaces $T_{f(p)}\Agot_*$ over all points $p\in A$ span  $\C^n$.
This condition and the following simple lemma are used in the proof of Proposition \ref{prop:period-dominating-spray}.

\begin{lemma}\label{lem:nonflat}
Let $(\Ncal,\Igot)$ be a standard pair \eqref{eq:pair} and let $A$ be a compact $\Igot$-invariant domain  in $\Ncal$.
Every $\Igot$-invariant holomorphic map $f\colon A\to \Agot_*$ is nonflat.
\end{lemma}

\begin{proof}
Since $f$ is $\Igot$-invariant, we have $f(\Igot (p))= \overline{f(p)}$ for all $p\in A$. If $f(A)\subset \C\nu$ for some $\nu\in \Agot_*$, 
then the above condition implies $\bar \nu\in\C \nu$, and hence $\nu=e^{2\imath t}\bar \nu$ for some $t\in\R$. This means that 
the vector $e^{-\imath t}\nu\in \Agot_*$ is real, a contradiction since the punctured null quadric does not contain any real vectors.
\end{proof}

Clearly, a conformal minimal immersion $X\colon A\to \R^n$ is flat, in the sense that 
the image $X(A)$ is contained in an affine $2$-plane in $\R^n$,
if and only if the map $f=2\di X/\theta\colon A\to \Agot_*$ is flat; here 
$\theta$ is any nowhere vanishing holomorphic $1$-form on $\Ncal$. 
(See Section \ref{sec:CMI}.) Note that an immersion from a non-orientable surface
is never flat. Hence, if $X\colon A\to\R^n$ is an $\Igot$-invariant conformal minimal immersion
then $X$ and its derivative $f=2\di X/\theta \colon A\to \Agot_*$ are nonflat.

%
%

\begin{proposition}[Existence of $\Igot$-invariant period dominating sprays]
\label{prop:period-dominating-spray}
Assume that  $(\Ncal,\Igot,\theta)$ is a standard triple \eqref{eq:triple},
$A\subset\Ncal$ is a compact $\Igot$-invariant domain with $\Cscr^1$ boundary in $\Ncal$,
$\Bcal \subset H_1(A;\Z)$ is a homology basis of the form \eqref{eq:basis1} or \eqref{eq:basis2},
and $\Bcal^+=\{\delta_1,\ldots,\delta_l\} \subset \Bcal$ is given by \eqref{eq:B0}. 
Let $\Pcal\colon\Ascr(A,\C^n)\to (\C^n)^l$ denote the period map \eqref{eq:periodmap} 
associated to $\theta$ and $\Bcal^+$.
Given a $\Igot$-invariant  map $f\colon A\to \Agot_*$ of class $\Ascr(A)$, there exists a ball 
$U\subset\C^m$ around $0\in \C^m$ for some $m\in\N$ and a dominating $(\Igot,\Igot_0)$-invariant spray 
$F\colon A\times U\to \Agot_*$ of class $\Ascr(A)$ with core $F(\cdotp,0)=f$ 
which is also {\em period dominating}, in the sense that the differential
\begin{equation}\label{eq:dif-period}
	\di_\zeta\big|_{\zeta=0}  \Pcal(F(\cdotp,\zeta)) : \C^m \lra (\C^n)^l 
\end{equation}
maps $\R^m$ (the real part of $\C^m$) surjectively onto $\R^n\times (\C^n)^{l-1}$. 
\end{proposition}

%
%
\begin{remark}\label{rem:domination}
This notion of period domination is stronger than the one in \cite{AlarconForstneric2014IM} which merely
asks that the differential \eqref{eq:dif-period} be  surjective.
The reason for introducing this stronger condition is that we shall use maps 
$F(\cdotp,\zeta)\colon \Ncal\to \Agot_*$ for real values of the parameter $\zeta$ since these  
are $\Igot$-invariant (see Lemma \ref{lem:composition}).
The first factor on the right hand side is $\R^n$  (and not $\C^n$) since the first curve $\delta_1=\alpha_0$ in 
the homology basis $\Bcal^+$ satisfies $\Igot_*\delta_1=\delta_1$, 
and hence $\int_{\delta_1} \Im \phi=0$ for any $\Igot$-invariant
holomorphic $1$-form $\phi$ (see \eqref{eq:invariance}). 
\qed\end{remark}

\begin{proof}[Proof of Proposition \ref{prop:period-dominating-spray}]
We follow the proof of \cite[Lemma 5.1]{AlarconForstneric2014IM} with certain modifications. 
We shall find a desired spray $F$ of the form \eqref{eq:flow-spray2}:
\begin{equation}\label{eq:flow-spray3}
	F(p,\zeta_1,\ldots,\zeta_m) =
	\phi^1_{\zeta_1h_1(p)}\circ \phi^2_{\zeta_2 h_2(p)}\circ\cdots\circ \phi^m_{\zeta_m h_m(p)} (f(p)),
	\quad p\in A,
\end{equation}
where $\phi^i$ is the flow of some holomorphic vector field $V_i$ of type \eqref{eq:Vjk} on $\C^n$ 
and $\zeta=(\zeta_1,\ldots,\zeta_m)\in \C^m$. Any such spray is $(\Igot,\Igot_0)$-invariant and has range in $\Agot_*$.  

We begin by constructing continuous complex-valued functions $h_1,\ldots,h_m$ (for a suitably big $m\in \N$) on the 
union of the curves $\delta_1,\ldots,\delta_l$ so that the domination condition in 
Proposition \ref{prop:period-dominating-spray} 
holds. (This condition clearly makes sense for sprays $F$ defined over the union of these curves.)
Each of the functions $h_i$ is supported on only one of the curves $\delta_j$, their supports
are pairwise disjoint, and they  do not contain any of the intersection points of these curves.
Such functions can be found by following the proof of \cite[Lemma 5.1]{AlarconForstneric2014IM},
except for the first curve $\delta_1$ where we need to modify the construction as explained below.
In the construction of the $h_i$'s it is crucial that $f$ is nonflat (cf.\ Lemma \ref{lem:nonflat}) 
which ensures that the tangent spaces $T_{f(p)}\Agot_*$ over all point $p\in A$ span $\C^n$; 
see \cite[Lemma 2.3]{AlarconForstnericLopez2016MZ}.

In order to get a $(\Igot,\Igot_0)$-invariant spray, the functions $h_i$ in \eqref{eq:flow-spray3} 
must be $\Igot$-invariant. Recall that $\Igot(\delta_1)=\delta_1$ while $\Igot(\delta_j)\cap\delta_j=\emptyset$
for $j=2,\ldots,l$. Each of the functions $h_i$ is extended to the curves $\Igot(\delta_j)$
for $j=2,\ldots,l$ by setting $h_i(\Igot(p))= \overline{h_i(p)}$ for $p\in \delta_j$.
The first curve $\delta_1=\alpha_0$ is of the form $\delta_1=\sigma\cup \Igot(\sigma)$,
where $\sigma$ is an embedded Jordan arc intersecting $\Igot(\sigma)$ only 
in both endpoints $q$, $\Igot(q)$.  We construct the functions $h_i$ first on $\sigma$, 
satisfying $h_i(\Igot(q))=\overline{h_i(q)}$, and extend them to $\Igot(\sigma)$  by setting 
$h_i(\Igot(p))= \overline{h_i(p)}$ for $p\in \sigma$, so they are  $\Igot$-invariant. 
(This only concerns the functions $h_i$ supported on $\delta_1$.)

Recall that the support $|\Bcal|$ of the union of the curves in $\Bcal$ is Runge in $A$;
hence we can approximate each $h_i$ by a holomorphic function $g_i\in\Oscr(A)$.
The holomorphic function $\tilde h_i = \frac{1}{2}\left(g_i+\overline{g_i\circ \Igot}\right)$ on $A$
is then $\Igot$-invariant and it approximates $h_i$ uniformly on $|\Bcal|$.
Replacing $h_i$ by $\tilde h_i$ in \eqref{eq:flow-spray3} furnishes a $(\Igot,\Igot_0)$-invariant 
holomorphic spray on $A$ with values in $\Agot_*$. If the approximations of $h_i$ by $\tilde h_i$ 
were sufficiently close, then this spray still satisfies the period domination property
in Proposition \ref{prop:period-dominating-spray}. 

Finally, to get the domination property of $F$ we insert additional flow terms into the definition of $F$
\eqref{eq:flow-spray3}. For further details, see \cite[Lemma 5.1]{AlarconForstneric2014IM}.
\end{proof}

%
%

\section{Banach manifold structure of the space $\CMI_\Igot^n(\Ncal)$}
\label{sec:structure}

A {\em compact bordered Riemann surface} is a compact Riemann surface, $\Ncal$, with nonempty boundary 
$b\Ncal\ne\emptyset$ consisting of finitely many Jordan curves. The interior $\mathring \Ncal=\Ncal\setminus b\Ncal$ 
of such an $\Ncal$ is said to be a {\em bordered Riemann surface}. By standard results
(see e.g.\ Stout \cite{Stout1965}), $\Ncal$ embeds as a compact domain with 
smooth real analytic boundary in an open (or closed) Riemann surface $\wt\Ncal$. In particular,
we can compactify $\Ncal$ by gluing a complex disk along each boundary curve of $\Ncal$
so that their complex structures match on the boundary.
Furthermore, a  fixed-point-free antiholomorphic involution $\Igot\colon \Ncal\to \Ncal$ 
extends by the reflection principle to an involution of the same kind on 
a neighborhoood of $\Ncal$ in $\wt \Ncal$.

Given $(\Ncal,\Igot)$ as above and an integer $n\geq 3$, we shall denote by 
\[
	\CMI_\Igot^n(\Ncal)
\]
the set of $\Igot$-invariant conformal immersions $\Ncal\to\r^n$ of class $\Cscr^1(\Ncal)$ 
which are harmonic on the interior $\mathring \Ncal$. (Compare with \eqref{eq:CMI}.)  
With a minor abuse of language, maps $X\in \CMI_\Igot^n(\Ncal)$ will be called 
$\Igot$-invariant conformal minimal immersions on $\Ncal$. Given an integer $r\in\N$, we set 
\[
	\CMI_\Igot^{n,r}(\Ncal)= \CMI_\Igot^n(\Ncal)\cap \Cscr^r(\Ncal,\R^n).
\]
By \cite[Theorem 3.1 (a)]{AlarconForstnericLopez2016MZ}, every $X\in \CMI_\Igot^{n,r}(\Ncal)$
can be approximated in the $\Cscr^r(\Ncal)$ topology by conformal minimal immersions
in a neighborhood of $\Ncal$ in any open Riemann surface $\wt \Ncal$ containing
$\Ncal$ as a smoothly bounded domain. By Corollary \ref{cor:noncritical}, the approximating
immersions may be chosen $\Igot$-invariant.

In this section, we prove the following result.

%
%

\begin{theorem}\label{th:structure}
Let $(\Ncal,\Igot)$ be a compact bordered Riemann surface with a fixed-point-free antiholomorphic involution.
Then, $\CMI_\Igot^{n,r}(\Ncal)$ is a real analytic Banach manifold for any $n\ge 3$ and $r\in\N$.
\end{theorem}

An equivalent formulation of Theorem \ref{th:structure} is the following.

\begin{corollary}\label{cor:structure}
Let $\underline \Ncal$ be a compact non-orientable bordered surface endowed with a conformal structure.
The set of all conformal minimal immersions $\underline \Ncal\to\R^n$ of class $\Cscr^r(\UnNcal,\R^n)$ 
is a real analytic Banach manifold for any $n\ge 3$ and $r\in\n$. 
\end{corollary}

The corresponding result in the orientable case is \cite[Theorem 3.2 (b)]{AlarconForstnericLopez2016MZ};
however, flat immersions were excluded for technical reasons. This is not necessary
here since every immersion of a non-orientable surface to $\R^n$ is nonflat.

Let $\theta\in\Omega_\Igot(\Ncal)$ be a nowhere vanishing $\Igot$-invariant $1$-form on $\Ncal$,
and let $\Pcal$ denote the period map \eqref{eq:periodmap} associated to $\theta$
and a homology basis $\Bcal^+$ (see \eqref{eq:B0}). By the discussion in 
Section \ref{sec:CMI}, we have the following equivalence:
\[
	\text{$X\in \CMI_\Igot^{n,r}(\Ncal)$ if and only if  
	$f:=2\di X/\theta \in \Ascr^{r-1}_\Igot(\Ncal,\Agot_*)$ and $\Re \Pcal(f)=0$.}
\]
We obtain $X$ from $f$ by integration, $X(p)=\int^p \Re(f\theta)$ $(p\in \Ncal)$. 
Assuming as we may that $\Ncal$ is connected, this gives for every $r\in\N$ an isomorphism  
\[
	\CMI_\Igot^{n,r}(\Ncal) \cong \R^n \oplus \Ascr^{r-1}_{\Igot,0}(\Ncal,\Agot_*),
\]
where the first summand accounts for the choice of the initial value at some point $p_0\in\Ncal$
and the second summand equals 
\begin{equation}\label{eq:derivative-space}
	\Ascr^{r-1}_{\Igot,0}(\Ncal,\Agot_*)=\{f\in \Ascr^{r-1}_\Igot(\Ncal,\Agot_*):  \Re \Pcal(f)=0\}.
\end{equation}
This shows that Theorem \ref{th:structure} is an immediate consequence of the following result.

\begin{proposition}\label{prop:Banach-structure}
Let $(\Ncal,\Igot,\theta)$ be a standard triple \eqref{eq:triple}, where $\Ncal$ is a compact bordered 
Riemann surface with smooth boundary. Let $n\ge 3$ and $r\in \Z_+$ be integers, 
and let $\Agot$ denote the quadric \eqref{eq:null} in $\C^n$. Then, the following hold:
\begin{itemize}
\item[\rm (1)] The space $\Ascr^r(\Ncal,\Agot_*)$ is a locally closed direct complex Banach submanifold
of the complex Banach space $\Ascr^r(\Ncal,\C^n)$.
\vspace{1mm}
\item[\rm (2)] The space $\Ascr^r_{\Igot}(\Ncal,\Agot_*)$ is a locally closed direct real analytic Banach submanifold
of the real Banach space $\Ascr^r_{\Igot}(\Ncal,\C^n)$.
\vspace{1mm}
\item[\rm (3)] The space $\Ascr^{r-1}_{\Igot,0}(\Ncal,\Agot_*)$ (see \eqref{eq:derivative-space})  
is a closed real analytic Banach submanifold of finite codimension in the Banach manifold $\Ascr^r_{\Igot}(\Ncal,\Agot_*)$.
\end{itemize}
\end{proposition}

The  notion of a direct Banach submanifold will be explained in the proof. Although 
part (1) is known (see \cite{Forstneric2007AJM}), it is included because we give here 
a different proof which also applies to obtain parts (2) and (3).

\begin{proof}
Consider the map $\Psi : \Ascr^r(\Ncal,\C^n) \to \Ascr^r(\Ncal,\C)=\Ascr^r(\Ncal)$ given by
\[
	\Psi(f)=\Psi(f_1,f_2,\ldots,f_n) =\sum_{j=1}^n (f_j)^2.
\]
Clearly, $\Psi$ is a holomorphic map of complex Banach spaces which takes the closed real subspace
$\Ascr^r_{\Igot}(\Ncal,\C^n)$ of $\Ascr^r(\Ncal,\C^n)$ into the real subspace
$\Ascr^r_{\Igot}(\Ncal,\C)$ of the space  $\Ascr^r(\Ncal,\C)$. Note that 
\begin{eqnarray*}
	\Ascr^r_{\Igot}(\Ncal,\Agot_*)    &=&  \{f\in \Ascr^r_{\Igot}(\Ncal,\C^n \setminus\{0\}) : \Psi(f)=0\}, \\
	\Ascr^r_{\Igot,0}(\Ncal,\Agot_*) &=& \{f\in \Ascr^r_{\Igot,0}(\Ncal,\C^n \setminus\{0\}): \Psi(f)=0\}.
\end{eqnarray*}
Fix $f=(f_1,f_2,\ldots,f_n)\in \Ascr^r(\Ncal,\C^n)$. The differential 
\[
	\Theta:=D\Psi_f\colon \Ascr^r(\Ncal,\C^n) \to \Ascr^r(\Ncal)
\] 
is given on any tangent vector $g=(g_1,\ldots,g_n)\in \Ascr^r(\Ncal,\C^n)$ by
\begin{equation}\label{eq:Theta}
	\Theta(g) = \frac{\di}{\di t}\bigg|_{t=0} \Psi(f+tg) = 2 \sum_{j=1}^n f_j g_j.
\end{equation}
Note that $\Theta(hg)=h\Theta(g)$ for any $h\in   \Ascr^r(\Ncal)$.

Assume now that $f=(f_1,f_2,\ldots,f_n) \in \Ascr^r(\Ncal,\Agot_*)$. Then, the component functions
$f_j \in \Ascr^r(\Ncal)$ have no common zeros on $\Ncal$. We shall need the 
following result. (We thank E.\ L.\ Stout for pointing out the relevant 
references in a private communication.)

%
%
\begin{lemma}\label{lem:maximal-ideals}
Let $\Ncal$ be a compact bordered Riemann surface with smooth boundary,
and let $r\in\Z_+$. Given functions $f_1,\ldots, f_n\in\Ascr^r(\Ncal)$ without 
common zeros, there exist functions $g_1,\ldots, g_n\in\Ascr^r(\Ncal)$ such that
\begin{equation}\label{eq:sumfjgj}
	\sum_{j=1}^n f_j \, g_j =1. 
\end{equation}
If $\Ncal$ admits a fixed-point-free antiholomorphic involution $\Igot$
and $f_1,\ldots, f_n\in\Ascr^r_\Igot(\Ncal)$ have no common zeros, then there exist
$g_1,\ldots, g_n\in\Ascr^r_\Igot(\Ncal)$ satisfying \eqref{eq:sumfjgj}.
\end{lemma}

\begin{proof} 
When $\Ncal=\overline\D$ is the closed unit disk in the plane and $r=0$, 
this is an immediate consequence of the classical result that the maximal ideal space of 
the disk algebra $\Ascr(\D)$ equals $\overline\D$ (see e.g.\ Hoffman \cite[Corollary, p.\ 88]{Hoffman1962-book}). 
Precisely, every maximal ideal in $\Ascr(\D)$ equals $I_p=\{f\in  \Ascr(\D): f(p)=0\}$
for some point $p\in \overline\D$. (For the theory of closed ideals in $\Ascr(\D)$, see Rudin \cite{Rudin1957CJM}.)
The collection of all elements on the left hand side of  the equation \eqref{eq:sumfjgj} over all possible
choices of $g_1,\ldots,g_n\in\Ascr(\D)$ is an ideal of $\Ascr(\D)$. Since the functions $f_j$ 
have no common zeros, this ideal is not contained in any of the maximal ideals $I_p$ for $p\in\overline \D$,
and hence it equals $\Acal(\D)$.

The case of a general bordered Riemann surface $\Ncal$ follows from the disk case
as shown by Arens \cite{Arens1958RCMP}; here is a brief outline (still for the case $r=0$). 
Choose an Ahlfors map $h \colon \Ncal \to \overline\D$ (see Ahlfors \cite{Ahlfors1950CMH}). 
The map $h$ is a proper branched covering, unbranched over the boundary $b\D$. 
With only finitely many exceptions, each fiber $h^{-1}(z)$ for $z\in\overline \D$ has precisely $d$ points 
for some fixed integer $d\in\N$. The algebra $\Ascr(\Ncal)$ is integral over 
$h^*\Acal(\D)$, meaning that each $g\in \Acal(\Ncal)$ satisfies monic polynomial equation 
of degree $d$ with coefficients from $h^*\Acal(\D)$. The homomorphism 
$\lambda \colon \Acal(\D)\to \Acal(\Ncal)$ given by $\lambda(g)=g\circ h$ induces a dual 
map $\lambda^*$ from the maximal ideal space of $\Ascr(\Ncal)$ onto the maximal ideal space 
of $\Acal(\D)$ which equals $\overline\D$. The integrality condition implies that fibers of $\lambda$ contain at 
most $d$ points. The map $\lambda$ is an extension of the Ahlfors map $h$, so for $z\in \overline\D$, 
the fibers $\lambda^{-1}(z)$ and $h^{-1}(z)$ have the same number of points. They therefore coincide,
which shows that the maximal ideal space of $\Acal(\Ncal)$ is $\Ncal$. The existence of a solution
to the equation \eqref{eq:sumfjgj} then follows as in the case of the disk.

Assume now that $r>0$. By Mergelyan's theorem \cite{Mergelyan1951DAN},
$\Ascr^r(\Ncal)\subset \Ascr(\Ncal)$ is dense in $\Ascr(\Ncal)$. 
Given functions $f_1,\ldots, f_n\in\Ascr^r(\Ncal)$ without common zeros,  
choose $g_1,\ldots, g_n\in\Ascr(\Ncal)$ such that \eqref{eq:sumfjgj} holds.
Approximating each $g_j$ sufficiently closely by a function $\tilde g_j \in \Ascr^r(\Ncal)$,
the function $\phi=\sum_{j=1}^n f_j\tilde g_j\in \Ascr^r(\Ncal)$ has no zeros on $\Ncal$, and hence 
the functions $G_j=\tilde g_j/\phi\in  \Ascr^r(\Ncal)$ satisfy $\sum_{j=1}^n f_j G_j=1$.

Assume now that $f_1,\ldots, f_n\in\Ascr^r_\Igot(\Ncal)$.
Let $g_1,\ldots, g_n\in\Ascr^r(\Ncal)$ be chosen such that \eqref{eq:sumfjgj} holds.
Since $f_j=\bar f_j\circ \Igot$, the functions $G_j=\frac{1}{2}\left(g_j+ \bar g_j\circ\Igot \right) \in \Ascr^r_\Igot(\Ncal)$
also satisfy $\sum_{j=1}^n f_j G_j=1$ and we are done.
\end{proof}

We continue with the proof of Proposition \ref{prop:Banach-structure}. Lemma \ref{lem:maximal-ideals}
shows that for every $f=(f_1,\ldots,f_n) \in \Ascr^r(\Ncal,\Agot_*)$ 
there exists $\psi=(\psi_1,\ldots,\psi_n)\in \Ascr^r(\Ncal,\C^n)$ such that $\Theta(\psi) =1$ on $\Ncal$.
It follows that 
\begin{equation}\label{eq:iso2nd}
	\Theta(h \psi)=h\quad \text{for every $h \in  \Ascr^r(\Ncal)$}.
\end{equation}
The set
\[	
	E =  \{\psi g: g\in \Ascr^r(\Ncal)\} \subset \Ascr^r(\Ncal,\C^n)
\]
is a closed subspace of $\Ascr^r(\Ncal,\C^n)$, and we have a direct sum decomposition
\begin{equation}\label{eq:directsum}
	\Ascr^r(\Ncal,\C^n) = (\ker \Theta) \oplus E,\quad g=\left(g- \Theta(g)\psi\right) + \Theta(g)\psi.
\end{equation}
Equation \eqref{eq:iso2nd}  shows that the differential $\Theta= D\Psi_f$ maps $E$ 
isomorphically onto $\Ascr^r(\Ncal)$.
Hence, the implicit function theorem for maps of Banach spaces provides 
an open neighborhood $U$ of $f$ in $\Ascr^r(\Ncal,\C^n\setminus\{0\})$ such that
$U\cap \Ascr^r(\Ncal,\Agot_*)$ is a closed complex Banach submanifold
of $U$ whose tangent space at $f$ equals $\ker\Theta$. (We can locally represent
this submanifold as a graph over $\ker\Theta$.) This proves part (1) of Proposition \ref{prop:Banach-structure}. 

The notion of a {\em direct submanifold} refers to the fact that at every point of the submanifold, 
its tangent space admits a closed direct sum complement in the ambient Banach space.
This always holds for submanifolds of finite codimension, but in our case 
the normal direction corresponds to the infinite dimensional space $E\cong \Ascr^r(\Ncal)$.

Exactly the same proof gives part (2). Indeed, 
for any $f\in \Ascr^r_{\Igot}(\Ncal,\C^n\setminus\{0\})$ the differential $\Theta= D\Psi_f$ (see \eqref{eq:Theta}) 
maps $\Ascr^r_{\Igot}(\Ncal,\C^n)$ to $\Ascr^r_{\Igot}(\Ncal)$. 
Choosing $f\in \Ascr^r_{\Igot}(\Ncal,\Agot_*)$, Lemma \ref{lem:maximal-ideals}
gives  $\psi \in \Ascr^r_\Igot(\Ncal,\C^n\setminus\{0\})$ such that $\Theta(\psi)=1$.  It follows as before 
that $\Ascr^r_\Igot(\Ncal,\C^n) = (\ker \Theta) \oplus E$ as in \eqref{eq:directsum}, where 
\[
	E = \{\psi g : g\in \Ascr^r_\Igot(\Ncal)\} \subset \Ascr^r_\Igot(\Ncal,\C^n).
\]
The rest of the argument is exactly as in part (1).

It remains to prove part (3). Recall from \eqref{eq:derivative-space} that
\[
	\Ascr^{r}_{\Igot,0}(\Ncal,\Agot_*)= \{f\in \Ascr^{r}_\Igot(\Ncal,\Agot_*):  \Re \Pcal(f)=0\}.
\]
Note that the equation $\Re \Pcal(f)=0$ is linear on the Banach space  $\Ascr^{r}_\Igot(\Ncal,\C^n)$, 
and hence real analytic on $\Ascr^{r}_\Igot(\Ncal,\Agot_*)$.
Fix $f\in \Ascr^r_{\Igot,0}(\Ncal,\Agot_*)$. By Proposition \ref{prop:period-dominating-spray},
$f$ is the core $f=F(\cdotp,0)$ of a period dominating $(\Igot,\Igot_0)$-invariant
spray $F\colon \Ncal\times B\to \Agot_*$ for some ball $0\in B\subset\C^N$.
Hence, the equation $\Re \Pcal(f)=0$, defining $\Ascr^{r}_{\Igot,0}(\Ncal,\Agot_*)$ as a subset of 
the real analytic Banach manifold $\Ascr^{r}_\Igot(\Ncal,\Agot_*)$, has maximal rank at $f$. 
Since this amounts to finitely many scalar equations and every closed subspace 
of finite codimension in a Banach space is complemented, the implicit function 
theorem completes the proof. See \cite[proof of Theorem 3.2]{AlarconForstnericLopez2016MZ}
for more details. 
\end{proof}

%
%

\section{Basic approximation results}
\label{sec:basic}

The main result of this section is a basic Runge approximation theorem for
$\Igot$-invariant conformal minimal immersions; see Proposition \ref{prop:noncritical}.
This is the non-critical case in the proof of the Mergelyan theorems \ref{th:Mergelyan} and \ref{th:Mergelyan0}.
Its proof is based on the following lemma concerning the approximation of $\Igot$-invariant 
holomorphic maps into the punctured null quadric $\Agot_*$ \eqref{eq:null-quadric} 
with control of the periods.

%
%

\begin{lemma}[Approximate extension to a bump] \label{lem:bumping}
Let $(\Ncal,\Igot,\theta)$ be a standard triple \eqref{eq:triple},
and let $(A,B)$ be a very special { $\Igot$-invariant} Cartan pair in $\Ncal$ (see Definition \ref{def:Cartan}).
Let $\Pcal$ denote the period map \eqref{eq:periodmap} for the domain $A$.
Then, every map $f\in \Ascr_{\Igot}(A,\Agot_*)$ can be approximated uniformly on $A$ by 
maps $\tilde f\in \Ascr_{\Igot}(A\cup B,\Agot_*)$ satisfying $\Pcal(\tilde f)=\Pcal(f)$.
Furthermore, $\tilde f$ can be chosen such that there exists a homotopy $f_t\in \Ascr_{\Igot}(A,\Agot_*)$  
$(t\in [0,1])$, with $f_0=f$ and $f_1=\tilde f|_A$, such that $f_t$ is uniformly close to $f$ on $A$ and 
$\Pcal(f_t)=\Pcal(f)$ for each $t\in[0,1]$.
\end{lemma}

\begin{proof}
Let $\Bcal^+=\{\delta_1,\ldots,\delta_l\}$ be the period basis \eqref{eq:B0} for the domain $A$. 
Since $A$ contains all nontrivial topology of $A\cup B$, $\Bcal^+$ is then also a period basis for $A\cup B$. 
Let $\Pcal$ denote the associated period map \eqref{eq:periodmap}.

By Proposition \ref{prop:period-dominating-spray} there exist a ball $r\B^m \subset \C^m$
for some $m\in\N$ and $r>0$ and a dominating and period dominating spray 
$F\colon A\times r\B^m \to\Agot_*$ of class $\Ascr_{\Igot,\Igot_0}(A)$ 
with the core $F(\cdotp,0)=f$.  Note that $\Pcal_1(f)=\int_{\delta_1} f\theta \in \R^n$ since $f\theta$ is $\Igot$-invariant. 

By the definition of a very special Cartan pair (see Definition \ref{def:Cartan}), we have
$B=B'\cup \Igot(B')$ where the sets $B'$ and $C'=A\cap B'$ are disks and $B'\cap \Igot(B')=\emptyset$.
Pick a number $r'\in (0,r)$. Since $\Agot_*$ is an Oka manifold, we can 
approximate $F$ as closely as desired uniformly on $C'\times r'\B^m$  by a holomorphic spray 
$G\colon  B'\times r'\B^m\to \Agot_*$ (cf.\ \cite[Theorem 5.4.4, p.\ 193]{Forstneric2011-book}).
We extend $G$ to $\Igot(B')\times r'\B^m$ by setting
\begin{equation}\label{eq:simmetrization}
	G(p,\zeta) = \overline{G\big(\Igot(p),\bar\zeta\big)} \quad 
	\text{ for $p\in \Igot(B')$ and $\zeta\in r'\B^m$.}
\end{equation}
Hence, $G$ is a $(\Igot,\Igot_0)$-invariant holomorphic spray on $B\times r'\B^m$ with values in $\Agot_*$.

By Proposition \ref{prop:gluing} (see also Remark \ref{rem:estimates}) there is a number $r''\in (0,r')$
with the following property: For any spray $G$ as above such that 
$||F-G||_{0,C\times r \B^m}$ is small enough there is a spray 
$H \colon (A\cup B)\times r''\B^m \to \Agot_*$ of class $\Ascr_{\Igot,\Igot_0}(A\cup B)$ 
(obtained by gluing $F$ and $G$), and we can estimate $||H-F||_{0,A\times r''\B^m}$ 
in terms of $||F-G||_{0,C\times r' \B^m}$. 
If $||H-F||_{0,A\times r''\B^m}$ is small enough, the period domination property of $F$ 
gives a point $\zeta_0 \in r''\B^m \cap \R^m$ close to $0$ such that the $\Igot$-invariant map 
\begin{equation}\label{eq:tildef}
	\tilde f = H(\cdotp,\zeta_0) \colon A\cup B\to \Agot_*  
\end{equation}
of class $\Ascr(A\cup B)$ approximates $f$ uniformly on $A$ and satisfies $\Pcal(f)=\Pcal(\tilde f)$.

It remains to explain the existence of a homotopy $f_t$ in the lemma.
Consider the space $\Ascr_{\Igot}(A,r \B^m)$ of all  $\Igot$-invariant maps $\alpha\colon A\to r \B^m$;
this is an open neighborhood of the origin in the real Banach space $\Ascr_{\Igot}(A,\C^m)$. 
Let $F\colon A\times r\B^m \to\Agot_*$ be a $(\Igot,\Igot_0)$-invariant spray as above. 
For every $\alpha\in \Ascr_{\Igot}(A,r\B^m)$ the map
\[
	F_\alpha\colon A\to \Agot_*,\quad F_\alpha(p)=F(p,\alpha(p))\ \ (p\in A)
\]
is $\Igot$-invariant according to Lemma \ref{lem:invariant}, and hence
\[
	\Pcal(F_\alpha) \in \R^n\times (\C^n)^{l-1} \quad \text{for all}\ \  \alpha \in \Ascr_{\Igot}(A,r\B^m).
\]
(The first factor on the right hand side is $\R^n$ and not $\C^n$; cf.\ Remark \ref{rem:domination}.)
Since $\Ascr_{\Igot}(A,r\B^m)$ contains all constant real maps $A\to \zeta \in \R^m\cap r\B^m$,
the period domination property of $F$ (see \eqref{eq:dif-period}) implies that the map
\[
	\Ascr_{\Igot}(A,r \B^m) \ni \alpha \longmapsto \Pcal(F_\alpha) \in \R^n\times (\C^n)^{l-1} 
\]
is submersive at $\alpha=0$. It follows from the implicit function theorem in Banach spaces
that, locally near $\alpha=0$, the set
\begin{equation}\label{eq:constantperiod}
	 \Ascr_{\Igot,f}(A,r \B^m)  = \{\alpha \in \Ascr_{\Igot}(A,r \B^m) : \Pcal(F_\alpha)=\Pcal(F_0)=\Pcal(f)\}
\end{equation}
is a real analytic Banach submanifold of the Banach space $\Ascr_{\Igot}(A,\C^m)$.
More precisely, we can represent it locally as a graph over the kernel of the differential of the period map
at $\alpha=0$; the latter is a closed linear subspace of $\Ascr_{\Igot}(A,\C^m)$ of finite codimension, 
and hence it admits a closed complementary subspace.

By the proof of Proposition \ref{prop:gluing} (see in particular \eqref{eq:FaGb}), 
the spray $H$ obtained by gluing $F$ and $G$ is of the following form over the domain $A$:
\[
	H(p,\zeta) = F(p,\zeta+a(p,\zeta)),\quad p\in A,\ \zeta\in r''\B^m,
\]
where $a\in \Ascr_{\Igot,\Igot_0}(A\times r''\B^m,\C^m)$ is close to the zero map. 
Let $\zeta_0 \in r'' \B^m \cap \R^m$ be the point determining the map $\tilde f$ (see \eqref{eq:tildef}),
and set 
\[	
	\alpha(p)=\zeta_0 + a(p,\zeta_0),\quad  p\in A. 
\]
Since $\Pcal(\tilde f)=\Pcal(f)$ by the choice of $\zeta_0$, the map
$\alpha$ belongs to the space $\Ascr_{\Igot,f}(A, r\B^m)$ \eqref{eq:constantperiod}. 
Assuming as we may that $\alpha_0$ is sufficiently close to the origin 
where this space is a local Banach manifold, we may connect it to the origin by 
a path $\alpha_t\in \Ascr_{\Igot,f}(A,r \B^m)$ $(t\in [0,1])$ such that $\alpha_0=0$ and $\alpha_1=\alpha$. Set
\[
	f_t(p) = F(p,\alpha_t(p)),\quad  p\in A,\ t\in[0,1].
\]
By the construction, the homotopy $f_t\colon A \to \Agot_*$ consists of maps of class 
$\Acal_\Igot(A)$ which are close to $f=f_0$ and satisfy $\Pcal(f_t)=\Pcal(f)$ for all $t\in[0,1]$.
\end{proof}

Let $(\Ncal,\Igot)$ be a standard pair \eqref{eq:pair}, and let 
$\iota \colon \Ncal \to  \UnNcal=\Ncal/\Igot$ denote the quotient projection. By pulling back a strongly subharmonic 
exhaustion function $\underline{\rho}\colon \UnNcal\to\R$, we get a strongly subharmonic exhaustion function 
$\rho=\underline{\rho}\circ\iota\colon \Ncal\to\R$ which is $\Igot$-invariant in the sense that $\rho\circ\Igot=\rho$.

%
%

\begin{proposition}
\label{prop:noncritical}
Let $(\Ncal,\Igot,\theta)$ be a standard triple \eqref{eq:triple}.
Assume that $\rho\colon \Ncal\to \R$ is a smooth $\Igot$-invariant exhaustion function
which has no critical values in $[a,b]$ for some $a<b$. Set $A=\{\rho\le a\}$ and $\wt A=\{\rho\le b\}$. 
Let $n\geq 3$ and let $\Agot_*\subset\c^n$ be the null quadric \eqref{eq:null-quadric0}. 
Then, every map $f\in\Ascr_\Igot(A,\Agot_*)$ can be approximated uniformly on $A$ 
by maps $\tilde f\in\Ascr_\Igot(\wt A,\Agot_*)$ such that $(\tilde f-f)\theta$ is exact on $A$.
Furthermore, $\tilde f$ can be chosen such that there is a homotopy 
$f_t\in\Ascr_\Igot(A,\Agot_*)$ $(t\in [0,1])$ satisfying $f_0=f$, $f_1=\tilde f|_A$, 
and for every $t\in [0,1]$ the difference $(f_t-f)\theta$ is uniformly small and exact on $A$.
\end{proposition}

\begin{proof}
By elementary arguments we can find an increasing sequence of $\Igot$-invariant 
compact sets $A=A_0\subset A_1\subset\cdots\subset A_k=\wt A$ such that for every 
$i=0,\ldots, k-1$ we have $A_{i+1}=A_i\cup B_i$, where $(A_i,B_i)$ is a very special
Cartan pair (see Definition \ref{def:Cartan}). 
We can get such a sequence for example by using \cite[Lemma 5.10.3, p.\ 218]{Forstneric2011-book}
with the function $\underline{\rho}$ in the quotient surface $\UnNcal=\Ncal/\Igot$.
It remains to apply Lemma \ref{lem:bumping} to each pair $(A_i,B_i)$.
\end{proof}

%
%

\begin{corollary}[Mergelyan approximation theorem in the noncritical case]
\label{cor:noncritical}
Let $(\Ncal,\Igot)$ and $A\subset \wt A\subset \Ncal$ be as in Proposition \ref{prop:noncritical}.
Every $\Igot$-invariant conformal minimal immersion $X\colon A \to \R^n$
of class $\Cscr^1(A)$ can be approximated in the $\Cscr^1(A)$ topology by 
$\Igot$-invariant conformal minimal immersions $\wt X\colon \wt A \to \R^n$
satisfying $\Flux_{\wt X}=\Flux_X$. Furthermore, $\wt X$ can be chosen such that
there exists a homotopy of $\Igot$-invariant conformal minimal immersion $X_t\colon A \to \R^n$ 
$(t\in [0,1])$, with $X_0=X$ and $X_1=\wt X$, such that for every $t\in [0,1]$ the 
map $X_t$ is close to $X$ in $\Cscr^1(A,\R^n)$ and satisfies $\Flux_{X_t}=\Flux_X$.
\end{corollary}

\begin{proof}
Let $\theta\in \OmI(\Ncal)$ be a nowhere vanishing $\Igot$-invariant 
$1$-form,  $\Bcal^+ \subset \Bcal$ be a period basis for $H_1(A;\Z)$ (cf.\ \eqref{eq:B0}),
and $\Pcal\colon\Ascr(\Ncal,\C^n)\to (\C^n)^l$ be the period map \eqref{eq:periodmap} 
associated to $\Bcal^+$ and $\theta$. 
The map $f=2\di X/\theta \colon A\to \Agot_*$  is of class $\Ascr_\Igot(A)$.
Proposition \ref{prop:noncritical} furnishes a map $\tilde f\colon \wt A\to \Agot_*$ 
of class $\Ascr_\Igot(\wt A)$ which approximates $f$ as closely as desired uniformly on $A$
and satisfies $\Pcal (\tilde f)=\Pcal(f)$; in particular, $\Re\Pcal(\tilde f)=0$. 
Note that $A$ is a strong deformation retract of $\wt A$ and hence these two domains
have the same topology. 
By Lemma  \ref{lem:partial-period}, the $1$-form $\Re(\tilde f\theta)$ integrates to a  
$\Igot$-invariant conformal minimal immersion $\wt X\colon \wt A\to \R^n$ with the same flux as $X$.
Similarly, for each map in the homotopy 
$f_t\colon A\to \Agot_*$ $(t\in [0,1])$ furnished by Proposition \ref{prop:noncritical}, 
the $1$-form $\Re(\tilde f_t \theta)$ integrates to a $\Igot$-invariant conformal minimal 
immersion $X_t \colon A\to \R^n$ with the same flux as $X$. Since $f_0=f$ and $f_1=\tilde f$, 
a suitable choice of constants of integration ensures that $X_0=X$ and $X_1=\wt X|_A$.
\end{proof}

\begin{corollary}[A local Mergelyan theorem for conformal minimal immersions]
\label{cor:local-Mergelyan}
Let $(\Ncal,\Igot)$ be a standard pair \eqref{eq:pair} and $K$ be a compact $\Igot$-invariant domain 
with $\Cscr^1$ boundary in $\Ncal$. Every $\Igot$-invariant  conformal minimal immersion 
$X\in \CMI^n_\Igot(K)$ $(n\ge 3)$ can be approximated in the $\Cscr^1(K)$-topology by $\Igot$-invariant 
conformal minimal immersions in an open neighborhood of $K$ in $\Ncal$.
\end{corollary}

%
%

\section{The Riemann-Hilbert method for non-orientable minimal surfaces}
\label{sec:RH}

In this section, we prove the following version of the Riemann-Hilbert method for $\Igot$-invariant
conformal minimal immersions of bordered Riemann surfaces. Its main application will be to 
conformal minimal immersions of non-orientable bordered surfaces. 
For background and motivation, see Section \ref{sec:intro-Runge}.

%
%

\begin{theorem} 
\label{th:RH}
Let $(\Ncal,\Igot)$ be a compact bordered Riemann surface with an antiholomorphic involution without fixed points. 
Let $I_1,\Igot(I_1),\ldots,I_k,\Igot(I_k)$ be a finite collection of  pairwise disjoint compact subarcs of $b\Ncal$ 
which are not connected components of $b\Ncal$. Set $I':=\bigcup_{i=1}^k I_i$ and $I:=I'\cup\Igot(I')$. 
Choose a thin annular neighborhood  $R \subset \Ncal$ of $b\Ncal$ (a collar) and a smooth retraction 
$\rho\colon R \to b\Ncal$. Let $n\geq 3$ be an integer. Assume that
\begin{itemize}
\item[\rm (a)]  $X\colon \Ncal \to\r^n$ is a conformal minimal immersion of class $\Cscr^1(\Ncal)$;
\vspace{1mm}
\item[\rm (b)] $r \colon b\Ncal \to \r_+$ is a continuous nonnegative $\Igot$-invariant  function supported on  $I$;
\vspace{1mm}
\item[\rm (c)] $\alpha \colon I \times\overline{\d}\to\r^n$ is a map of class $\Cscr^1$ such that
for every $p\in I$ the map $\cd \ni \zeta \mapsto \alpha(p,\zeta)\in\r^n$ is a conformal minimal 
immersion satisfying $\alpha(p,0)=0$ and $\alpha(\Igot(p),\zeta)=\alpha(p,\zeta)$;
\vspace{1mm}
\item[\rm (d)]  if $n>3$, we assume in addition that $\alpha$ is planar of the form
\[
	\alpha(p,\zeta) = \Re \sigma(p,\zeta) \bu_i+ \Im \sigma(p,\zeta)  \bv_i, \quad
	p \in I_i\cup\Igot(I_i),\;  i\in\{1,\ldots,k\},
\]
where for each $i\in\{1,\ldots,k\}$, 
$\bu_i,\bv_i\in \r^n$ is a pair of orthogonal vectors satisfying 
$\|\bu_i\|=\|\bv_i\|>0$, and $\sigma \colon I \times\overline{\d}\to\c$ is a function of class $\Cscr^1$ such that
for every $p\in I$ the function $\cd\ni \zeta \mapsto \sigma(p,\zeta)$ is holomorphic on $\d$,
$\sigma(p,0)=0$, the partial derivative $\di \sigma(p,\zeta)/\di \zeta$ is nowhere vanishing on $\cd$,
and $\sigma(\Igot(p),\zeta)=\sigma(p,\zeta)$ for all $(p,\zeta)\in I\times\cd$.
\end{itemize}
Consider  the map $\varkappa\colon b\Ncal \times \overline{\d}\to \r^n$ given by
\begin{equation}\label{eq:varkappa}
	\varkappa(p,\zeta)=X(p) +  \alpha(p,r(p)\zeta),
\end{equation}
where we take $\varkappa(p,\zeta)=X(p)$ for $p \in b\Ncal\setminus I$.
Given a number $\epsilon>0$, there exist an arbitrarily small neighborhood $\Omega\subset R$ 
of  $I$  and an $\Igot$-invariant conformal minimal immersion  
$Y\in\CMI_\Igot^n(\Ncal)$ (see \eqref{eq:CMI}) satisfying  the following conditions:
\begin{enumerate}[\it i)]
\item $\dist(Y(p),\varkappa(p,\t))<\epsilon$ for all $p \in b\Ncal$;
\item $\dist(Y(p),\varkappa(\rho(p),\overline{\d}))<\epsilon$ for all $p \in \Omega$;
\item $Y$ is $\epsilon$-close to $X$ in the $\Cscr^1$ norm on $\Ncal\setminus \Omega$;
\item $\Flux_Y=\Flux_X$.
\end{enumerate}
\end{theorem}

Note that the conditions on the vectors $\bu_i,\ \bv_i\in\r^n$ in part (d) imply that 
$\bw_i=\bu_i - \imath \bv_i \in\Agot_*\subset \C^n$ is a null vector for every $i=1,\ldots,k$. 
Hence, the map $\alpha(p,\cdotp)\colon\cd\to \R^n$ in (d) is a planar conformal minimal immersion
for every $p\in I$.

\begin{proof}
This is an adaptation of 
\cite[Theorem 3.2]{AlarconDrinovecForstnericLopez2015AJM} (for $n=3)$ and of 
\cite[Theorem 3.6]{AlarconDrinovecForstnericLopez2015PLMS} (for any $n\ge 3$) 
to the case of $\Igot$-invariant maps. The latter result is essentially a consequence of 
\cite[Theorem 3.5]{AlarconDrinovecForstnericLopez2015PLMS} which gives the 
Riemann-Hilbert method for null holomorphic immersions $\Ncal\to\C^n$. This method was first developed 
for null curves in dimension $n=3$ by Alarc\'on and Forstneri\v c in \cite[Theorem 4]{AlarconForstneric2015MA}. 

It suffices to explain the proof for $k=1$, i.e., when $I=I' \cup \Igot(I')$ is the union of a pair of
disjoint $\Igot$-symmetric arcs; the general case follows by applying this special case 
to each pair $(I_i,\Igot(I_i))$ for $i=1,\ldots, k$.

Pick an $\Igot$-invariant holomorphic $1$-form $\theta$ without zeros on $\Ncal$ (cf.\ Proposition \ref{prop:Cor6.5}). 
Let $\Bcal \subset H_1(\Ncal;\Z)$ be a homology basis of the form \eqref{eq:basis1} or \eqref{eq:basis2}
(depending on the parity of the genus of $\Ncal$), 
with support $|\Bcal|\subset\mathring\Ncal$,
and let $\Bcal^+=\{\delta_1,\ldots,\delta_l\} \subset \Bcal$ be given by \eqref{eq:B0}. 
Let $\Pcal\colon\Ascr(\Ncal,\C^n)\to (\C^n)^l$ denote the period map \eqref{eq:periodmap} 
associated to $\Bcal^+$ and $\theta$. Then, $2\di X=f\theta$ where $f\colon \Ncal\to \Agot_*$ 
is a nonflat map of class $\Ascr_{\Igot}(\Ncal,\Agot_*)$ satisfying $\Re\Pcal(f)=0$. 

We can write $\Ncal=A\cup B$ where $(A,B)$ is a very special Cartan pair (see Definition
\ref{def:Cartan}) such that $|\Bcal^+|\subset A$ and the set $B=B'\cup \Igot(B') \subset R$ 
is the disjoint union of a pair of small closed disks $B'$ (a neighborhood of $I'$) and  
$\Igot(B')$ (a neighborhood of $\Igot(I')$). Set $C'=A\cap B'$ and $C=A\cap B = C'\cup \Igot(C')$. 

By Proposition \ref{prop:period-dominating-spray},
there exist a ball $U \subset\C^m$ around $0\in \C^m$ for some $m\in\N$ 
and a dominating and period dominating $(\Igot,\Igot_0)$-invariant spray 
$F\colon \Ncal \times U\to \Agot_*$ of class $\Ascr(\Ncal)$ (cf.\ \eqref{eq:dif-period}).
Pick a point $p_0\in C'=A\cap B'$. 
After shrinking $U$ slightly around $0$ if necessary, we can 
apply \cite[Lemma 3.1]{AlarconDrinovecForstnericLopez2015PLMS}
(when $n=3$) or \cite[Lemma 3.3]{AlarconDrinovecForstnericLopez2015PLMS}
(when $n\ge 3$ and $\alpha$ satisfies Condition (d) in the theorem)
to approximate the spray $F$ as closely as desired uniformly on $C'\times U$ 
by a spray $G\colon B' \times U\to \Agot_*$ such that for every fixed value
of $\zeta\in U$, the integral
\[
	X'(p,\zeta) = X(p_0) + \Re \int_{p_0}^p G(\cdotp, \zeta)\theta,\quad p\in B' ,
\]
is a conformal minimal immersion $B'\to \R^n$, and $X'(\cdotp,0)$ satisfies 
Conditions (i)--(iii) in Theorem \ref{th:RH} with $\Ncal$ replaced by $B'$. 
(The integral  is independent of the choice of path in the simply connected domain $B'$.)
A spray $G$ with these properties is obtained by deforming $F$ near each point $p\in I$ 
in the direction of the null holomorphic disk in $\C^n$ whose real part equals
 $r(p)\alpha(p,\cdotp)\colon \cd\to \C^n$. 

We extend $G$ to $\Igot(B')\times U$ by simmetrization (cf.\ \eqref{eq:simmetrization}) 
so that the resulting spray $G\colon B \times U\to\Agot_*$ is $(\Igot,\Igot_0)$-invariant. 
Since the spray $F$ is $(\Igot,\Igot_0)$-invariant, it follows that $G$ approximates 
$F$ also on $\Igot(C')\times U$.  
Assuming that the approximation of $F$ by $G$ is close enough on $C\times U$,
we can apply Proposition \ref{prop:gluing} to glue the sprays $F|_{A\times U}$ and $G$ into a 
$(\Igot,\Igot_0)$-invariant spray $H\colon \Ncal\times U' \to \Agot_*$ of class $\Acal(\Ncal\times U')$,
where $U' \subset \C^m$ is a smaller neighborhood of the origin, such that
$H|_{A\times U'}$ is uniformly close to  $F|_{A\times U'}$. If the approximation of $F$ by $H$
is close enough, 
the period domination property of $F$ implies that there exists a parameter value $\zeta_0\in U'$ close to $0$ such that the map
$h=H(\cdotp,\zeta_0) \colon \Ncal\to \Agot_*$ satisfies $\Pcal(h)=\Pcal(f)$; 
in particular, $\Re \Pcal(h)=0$. Fix a point $p_0\in \Ncal$. The map $Y\colon \Ncal\to \c^n$ given by 
\[
	Y(p) = X(p_0) + \Re \int_{p_0}^p  h \theta,\quad p \in \Ncal,
\]
is then a well defined $\Igot$-invariant conformal minimal immersion satisfying the conclusion 
of Theorem \ref{th:RH} provided that all approximations were sufficiently close.
In particular, from $\Im \Pcal(h)=\Im\Pcal(f)$ it follows that $\Flux_Y=\Flux_X$, so Condition (iv) holds.
For further details, we refer to \cite[proof of Theorems 3.5 and 3.6]{AlarconDrinovecForstnericLopez2015PLMS}
or \cite[proof of Theorem 4]{AlarconForstneric2015MA}.
\end{proof}

%
%

\chapter{Approximation theorems for non-orientable minimal surfaces}
\label{ch:Runge}

The aim of this chapter is to provide general Runge-Mergelyan type theorems for non-orientable conformal minimal 
immersions into $\r^n$ for any $n\geq 3$. The main results are Theorems \ref{th:Mergelyan}, \ref{th:Mergelyan0}, 
\ref{th:Mergelyan1}, and \ref{th:Mergelyan2}; these include Theorem \ref{th:intro-Runge}
stated in the Introduction as a special case.  For the corresponding approximation results in the orientable case, 
see \cite{AlarconForstnericLopez2016MZ} and the references therein.

We assume that $(\Ncal,\Igot)$ is a standard pair \eqref{eq:pair}, that is,
an open Riemann surface $\Ncal$ with a fixed-point-free  antiholomorphic involution $\Igot$.

We begin by introducing suitable subsets and maps for Mergelyan approximation;
compare with \cite[Definition 5.1]{AlarconForstnericLopez2016MZ} in the orientable case.

%
%

\begin{definition}\label{def:admissible}
A nonempty compact subset $S$ of $\Ncal$ is {\em $\Igot$-admissible} if it is Runge in $\Ncal$, 
$\Igot$-invariant ($\Igot(S)=S$), and of the form $S=K\cup \Gamma$, where $K=\overline{\mathring S}$ 
is a union of finitely many pairwise disjoint  smoothly bounded compact domains in $\Ncal$ 
and $\Gamma=\overline{S\setminus K}=\Gamma'\cup\Igot(\Gamma')$, where $\Gamma'\cap\Igot(\Gamma')=\emptyset$ 
and $\Gamma'$ is a a finite union of pairwise disjoint smooth closed curves disjoint from $K$ and smooth Jordan arcs
meeting $K$ only in their endpoints (or not at all) such that their intersections with the boundary 
$bK$ of $K$ are transverse.
\end{definition}

Given a complex submanifold $Z$ of $\C^n$ which is invariant under conjugation $\Igot_0(z)= \bar z$
(note that the null quadric  $\Agot_*$ \eqref{eq:null} is such),  we denote by 
\begin{equation}\label{eq:Fscr}
	\Fscr_\Igot(S,Z)
\end{equation}
the set of all smooth $\Igot$-invariant functions $f\colon S\to Z$, i.e. $f\circ \Igot=\overline f$, 
that are holomorphic on an  open neighborhood of $K$ depending on the map. We  write 
\[	
	\Fscr_\Igot(S)=\Fscr_\Igot(S,\c).
\]

%
%

\begin{definition}\label{def:GCMI}
Let $S=K\cup\Gamma \subset \Ncal$ be an $\Igot$-admissible set, and let $n\ge 3$. 
A {\em generalized $\Igot$-invariant conformal minimal immersion} $S\to\r^n$ is a pair $(X,f\theta)$, 
where $\theta\in\Omega_\Igot(\Ncal)$ has no zeros (hence $(\Ncal,\Igot,\theta)$ is a standard triple \eqref{eq:triple}), 
$f\in\Fscr_\Igot(S,\Agot_*)$ (cf.\ \eqref{eq:Fscr}) and $X\colon S\to\r^n$ is a smooth $\Igot$-invariant map 
such that
\begin{equation}\label{eq:XpXq}
	X(p)=X(q)+\int_{q}^p \Re(f\theta)
\end{equation}
holds for all pairs of points $p,q$ in the same connected component of $S$ and for any choice of path of integration 
in $S$ from $p$ to $q$
%
%
(i.e., $f\theta$ has no real periods).

The group homomorphism 
\begin{equation}\label{eq:GFlux}
	\Flux_{(X,f \theta)}\colon H_1(S;\z)\to \r^n, \quad  
	\Flux_{(X,f \theta)}(\gamma)=\int_\gamma \Im (f \theta)
\end{equation}
is called the {\em flux map} of $(X,f \theta)$.
\end{definition}

Except for  $\Igot$-invariance, Definition \ref{def:GCMI} coincides with 
\cite[Definition 5.2]{AlarconForstnericLopez2016MZ}. 
It implies that $X$ extends to a conformal minimal immersion in an open neighborhood of $K$
such that $2\di X=f\theta$ there, and for any smooth path $\alpha$ parametrizing a curve in $\Gamma$ we have
$d(X\circ \alpha)=\Re(\alpha^*(f\theta))$.

We denote by 
\begin{equation}\label{eq:GCMI}
	 \GCMII^n(S)
\end{equation} 
the set of all generalized $\Igot$-invariant conformal minimal immersions $S\to\r^n$.

%
%

We shall use the following notion of $\Cscr^1$ approximation of generalized conformal minimal immersions.

\begin{definition}\label{def:appGCMI}
Let $S\subset\Ncal$ be a $\Igot$-admissible set, and let $(X,f\theta)\in \GCMII^n(S)$ for some $n\ge 3$
(see \eqref{eq:GCMI}). We say that {\em $(X,f\theta)$ may be approximated in the $\Cscr^1(S)$-topology
by $\Igot$-invariant conformal minimal immersions on $\Ncal$} if the map $X\colon S\to\R^n$ may be approximated 
as closely as desired in the $\Cscr^1(S)$-topology by $\Igot$-invariant conformal minimal immersions 
$Y\colon\Ncal\to\r^n$ such that the function $2\partial Y/\theta\colon S\to \C^n$ approximates $f$ uniformly on $S$.
\end{definition}

We add a comment about this approximation condition. Let $S=K\cup \Gamma$ be an admissible set and $(X,f\theta)\in \GCMII^n(S)$. 
If a conformal minimal immersion $Y\colon \Ncal \to\R^n$ approximates $X$ in $\Cscr^1(K)$, then $2\di Y$ approximates 
$2\di X=f\theta$ uniformly on $K$, so $2\di Y/\theta$ approximates $f$. On the other hand, on the curves in $\Gamma$ the above
approximation condition implies that the full $1$-jet of $Y$ along $\Gamma$ (also in the direction
transverse to $\Gamma$) uniformly approximates the $1$-jet determined by $f\theta$. 

%
%

\section{A Mergelyan approximation theorem} \label{sec:Mergelyan0}

The following is the first main result of this section; this is a more precise version of 
Theorem \ref{th:intro-Runge} stated in the Introduction.

%
%
\begin{theorem}[Mergelyan theorem for conformal non-orientable minimal surfaces in $\R^n$]\label{th:Mergelyan}
Let $(\Ncal,\Igot,\theta)$ be a standard triple \eqref{eq:triple}), 
and let $S=K\cup\Gamma\subset\Ncal$ be a $\Igot$-admissible set (cf.\ Definition \ref{def:GCMI}). 
Assume that $(X,f\theta)\in \GCMII^n(S)$ for some $n\ge 3$ (see \eqref{eq:GCMI}), and let $\pgot\colon H_1(\Ncal;\z)\to \r^n$ 
be a group homomorphism such that $\pgot\circ \Igot_*=-\pgot$ and $\pgot|_{H_1(S;\z)}=\Flux_{(X,f \theta)}$
(cf.\ \eqref{eq:GFlux}). Then, $(X,f \theta)$ may be approximated in the $\Cscr^1(S)$-topology by $\Igot$-invariant 
conformal minimal immersions  $Y\colon \Ncal\to\r^n$ with $\Flux_Y=\pgot$. 
\end{theorem}

By passing to the quotient $\UnNcal=\Ncal/\Igot$, Theorem \ref{th:Mergelyan}
gives the corresponding Mergelyan approximation theorem for conformal minimal
immersions from non-orientable surfaces; see Corollary \ref{cor:intro-Runge}.
The case $n=3$ is proved in \cite{AlarconLopez2015GT} by using the L\'opez-Ros transformation 
for minimal surfaces in $\r^3$,  a tool which is not available in higher dimensions. 

Theorem \ref{th:Mergelyan} follows easily  from the following result which pertains to derivatives
of generalized $\Igot$-invariant conformal minimal immersions.

%
%
%
%
\begin{theorem} \label{th:Mergelyan0}
Let $(\Ncal,\Igot,\theta)$ be a standard triple \eqref{eq:triple},
let $S=K\cup\Gamma\subset\Ncal$ be a $\Igot$-admissible set (cf.\ Definition \ref{def:GCMI}),
let $f\colon S\to \Agot_*\subset \c^n$ ($n\geq 3$) be a map in $\Fscr_\Igot(S, \Agot_*)$ (see \eqref{eq:Fscr}), 
and let $\qgot\colon H_1(\Ncal;\z)\to\c^n$ be a group homomorphism such that 
\[
	\qgot\circ \Igot_*=\overline\qgot \quad \text{and} \quad 
	\qgot(\gamma)=\int_\gamma f\theta\ \ \text{for every $\gamma \in H_1(S;\z)$}.
\]
Then, $f$ may be approximated uniformly on $S$ by maps $h\in \Oscr_\Igot(\Ncal,\Agot_*)$ such that 
$\int_\gamma h\theta=\qgot(\gamma)$ for all $\gamma \in H_1(\Ncal;\z)$.
Furthermore, $h$ can be chosen such that there exists a homotopy $f_t\in \Fscr_\Igot(S, \Agot_*)$ $(t\in [0,1])$ 
from $f_0=f$ to $f_1=h|_S$ such that for every $t\in [0,1]$, $f_t$ is holomorphic in a neighborhood of $K$
(independent of $t$), uniformly close to $f_0=f$ on $S$ (also uniformly in $t\in [0,1]$), and 
\begin{equation}\label{eq:qgot}
	\qgot(\gamma)=\int_\gamma f_t\theta\quad  \text{for every $\gamma \in H_1(S;\z)$}.
\end{equation}
\end{theorem}

Note that \eqref{eq:qgot} implies that the $1$-form $(f-f_t)\theta$ is exact on $S$ for every $t\in [0,1]$;
in particular, $(f-h)\theta$ is exact on $S$.

\begin{proof}[Proof of Theorem \ref{th:Mergelyan} assuming Theorem \ref{th:Mergelyan0}]
We may clearly assume that $\Ncal$ is connected. 

We begin by reducing the problem to the case 
when  $\Gamma=\overline{S\setminus K}$ does not contain any closed curves. 
Indeed, for every closed curve $\gamma$ in $\Gamma$ (and its symmetric image $\Igot(\gamma)$) 
we add to $K$ a pair $D\cup \Igot(D)$, where $D\subset \Ncal\setminus K$
is a small closed disk  centered at a point of $\gamma$ such that $D\cap \Igot(D)=\emptyset$
and $D\cap \Gamma= D\cap \gamma$.  
Let $K'$ be the new compact set obtained by adding to $K$ pairs of such disks for all closed curves in $\Gamma$.
Let $\Gamma'=\overline{\Gamma \setminus K'}$ and $S'=K'\cup \Gamma'$. Then, 
$S'$ is an $\Igot$-admissible set and $\Gamma'$ does not contain any closed curves. 
Futhermore, choosing the disks $D$ as above small enough, we can approximate 
$(X,f\theta)\in \GCMII^n(S)$ in the $\Cscr^1(S)$ topology by $(X',f'\theta)\in \GCMII^n(S')$ simply by choosing 
$X'$ to be a conformal minimal immersion on a small neighborhood of each of the new disks $D$ which 
approximates $X$ in the $\Cscr^1$ topology on the arc $D\cap \Gamma$ and 
suitably redefining the map $f$. 

Consider first the case when $S$ is connected. 
Let $\qgot\colon H_1(\Ncal;\z)\to \imath\, \r^n\subset \c^n$ be the group homomorphism given by 
$\qgot(\gamma)=\imath\, \pgot(\gamma)$ for all $\gamma\in H_1(\Ncal;\z)$. 
Clearly, we have that $\qgot\circ \Igot_*=\overline\qgot$ and $\qgot(\gamma)=\int_\gamma f\theta$ for every 
closed curve $\gamma\subset S$. By Theorem \ref{th:Mergelyan0} we may approximate $f$ uniformly 
on $S$ by maps $h\in \Oscr_\Igot(S,\Agot_*)$ satisfying the condition $\int_\gamma h\theta=\imath\, \pgot(\gamma)$ 
for all $\gamma \in H_1(\Ncal;\z)$; in particular $\Re(h\theta)$ is exact. Fix a point $p_0\in S$. The $\Igot$-invariant 
conformal minimal immersions $Y\colon \Ncal\to\r^n$ given by 
$Y(p)=X(p_0)+\int_{p_0}^p \Re(h\theta)$ satisfies the conclusion of the theorem.
Indeed, the integral is independent of the choice of path of integration from $p_0$ to $p$,  and for points $p\in S$ 
we may choose the path to lie entirely within the connected set $S$.

The case when $S$ is not connected can be reduced to the former case as follows.
Choose a connected $\Igot$-admissible set $\wt S=K\cup\wt \Gamma$ in $\Ncal$ such that 
$\Gamma\subset\wt \Gamma$; i.e., $\wt S$ is obtained by adding finitely many smooth Jordan arcs to $S$.
(See \cite[Claim 5.7]{AlarconLopez2015GT} for this argument. It is important that $\Gamma$ does not
contain any closed curves, since the added arcs must be attached with their endpoints to $bK\setminus \Gamma$.)
 Consider any extension $\wt f\in \Fscr_\Igot(\wt S, \Agot_*)$ 
of $f\in\Fscr_\Igot(S, \Agot_*)$ such that $X(p)=X(q)+\int_{q}^p \Re(\wt f\theta)$ for all $p,q\in S$ and 
$\int_\gamma \wt f\theta=\imath \pgot(\gamma)$ for every closed curve $\gamma\subset \wt S$. 
To find a smooth map $\wt f$ satisfying these properties, it suffices to define it in a suitable way 
on the arcs $\wt \Gamma\setminus\Gamma$; this can be done as in \cite[Theorem 5.6]{AlarconLopez2015GT} 
where the details are given for $n=3$, or by symmetrizing the argument in 
\cite[Lemma 5.5]{AlarconForstnericLopez2016MZ}. 
Fix a point $p_0\in S$ and set $\wt X(p)=X(p_0)+\int_{p_0}^p \Re(\wt f\theta)$ 
for all $p\in \wt S$. It follows that $(\wt X,\wt f\theta)\in \GCMII^n(\wt S)$, $(\wt X,\wt f\theta)|_S=(X,f\theta)$, 
and $\Flux_{(\wt X,\wt f \theta)}=\pgot|_{H_1(\wt S;\z)}$. It remains to apply the previous case to 
$(\wt X,\wt f\theta)$ on the connected $\Igot$-admissible set $\wt S$.
\end{proof}

\begin{proof}[Proof of Theorem \ref{th:Mergelyan0}]
%
%
%
By the argument in the proof of Theorem \ref{th:Mergelyan} we may assume that the $\Igot$-invariant set 
$S=K\cup \Gamma$ is connected and $\Gamma$ does not contain any closed curves. 

Since $S$ is Runge in $\Ncal$, there exists a smooth  $\Igot$-invariant strongly subharmonic
Morse exhaustion function $\rho \colon \Ncal \to \r$  
such that $0$ is a regular value of $\rho$,  $S\subset \{\rho<0\}$, and $S$ is a strong deformation retract of 
the set $M_1:=\{\rho\leq 0\}$. (For the existence of such a function $\rho$ see the short discussion preceding 
Proposition \ref{prop:noncritical}.) Thus, the inclusion map $S\hookrightarrow M_1$ induces an isomorphism
$H_1(S;\z)\cong H_1(M_1;\z)$. 

Let $p_1, \Igot(p_1), p_2, \Igot(p_2),\ldots$ be the critical points 
of $\rho$ in $\Ncal \setminus M_1$; recall that $\rho\circ \Igot=\rho$. We may assume that $0 < \rho (p_1) < \rho (p_2) < \cdots$.
Choose an increasing divergent sequence $c_1 < c_2 < c_3 < \cdots$ such that $\rho(p_i) < c_i <  \rho(p_{i+1})$ holds 
for $i = 1, 2, \ldots$. If there are only finitely many $p_i$'s, then we choose the remainder of the
sequence $c_i$ arbitrarily subject to the condition $\lim_{i\to\infty} c_i =+\infty$. Set $M_i:= \{\rho \leq c_i\}$. It follows that 
\[
	S\Subset M_1\Subset M_2\Subset\cdots\Subset \bigcup_{i=1}^\infty M_i=\Ncal
\] 
and each $M_i$ is a (not necessarily connected) compact $\Igot$-invariant Runge domain in $\Ncal$
whose connected components are compact bordered Riemann surfaces. 

We shall inductively construct a sequence of maps $f_j\in \Oscr_\Igot(M_j,\Agot_*)$ 
which converges uniformly on compacts to a map $h\in \Oscr_\Igot(\Ncal,\Agot_*)$
approximating $f$ uniformly on $S$ and satisfying $\int_\gamma h\theta=\qgot(\gamma)$ 
for all $\gamma \in H_1(\Ncal;\z)$. 

The basis of the induction is given by the following lemma which will be applied to the domain $M=M_1$, and the resulting 
map $\tilde f$ will be taken as the first element $f_1$ in the sequence. The inductive step is provided by Lemma \ref{lem:step2}.

%
%
%
%
\begin{lemma}\label{lem:step1}
Assume that  $S$ is an $\Igot$-admissible set in $\Ncal$ (see Definition \ref{def:admissible}) and
$M\subset \Ncal$ is a compact smoothly bounded $\Igot$-invariant domain with $S\subset \mathring M$
such that $S$ is strong deformation retract of $M$. Then, any map $f\in \Fscr_\Igot(S, \Agot_*)$ 
can be approximated uniformly on $S$ by maps $\tilde f\in \Oscr_\Igot(M,\Agot_*)$ 
such that the $1$-form $(f-\tilde f)\theta$ is exact on $S$, and there is a homotopy $f_t\in \Fscr_\Igot(S, \Agot_*)$ 
with $f_0=f$ and $f_1=\tilde f|_{S}$ such that for every $t\in [0,1]$
the map $f_t$ is holomorphic on a neighborhood of $K$ (independent of $t$), 
arbitrarily uniformly close to $f$, and it satisfies \eqref{eq:qgot}.
\end{lemma}

\begin{proof}
We begin by considering two special cases.

\noindent{\em Case 1.} 
Assume that $S=C\cup \Igot(C)$, where $C$ is a connected component of $S$ and $C\cap \Igot(C)=\emptyset$.

Let $E$ be the connected component of $M$ containing $C$; then $E\cap \Igot(E)=\emptyset$ and 
$M=E\cup \Igot(E)$. The proof of \cite[Theorem 5.3]{AlarconForstnericLopez2016MZ} shows that $f|_C$ can be 
approximated uniformly on $C$ by holomorphic maps $g\colon E\to \Agot_*$ such that  $(f-g)\theta$ is exact on $C$. 
Setting $\tilde f=g$ on $E$ and $\tilde f=\overline g \circ \Igot$ on $\Igot(E)$ we are done.
 
\vspace{1mm}
 
\noindent{\em Case 2.} The $\Igot$-admissible set $S=K\cup \Gamma$ is connected,
and hence $K\neq \emptyset$. 

If $\Gamma=\emptyset$, it suffices to apply Proposition \ref{prop:noncritical} to the pair of domains $K\subset M$.

Assume now that $\Gamma\neq \emptyset$. By the definition of an admissible set, 
we have $\Gamma=\Gamma'\cup \Igot(\Gamma')$ where 
$\Gamma'$ is union of connected components  of $\Gamma$ and $\Gamma'\cap \Igot(\Gamma')=\emptyset$. 
Take a small compact neighborhood $B'$ of $\Gamma'$ such 
that $\Gamma'$ is a strong deformation retract of $B'$, $B'\cap \Igot(B')=\emptyset$ and, 
setting $B=B'\cup \Igot(B')$, $(K,B)$ is a special $\Igot$-invariant Cartan pair
(see Definition \ref{def:Cartan}). In particular, $K\cup B$ is a strong deformation retract of $M$ 
and each connected component of $B$ is a closed disk enclosing a unique connected component 
of $\Gamma$. See Figure \ref{fig:KB}.
\begin{figure}[ht]
    \begin{center}
    \scalebox{0.2}{\includegraphics{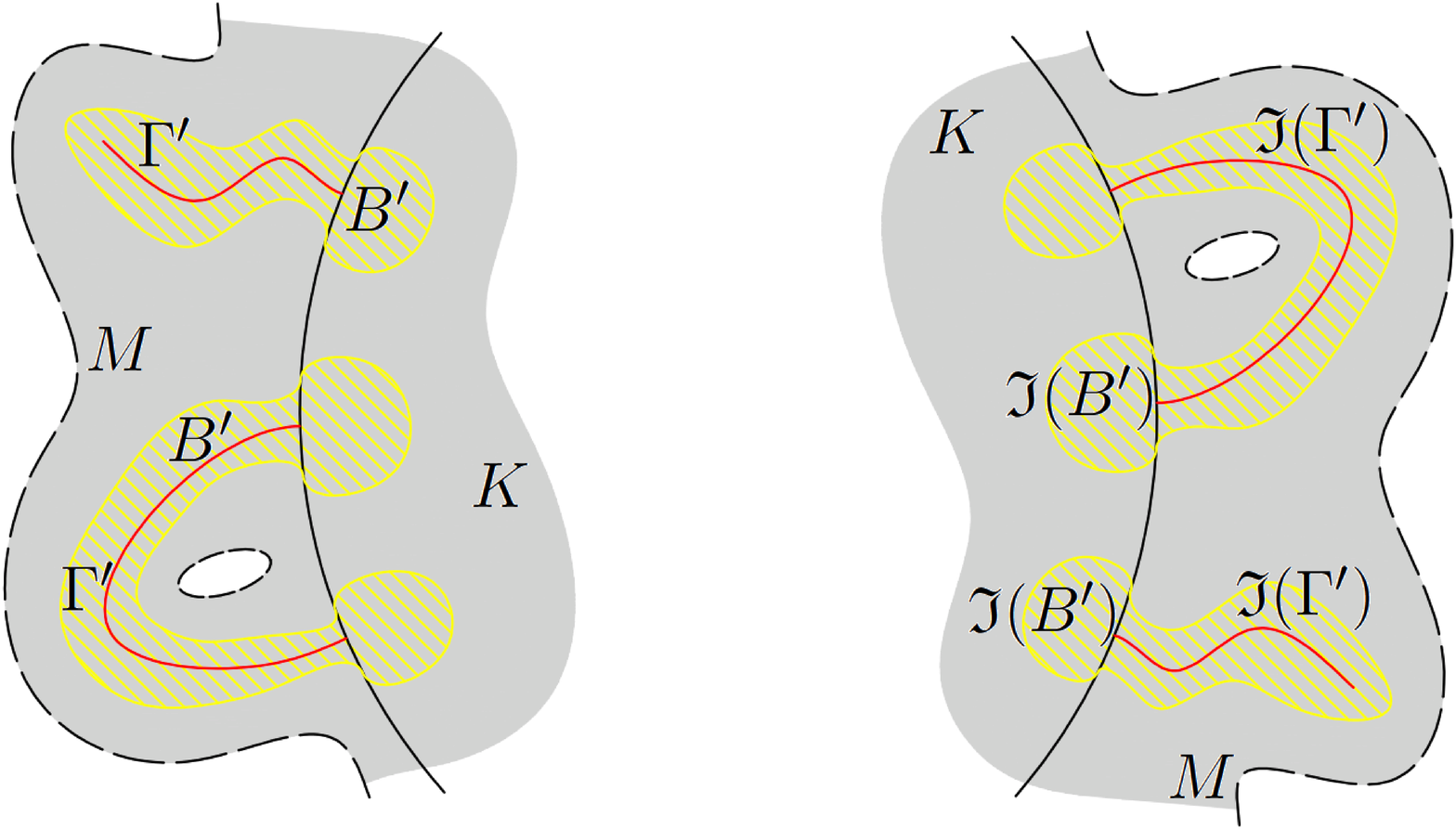}}
        \end{center}
\caption{The Cartan pair $(K,B)$.}\label{fig:KB}
\end{figure}

Fix a number $\epsilon>0$. We wish to find $\tilde f\in \Oscr_\Igot(M,\Agot_*)$ such that $||f-\tilde f||_{0,S}<\epsilon$ 
and $(f-\tilde f)\theta$ is exact on $S$; furthermore, there is a homotopy $f_t\in \Oscr_\Igot(M,\Agot_*)$ $(t\in [0,1])$ from $f_0=f$ 
to $f_1=\tilde f$ with the same properties.
It suffices to explain how to find a map $\tilde f\in \Oscr_\Igot(K\cup B,\Agot_*)$ with the stated properties
since we can extend the solution (by approximation) from $K\cup B$ to $M$ by Proposition \ref{prop:noncritical}.

Reasoning as in the proof of Proposition \ref{prop:period-dominating-spray}, we construct a dominating and 
period dominating $(\Igot,\Igot_0)$-invariant holomorphic spray $F\colon S\times r\b^m\to \Agot_*$ 
of the form \eqref{eq:flow-spray3} with the core $F(\cdotp,0)=f$ for some $r>0$ and $m\in\n$. 
The period map $\Pcal$, to which the period domination of the spray $F$ refers to, is associated 
to the $1$-form $\theta$ and a $\Igot$-basis $\Bcal^+$ of $H_1(S;\z)$; cf.\ \eqref{eq:periodmap}.
Decreasing $r>0$ if necessary, we may assume that 
\begin{equation}\label{eq:Ff}
	||F(\cdot,\zeta)-f||_{0,S}<\epsilon/2\quad \text{for all $\zeta\in r\b^m$}.
\end{equation}

We claim that for every $\delta>0$  there exists a spray $G \colon B \times r\B^m\to \Agot_*$ of class 
$\Ascr_{\Igot,\Igot_0}(B)$ satisfying
\begin{equation}\label{eq:FG}
	||F-G||_{0,(S\cap B)\times r\b^m} < \delta.
\end{equation}
To see this, note that the pair $S\cap B' \subset B'$
satisfies the classical Mergelyan approximation theorem, i.e., every continuous function 
on $S\cap B'$ that is holomorphic on $\mathring S\cap B'= \mathring K\cap B'$
is a uniform limit of functions holomorphic on $B'$. Furthermore, since  $B'$
deformation retracts onto $S\cap B'$, the same holds for maps to any Oka manifold,
in particular, to the null quadric $\Agot_*$. (This follows by combining the
local Mergelyan theorem for maps to an arbitrary complex manifold,
cf.\ \cite[Theorem 3.7.2]{Forstneric2011-book}, and the Runge approximation theorem 
for maps to Oka manifolds, cf.\ \cite[Theorem 5.4.4]{Forstneric2011-book}.) 
In the case at hand, we can get approximating sprays $G$ simply by approximating
the core map $f$ and the coefficient functions $h_i$ in \eqref{eq:periodmap}
by holomorphic maps on $B'$. 
We extend $G$ to $\Igot(B')\times r\B^m$ by simmetrization (cf.\ \eqref{eq:simmetrization}) so as to make
it $(\Igot,\Igot_0)$-invariant. 

Let $r'\in (0,r)$ be as in Proposition \ref{prop:gluing}; 
recall that $r'$ can be chosen to depend only on $F$. Pick a number $\eta\in (0,\epsilon/2)$. 
If $\delta>0$ in \eqref{eq:FG} is chosen sufficiently small, then Proposition \ref{prop:gluing} 
furnishes a $(\Igot,\Igot_0)$-invariant spray $H\colon (K\cup B)\times r'\b^m\to\Agot_*$ of class 
$\Ascr(K\cup B)$ satisfying 
\begin{equation}\label{eq:HFe2}
	||H-F||_{0,K\times r' \b^m}< \eta.
\end{equation} 
Furthermore, the proof of Proposition \ref{prop:gluing} shows that $H|_{B\times r'\B^m}$ is of the form
\[
	H(p,\zeta)=G(p,\zeta+b(p,\zeta)), \quad (p,\zeta)\in B\times r'\b^m
\]
for some holomorphic map $b\colon B\times r'\b^m\to \c^m$ with small sup-norm 
depending only on the constant $\delta$ in \eqref{eq:FG}. Since $G$ is $\delta$-close to $F$
over $S\times r\b^m $ (cf.\ \eqref{eq:FG}), it follows that for $\delta>0$ small enough we also have 
\begin{equation}\label{eq:HFe3}
	||H-F||_{0,(S\cap B) \times r' \b^m} < \eta.
\end{equation}
Since the spray $F$ is period dominating on $S$, the estimates \eqref{eq:HFe2} and \eqref{eq:HFe3}
show that for $\eta>0$ small enough there is  a point $\zeta_0\in \r^m\cap r'\b^m$ such that the map 
$\tilde f= H(\cdot,\zeta_0)\in \Ascr_{\Igot}(K\cup B,\Agot_*)$ satisfies
$\Pcal(\tilde f)=\Pcal(f)$, which means that the $1$-form $(\tilde f -f)\theta$ is exact on $S$. 
By combining the estimates \eqref{eq:Ff}, \eqref{eq:HFe2}, \eqref{eq:HFe3} and $\eta<\epsilon/2$,
we conclude that $||f-\tilde f||_{0,S}<\epsilon$.
The existence of a homotopy $f_t\in \Fscr_\Igot(S, \Agot_*)$ connecting $f_0=f$ to $f_1=\tilde f|_{S}$
can be seen as in the proof of Lemma \ref{lem:bumping}. 
This completes the proof in Case 2.

The case of a general $\Igot$-admissible set 
$S$ follows by combining the above two cases. Indeed, the argument in Case 1 applies to every 
pair of connected components $C$, $\Igot(C)$ of $S$ such that $C\cap\Igot(C)=\emptyset$, whereas the 
argument in Case 2 applies to every $\Igot$-invariant component of $S$.
\end{proof}

%
%

The inductive step in the proof of Theorem \ref{th:Mergelyan0} is given by the following lemma.

\begin{lemma}\label{lem:step2}
Let $j\in\n$. Each map $f_j\in \Oscr_\Igot(M_j,\Agot_*)$ satisfying $\int_\gamma f_j\theta=\qgot(\gamma)$ for all 
$\gamma \in H_1(M_j;\z)$ can be approximated uniformly on $M_j$ by maps $f_{j+1}\in \Oscr_\Igot(M_{j+1},\Agot_*)$ 
satisfying $\int_\gamma f_{j+1}\theta=\qgot(\gamma)$ for all $\gamma \in H_1(M_{j+1};\z)$.
\end{lemma}
\begin{proof}
We distinguish cases.

\vspace{1mm}

\noindent{\em The noncritical case:} $\rho$ has no critical points in $M_{j+1}\setminus \mathring M_j$. 
In this case, $M_j$ is a strong deformation retract of $M_{j+1}$ and the result is given by Proposition \ref{prop:noncritical}.

\vspace{1mm}

\noindent{\em The critical case, index $0$:}  $\rho$ has critical points $p_j$ and $\Igot(p_j)$ in 
$M_{j+1}\setminus \mathring M_j$ with Morse index equal to $0$. In this case, two new simply connected 
components $C$ and $\Igot(C)$ of the sublevel set $\{\rho\leq t\}$ appear at $p_j$ and $\Igot(p_j)$ when $\rho$ passes the 
value $\rho(p_j)=\rho(\Igot(p_j))$, with $C\cap\Igot(C)=\emptyset$. Define $\wt f_j$ on $M_j$ by $\wt f_j=f_j$,
on $C$ as any holomorphic map into $\Agot_*$, and on $\Igot(C)$ so as to make it $\Igot$-invariant, 
i.e. $\wt f_j(p)=\overline{\wt f_j(\Igot(p))}$ for all $p\in \Igot(C)$. This reduces the proof to the noncritical case.

\vspace{1mm}

\noindent{\em The critical case, index $1$:}  
$\rho$ has two critical points $p_j$ and $\Igot(p_j)$ in $\mathring M_{j+1}\setminus M_j$ with Morse index equal to $1$. 
In this case, the change of topology of the sublevel set 
$\{\rho\leq t\}$ at $p_j$ and $\Igot(p_j)$ is described by attaching to $M_j$ smooth disjoint arcs $\alpha$ and $\Igot(\alpha)$ 
in $\mathring M_{j+1}\setminus \mathring M_j$, with endpoints in $bM_j$ and otherwise disjoint from $M_j$. Set 
\begin{equation}\label{eq:Sj}
	  S_j:=M_j\cup\alpha\cup\Igot(\alpha)\Subset M_{j+1}.
\end{equation}
Let $p_\alpha,q_\alpha\in bM_j$ denote the endpoints of $\alpha$; hence, $\Igot(p_\alpha),\Igot(q_\alpha)$ are the 
endpoints of $\Igot(\alpha)$. 

There are four possibilities for the change of topology from $M_j$ to $S_j$, depending 
on the number and character of the connected components of $M_j$ that the arc $\alpha$ is attached to:
\begin{enumerate}[\rm (i)]
\item Assume that $\{p_\alpha,q_\alpha\}$ is contained in a connected component $C$ of $M_j$.
\vspace{1mm}
\begin{enumerate}[\rm ({i}.1)] 
\item If $C$ is $\Igot$-invariant, then $\{\Igot(p_\alpha),\Igot(q_\alpha)\}\subset C$ and $H_1(S_j;\z)$ contains two new 
closed curves which do not lie in  $H_1(M_j;\z)$. A detailed discussion of the topology of $S_j$ in this case can be found 
in \cite[Remark 5.8 and Figures 2, 3, and 4]{AlarconLopez2015GT}.
\vspace{1mm}
\item If $C$ is not $\Igot$-invariant, then $\{\Igot(p_\alpha),\Igot(q_\alpha)\}\subset \Igot(C)$ and $H_1(S_j;\z)$ contains 
two new closed curves which do not lie in  $H_1(M_j;\z)$, one contained in $H_1(C\cup\alpha;\z)$ and the other 
one in $H_1(\Igot(C)\cup\Igot(\alpha);\z)$.  See Figure \ref{fig:i2}.
\end{enumerate}
\vspace{2mm}
\item Assume that the endpoints $p_\alpha$ and $q_\alpha$ of $\alpha$  lie in different connected components 
$C_1$ and $C_2$ of $M_j$. 
\begin{enumerate}[\rm ({ii}.1)] 
\vspace{1mm}
\item If either $C_2=\Igot(C_1)$ or $C_1$ and $C_2$ are $\Igot$-invariant, then $C_1\cup C_2$ is $\Igot$-invariant and 
it follows that $H_1(S_j;\z)$ contains a single new closed curve which does not lie in  $H_1(M_j;\z)$. See Figure \ref{fig:ii1}.
\vspace{1mm}
\item If $C_2\neq \Igot(C_1)$ and either $C_1$ or $C_2$ (or both) is not $\Igot$-invariant, then $C_1\cup C_2$ is not 
$\Igot$-invariant. In this case, no new closed curves appear and the homology of $S_j$ coincides with the one of $M_j$.  
See Figure \ref{fig:ii2}.
\end{enumerate}
\end{enumerate} 
\begin{figure}[ht]
    \begin{center}
    \scalebox{0.4}{\includegraphics{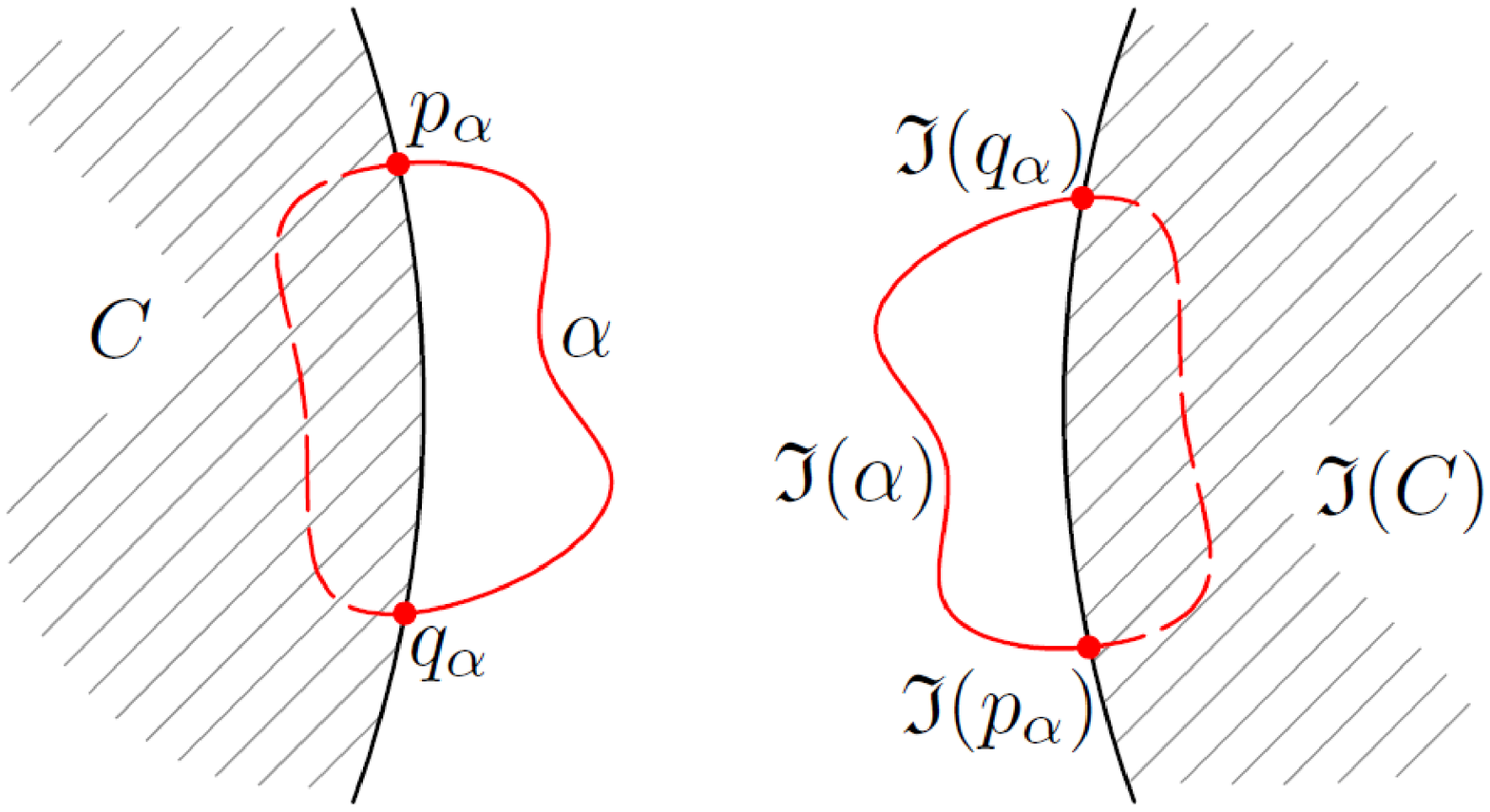}}
        \end{center}
\caption{Case {\rm (i.2).}}\label{fig:i2}
\end{figure}
\begin{figure}[ht]
    \begin{center}
    \scalebox{0.4}{\includegraphics{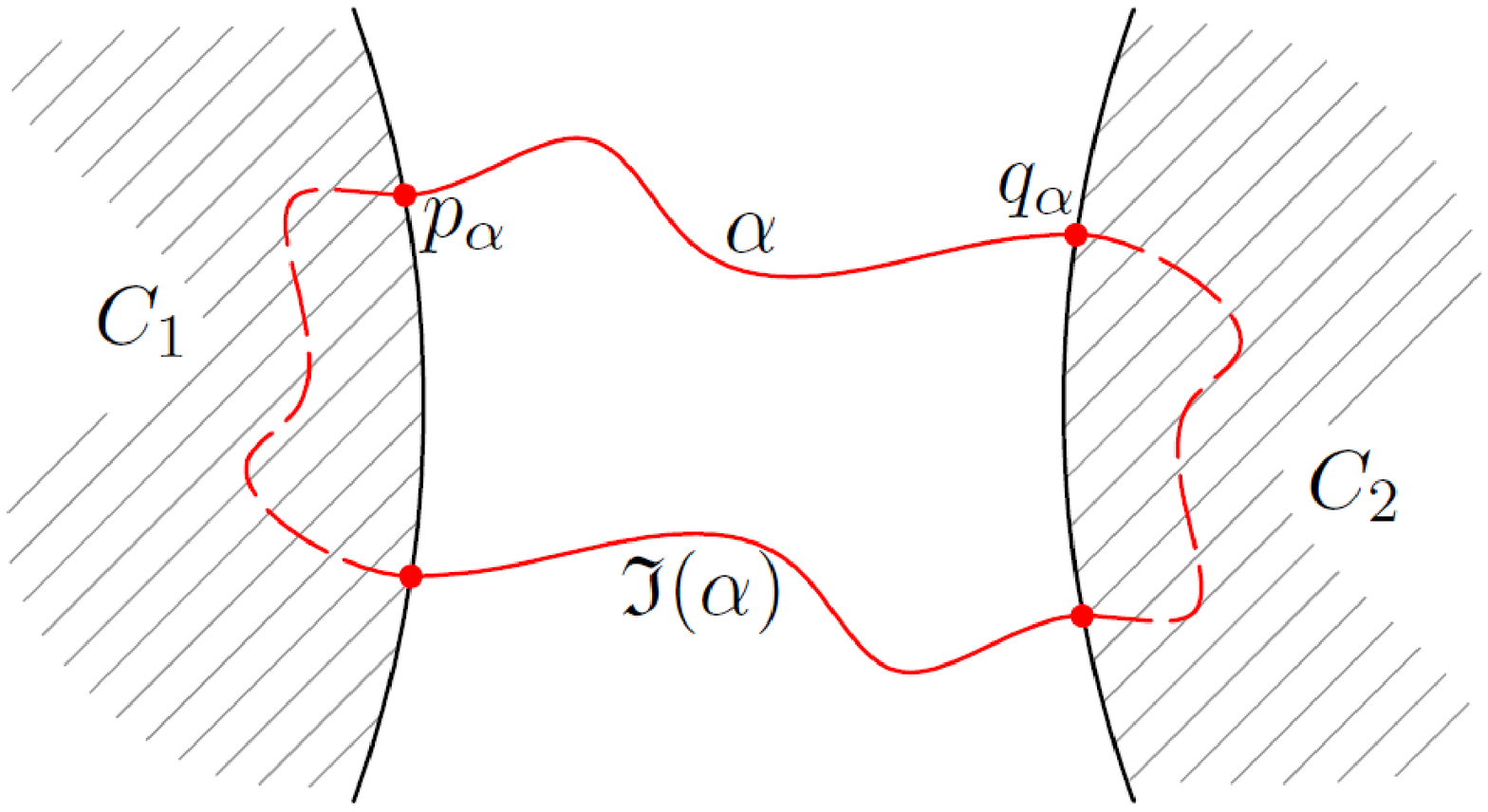}}
        \end{center}
\caption{Case {\rm (ii.1).}}\label{fig:ii1}
\end{figure}
\begin{figure}[ht]
    \begin{center}
    \scalebox{0.34}{\includegraphics{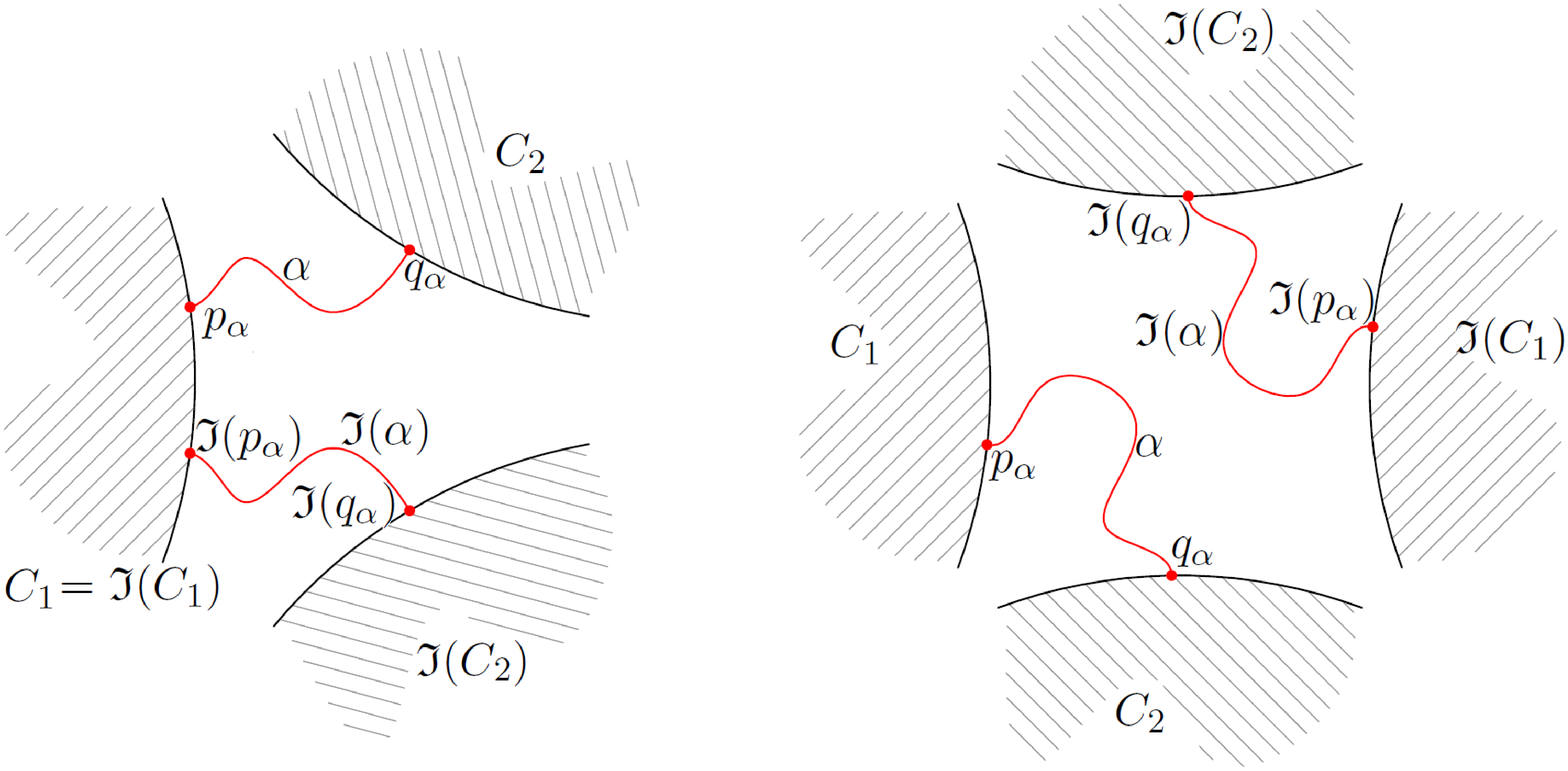}}
        \end{center}
\caption{Case {\rm (ii.2)}. 
Left: $C_1=\Igot(C_1)$ and $C_2\neq\Igot(C_2)$. 
Right: $C_1\neq\Igot(C_1)$ and $C_2\neq\Igot(C_2)$.}\label{fig:ii2}
\end{figure}

In any of the cases described above, the set $S_j$  defined by \eqref{eq:Sj} is a strong deformation retract of $M_{j+1}$, 
and we may further assume that it is $\Igot$-admissible in the sense of Definition \ref{def:admissible}.

Now, we extend $f_j$ to a function $\wt f_j\in \Fscr_\Igot(S_j,\Agot_*)$ such that $\int_\gamma \wt f_j\theta=\qgot(\gamma)$ 
for every closed curve $\gamma\subset S_j$. (In case {\rm (ii.2)} the latter condition holds for any extension 
$\wt f_j\in \Fscr_\Igot(S_j,\Agot_*)$ of $f_j$.) For that, we fix an orientation on $\alpha$, extend $f_j$ to $\alpha$ so that 
$\int_\alpha \wt f_j\theta$ has the required value (see  \cite[proof of Theorem 5.3]{AlarconForstnericLopez2016MZ}), 
and define $\wt f_j(p)=\overline{\wt f_j(\Igot(p))}$ for all $p\in \Igot(\alpha)$. Reasoning as in the proof of 
Lemma \ref{lem:step1} we conclude the proof.
\end{proof}

By using Lemmas \ref{lem:step1} and \ref{lem:step2}, it is straightforward to find a sequence  
$f_j\in \Oscr_\Igot(M_j,\Agot_*)$ converging to a map $h\in \Oscr_\Igot(\Ncal,\Agot_*)$ which 
approximates $f$ on $S$ as close as desired and satisfies $\int_\gamma h\theta=\qgot(\gamma)$ 
for all $\gamma \in H_1(\Ncal;\z)$. Likewise, the homotopies from $f_j$ to $f_{j+1}$ obtained in each step
can be used to get a homotopy $f_t\in \Fscr_\Igot(S, \Agot_*)$ $(t\in [0,1])$ satisfying the conclusion
of Theorem \ref{th:Mergelyan0}.  
\end{proof}

%
%

\section{A Mergelyan theorem with fixed components}
\label{sec:Mergelyan1}

In this section, we prove the following version of Mergelyan theorem for conformal non-orientable minimal surfaces in $\R^n$, 
$n\geq 3$, in which $n-2$ components of the immersion are preserved provided that they extend harmonically to 
the source Riemann surface. This will be useful in subsequent applications. 

%
%

\begin{theorem} 
\label{th:Mergelyan1}
Assume that  $(\Ncal,\Igot,\theta)$ is a standard triple \eqref{eq:triple},
$S=K\cup\Gamma$ is an $\Igot$-admissible subset of $\Ncal$, and $(X,f\theta)\in \GCMII^n(S)$ 
is a generalized $\Igot$-invariant conformal minimal immersion for some $n\ge 3$ (cf.\ Definition \ref{def:GCMI}).
Write $X=(X_j)_{j=1}^n$ and $f=(f_j)_{j=1}^n$. Assume that  for each $j=3,\ldots,n$ the function
$X_j$ extends harmonically to $\Ncal$, $f_j=2\partial X_j/\theta\in\Oscr_\Igot(\Ncal)$, 
and the function $\sum_{j=3}^n f_j^2$   does not vanish identically on $\Ncal$ and has no zeros on $\Gamma$.
Let $\pgot\colon H_1(\Ncal;\z)\to\r^n$ be a group homomorphism such that 
\[
	\pgot\circ \Igot_*=-\pgot \quad \text{and}\quad  
        \pgot(\gamma)=\Flux_{(X,f\theta)}(\gamma)\  \  \text{for every closed curve $\gamma\subset S$}.
\]   
Then, $(X,f\theta)$ can be approximated in the $\Cscr^1(S)$-topology (see Definition \ref{def:appGCMI}) 
by $\Igot$-invariant conformal minimal immersions $\wt X=(\wt X_j)_{j=1}^n\colon \Ncal \to \r^n$ such that  
$\wt X_j=X_j$ on $\Ncal$ for $j=3,\ldots,n$ and $\Flux_{\wt X}=\pgot$.
\end{theorem}

Note that the set $\Gamma$ could be empty; hence, the assumption that $\sum_{j=3}^n f_j^2$ does not vanish anywhere 
on $\Gamma$ need not imply that $\sum_{j=3}^n f_j^2$ is not identically zero on $\Ncal$. 
We point out that the latter condition is necessary in the theorem. Indeed, if $\sum_{j=3}^n f_j^2=0$ 
identically on $\Ncal$ then the arising functions $\wt X_1$ and $\wt X_2$ would be  constant on $\Ncal$ 
(since the only map $(\partial\wt X_1/\theta,\partial\wt X_2/\theta)\in\Oscr_\Igot(\Ncal,\c^2)$ satisfying
$(\partial\wt X_1/\theta)^2+(\partial\wt X_2/\theta)^2=-4\sum_{j=3}^n f_j^2=0$  
identically on $\Ncal$ is the constant $(0,0)\in\c^2$ in view of \eqref{eq:direct sum}).
Hence, their fluxes equal to zero, thereby imposing extra a priori conditions on $X$ and $\pgot$. 

Theorem \ref{th:Mergelyan1} was proved for $n=3$ in \cite{AlarconLopez2015GT} by using the L\'opez-Ros transformation 
for conformal minimal surfaces in $\r^3$, a tool which is not available in higher dimensions. 
On the other hand, a result analogous to Theorem \ref{th:Mergelyan1} in the orientable framework was proved in 
\cite{AlarconFernandezLopez2013CVPDE} for any $n\ge 3$. The deformation procedure used in the latter paper
does not work in the non-orientable case and hence we need here a different approach.

We begin with some preparations.

%
%
\begin{lemma}\label{lem:h}
Let $(\Ncal,\Igot)$ be a standard pair \eqref{eq:pair}.
Every $h\in\Oscr_\Igot(\Ncal) \setminus \{0\}$ can be written in the form 
$h=f_1^2+f_2^2$ with $(f_1,f_2)\in\Oscr_\Igot(\Ncal,\c^2\setminus \{0\})$.  
\end{lemma}

Note that the lemma fails for the zero function $h=0$ since by \eqref{eq:direct sum} 
the only map $(f_1,f_2)\in\Oscr_\Igot(\Ncal,\c^2)$ satisfying $f_1^2+f_2^2=0$ identically on $\Ncal$
is the constant $(0,0)\in\c^2$. 

\begin{proof}
For simplicity of notation we shall identify a path $\gamma$ in $\Ncal$ with its support $|\gamma|$ when
considering maps defined on the set $|\gamma|$.

Fix a Runge $\Igot$-basis $\Bcal=\{\alpha_0,\delta_2,\Igot(\delta_2),\ldots\delta_l,\Igot(\delta_l)\}$ of $H_1(\Ncal;\z)$ as 
in \eqref{eq:basis1}, \eqref{eq:basis2}, \eqref{eq:B0}; in particular, $|\Bcal|$ is a strong deformation retract of $\Ncal$.
Recall that $\alpha_0$ is the only curve in $\Bcal$ such that $\Igot(\alpha_0)=\alpha_0$, and we have 
\begin{equation}\label{eq:disjoint} 
	\text{$\alpha_0\cap \delta_j=\emptyset$,\ \ $\alpha_0\cap \Igot(\delta_j)=\emptyset$,\ \ and
	$\delta_j\cap \Igot(\delta_j)=\emptyset$\ \ for $j=2,\ldots,l$.}
\end{equation}
Since $h$ is holomorphic and does not vanish identically on $\Ncal$, we may assume 
by general position (choosing the curves in $\Bcal$ in a suitable way)
that $h$ has no zeros on $|\Bcal|$. 
Fix an orientation on the closed curve $\alpha_0$. Denote by $k\in \Z$ the topological degree (i.e., the winding number
around the origin) of the function $h|_{\alpha_0} \colon \alpha_0\to \C\setminus\{0\}$.
Since $\Igot\colon \alpha_0\to \alpha_0$ is a diffeomorphism without a fixed point, it has topological degree one. 
Hence, the topological degree of 
$h\circ\Igot\colon \alpha_0\to \C\setminus\{0\}$ equals the degree of $h$ which is $k$. 
However, since $h\circ\Igot=\bar h$ which has topological degree $-k$, we infer that $k=0$.
This implies that $h|_{\alpha_0}\colon\alpha_0\to\c\setminus\{0\}$ has a well defined square root 
$h_0\colon \alpha_0\to \c\setminus\{0\}$. Since $h$ is $\Igot$-invariant and $h_0^2=h|_{\alpha_0}$,
we have that either $h_0\circ \Igot=\overline{h}_0$ (i.e., $h_0$  is $\Igot$-invariant) or 
$h_0\circ \Igot=-\overline{h}_0$ (i.e., $\imath h_0$  is $\Igot$-invariant).

\begin{claim} \label{cla:G}
It suffices to prove the lemma under the assumption that $h_0$ is $\Igot$-invariant.
\end{claim}

\begin{proof} 
Assume for a moment that the lemma holds whenever $h_0\circ \Igot=\overline{h}_0$, and suppose that 
$h_0\circ \Igot=-\overline{h}_0$. By \cite[Lemma 4.6]{AlarconLopez2015GT} there exists 
a function $g\in\Oscr(\Ncal)$ satisfying $g\cdotp (\overline g\circ \Igot)=-1$ on $\Ncal$.
Since  $\imath h_0$ is a $\Igot$-invariant square root of $-h|_{\alpha_0}$, by our assumption there exists a 
map $(u_1,u_2)\in\Oscr_\Igot(\Ncal,\c^2\setminus \{0\})$ satisfying $u_1^2+u_2^2=-h$ on $\Ncal$. 
A calculation then shows that the map $(f_1,f_2)$ given by
\begin{eqnarray*}
	f_1 & = & \frac12 \big((u_1-\imath u_2) g+(u_1+\imath u_2)(\overline{g}\circ \Igot) \big),    \\
	f_2 & = & \frac{\imath}{2} \big((u_1-\imath u_2) g-(u_1+\imath u_2) (\overline{g}\circ \Igot) \big)
\end{eqnarray*}
belongs to $\Oscr_\Igot(\Ncal,\c^2\setminus\{0\})$ and satisfies $f_1^2+f_2^2=h$. This proves the claim.
\end{proof}

Assume now that $h_0$ is an $\Igot$-invariant square root of $h|_{\alpha_0}$. 
Consider the smooth function $v\colon|\Bcal|\to\c$ defined as follows:
\[
	v|_{\alpha_0}=h_0,\quad v|_{\bigcup_{j=2}^l \delta_j}=h,\quad v|_{\bigcup_{j=2}^l \Igot(\delta_j)}=1.
\]
Note that $v$ is well defined in view of \eqref{eq:disjoint} and it satisfies the following equation:
\begin{equation}\label{eq:v}
	v \cdotp (\overline v\circ\Igot)=h|_{|\Bcal|}.
\end{equation}
Since $h$ is $\Igot$-invariant and  does not vanish anywhere on $|\Bcal|$, the divisor of $h$ can be written in the form 
$(h)=D +\Igot(D)$ where
\begin{equation}\label{eq:noceros}
	\text{$\supp(D)\cap \supp(\Igot(D))=\emptyset$\ \ and\ \ $\big(\supp(D)\cup \supp(\Igot(D))\big) \cap |\Bcal|=\emptyset$.}
\end{equation}
By \eqref{eq:v}, $v$ has no zeros on $|\Bcal|$. Since $|\Bcal|$ is Runge in $\Ncal$, 
we may approximate $v$ uniformly on $|\Bcal|$ by holomorphic functions $w\in \Oscr(\Ncal)$ with $(w)=D$. 
It follows that the function 
\[
	\varphi:=\frac{w \cdotp (\overline w \circ \Igot)}{h} \in \Oscr_\Igot(\Ncal,\C\setminus\{0\})
\]
does not vanish anywhere on $\Ncal$ and is close to $1$ on $|\Bcal|$ in view of \eqref{eq:v}. 
If the approximation of $v$ by $w$ is close enough on $|\Bcal|$, 
then $\varphi|_{|\Bcal|}$ is so close to $1$ that it has an $\Igot$-invariant square root. 
As $|\Bcal|$ is a strong deformation retract of $\Ncal$ and $\varphi$ has no zeros on $\Ncal$, 
$\varphi$ has a holomorphic $\Igot$-invariant square root $\varphi_0\colon \Ncal\to \c\setminus\{0\}$ as well. 
Hence, the function $w_0:=w/\varphi_0 \in \Oscr(\Ncal)$ satisfies $w_0 \cdotp (\overline{w_0} \circ \Igot)=h$.
Equivalently, the $\Igot$-invariant holomorphic functions on $\Ncal$ given by
\[
	f_1= \frac12 \big(w_0+\overline{w}_0\circ \Igot \big), \quad 
	f_2=\frac{\imath}{2} \big(w_0-\overline{w}_0\circ \Igot \big),
\]
satisfy $f_1^2+f_2^2=h$. Finally, since $(w_0)=(w)=D$ and $(\overline{w}_0\circ \Igot)=(\overline{w}\circ \Igot)=\Igot(D)$, 
\eqref{eq:noceros} ensures that $f_1$ and $f_2$ have no common zeros. 
\end{proof}

%
%

\begin{lemma} \label{lem:Mergelyan1}
Assume that $(\Ncal,\Igot,\theta)$ is a standard triple \eqref{eq:triple} and the Riemann surface $\Ncal$ has finite topology.
Let $S=K\cup\Gamma\subset \Ncal$ be an $\Igot$-admissible set, and let  $(f_1,f_2)\in \Fscr_\Igot(S,\c^2)$ 
(cf.\ \eqref{eq:Fscr}) be such that the function $h:=f_1^2+f_2^2$ does not vanish anywhere on $bS=\Gamma \cup bK$ 
and  it extends to a function $h\in \OI(\Ncal)$. 
Then $(f_1,f_2)$ can be approximated uniformly on $S$ by $\Igot$-invariant holomorphic maps 
$(\wt f_1,\wt f_2)\colon \Ncal \to \c^2$ satisfying the following conditions: 
\begin{enumerate}[(i)]
\item the $1$-form $(\wt f_j-f_j)\,\theta$ is exact on $S$ for $j=1,2$;
\vspace{1mm}
\item $\wt f_1^2+\wt f_2^2=h$ holds identically on $\Ncal$;
\vspace{1mm}
\item the zero locus of $(\wt f_1,\wt f_2)$ on $\Ncal$ coincides with the zero locus of $(f_1,f_2)$ on $S$. 
In particular, $(\wt f_1,\wt f_2)$ does not vanish anywhere on $\Ncal \setminus \mathring K$.
\end{enumerate}
\end{lemma}

\begin{proof}
Without loss of generality, we may assume that $S$ is a strong deformation retract of $\Ncal$, hence connected. 
Otherwise, we replace $S$ by an $\Igot$-admissible subset $\wt S=S\cup \wt \Gamma$, where $\wt \Gamma$ consists of 
a collection of Jordan arcs such that $\wt S$ is a strong deformation retract of $\Ncal$ and $h$ does not vanish anywhere on 
$\wt \Gamma$ (this is possible by general position). 
Then, we extend $f_1$ and $f_2$ to $\wt S$ as arbitrary functions of class $\Fscr_\Igot(\wt S)$ 
satisfying $f_1^2+f_2^2 =h$.

Since the function  $h\in \OI(\Ncal)$ has no zeros on $bS=bK\cup\Gamma$, there exists a Runge 
$\Igot$-basis $\Bcal$ of $H_1(\Ncal;\z)$ such that $|\Bcal| \subset S$ and  
$h$ has no zeros on $|\Bcal|$ (see \eqref{eq:basis1} and \eqref{eq:basis2}).
Recall that $h|_S=f_1^2+f_2^2$. Lemma \ref{lem:h} provides  a map $(f_3,f_4)\in \Oscr_\Igot(\Ncal,\c^2\setminus\{0\})$ 
such that $f_3^2+f_4^2=-h$.  Set $f =(f_1,f_2,f_3,f_4) \in\Fscr_\Igot(S,\Agot_*)$, where $\Agot_*$ is the 
punctured null quadric in $\C^4$. Note that $f(|\Bcal|)\subset \Agot_*\setminus \{z_1=z_2=0\}$ since $h$ has no
zeros on $|\Bcal|$. The holomorphic vector field
\[
	V=z_1\frac{\partial}{\partial z_2}-z_2\frac{\partial}{\partial z_1}
\]
on $\C^4$ (which is one of the vector fields in \eqref{eq:Vjk}) is complete, tangent to $\Agot_*\subset\c^4$, 
orthogonal to $\partial/\partial z_k$ for $k=3,4$, and it spans the horizontal tangent space
\[
	T_z\Agot_* \cap \{(v_1,v_2,v_3,v_4)\in\c^4 : (v_3,v_4)=0\}
\]
at every point $z\in \Agot_*\setminus \{z_1=z_2=0\}$. The last property is seen by observing
that the above space is $1$-dimensional at every point $z\in \Agot_*\setminus \{z_1=z_2=0\}$ and $V$ does not
vanish anywhere on this set. Let $\phi$ denote the flow of $V$. 
Reasoning as in the proof of Proposition \ref{prop:period-dominating-spray}, 
we find $\Igot$-invariant holomorphic functions $h_1,\ldots,h_m\colon S\to\c$ for some $m\in\n$ such that the spray  
$F=(F_1,F_2,F_3,F_4) \colon S\times U\to\Agot_*$ with the core $f$, given by
\[
	F(p,\zeta_1,\ldots,\zeta_m)=\phi_{_{\sum_{j=1}^m \zeta_j h_j(p)}}(f(p)),\quad p\in S,\; \zeta\in U,
\]
where $U\subset \c^m$ is a ball around $0\in\c^m$, satisfies the following conditions:
\begin{enumerate}[{\rm (i)}]
\item $F$ is smooth, $(\Igot,\Igot_0)$-invariant, and of class $\Ascr(S)$.
\vspace{1mm}
\item $F_j(\cdot,\zeta)=f_j$ for all $\zeta\in U$ and $j\in\{3,4\}$; hence 
\begin{equation}\label{eq:F1F2}	
	F_1(\cdot,\zeta)^2+F_2(\cdot,\zeta)^2=f_1^2+f_2^2=h|_S\quad \text{for all}\ \ \zeta\in U.
\end{equation}
\item $(F_1,F_2)(p,\zeta)=0$ if and only if $(f_1,f_2)(p)=0$ for all $(p,\zeta)\in S\times U$. 

\vspace{1mm}
\item $F$ is {\em period dominating in the first two components}, in the sense that the partial differential
\[
	\di_\zeta \big|_{\zeta=0}  \Pcal((F_1,F_2)(\cdotp,\zeta)) \colon\C^m\to (\c^2)^l 
\]
maps $\R^m$ (the real part of $\C^m$) surjectively onto $\R^2\times (\c^2)^{l-1}$. Here, $\Pcal$ is the period map 
\eqref{eq:periodmap} with respect to the $1$-form $\theta$ and the $\Igot$-basis $\Bcal$.
\end{enumerate}

Recall that the function $h|_S=f_1^2+f_2^2=(f_1-\imath f_2)(f_1+\imath f_2)$ is $\Igot$-invariant and does not 
vanish anywhere on $bS$. 
Let  $D_1=(f_1-\imath f_2)$ denote the divisor of $f_1-\imath f_2$ on $S$; note that $D_1$ is supported on $\mathring K$. 
Since $(f_1-\imath f_2)\circ\Igot = \overline{f_1+\imath f_2}$, we have $(f_1+\imath f_2)=\Igot(D_1)$ and 
\begin{equation}\label{eq:D0}
	D_0: = \supp(D_1)\cap \supp(\Igot(D_1))=\{p\in S \colon f_1(p)=f_2(p)=0\} \subset \mathring K.
\end{equation}
In particular, $(h|_S) = D_1 +\Igot(D_1)$. Since $h$ is $\Igot$-invariant and $\Igot(S)=S$, 
we also have  $(h|_{\Ncal\setminus S})=D_2+\Igot(D_2)$ with $\supp(D_2) \cap \supp(\Igot(D_2)) = \emptyset$. 
Setting $D=D_1+D_2$, we have  $(h)=D +\Igot(D)$ and $\supp(D)\cap \supp(\Igot(D))=D_0$ (cf.\ \ref{eq:D0}).

It follows from \eqref{eq:F1F2} that for every $\zeta\in U$ we have a factorization
\[
	h|_S = 
	\left(F_1(\cdot,\zeta)-\imath F_2(\cdot,\zeta)\right) \cdotp \left(F_1(\cdot,\zeta) + \imath F_2(\cdot,\zeta)\right). 
\]
Since the divisors $D_{1,\zeta} := (F_1(\cdot,\zeta) - \imath F_2(\cdot,\zeta))$ and 
$D_{2,\zeta} := (F_2(\cdot,\zeta) + \imath F_2(\cdot,\zeta))$ on $S$
depend continuously on $\zeta\in U$ and we have $D_{1,\zeta} + D_{2,\zeta} = (h|_S) =D_1+\Igot(D_1)$ 
and $D_{1,0}=D_1$, $D_{2,0}=\Igot(D_1)$, we conclude that
\[
	(F_1(\cdot,\zeta)-\imath F_2(\cdot,\zeta))=D_1 \ \ \text{and} \ \ (F_1(\cdot,\zeta)+\imath F_2(\cdot,\zeta))
	= \Igot(D_1)  \ \ \text{for all $\zeta\in U$}.
\]
By the Oka Property with Approximation and Interpolation (cf.\ \cite[Theorem 5.4.4]{Forstneric2011-book}), 
we may approximate the function $F_1-\imath F_2\colon S\times U\to\c$ uniformly on $S\times U$ by a 
holomorphic function $G\colon \Ncal\times U \to \c$ with the divisor $(G(\cdot,\zeta))=D$ for every $\zeta\in U$.
(The parameter ball $U\subset \C^m$ is allowed to shrink slightly. 
For maps to $\C$, this result reduces to the combination of the classical Cartan extension theorem and the 
Oka-Weil approximation theorem for holomorphic functions on Stein manifolds.) 
By $\Igot$-invariance, it follows that the function $(p,\zeta) \mapsto \overline{G}(\Igot(p),\overline \zeta)$
approximates $F_1+\imath F_2$ uniformly on $S\times U$ and satisfies
$(\overline{G}(\Igot(\cdotp),\bar\zeta))=\Igot(D)$ for every $\zeta\in U$. 
Consider the $(\Igot,\Igot_0)$-invariant holomorphic function $H\colon\Ncal\times U\to\c$ defined by 
\[
	H(p,\zeta)=G(p,\zeta) \cdotp \overline{G}(\Igot(p),\overline\zeta)). 
\]
By the construction, the divisor of $H(\cdotp,\zeta)$ equals $D+\Igot(D)=(h)$ for every $\zeta\in U$.
Hence, the function $\Ncal\times U \ni (p,\zeta)\mapsto h(p)/H(p,\zeta)$ is holomorphic, it 
does not vanish anywhere on $\Ncal\times U$, and it is uniformly close to $1$ on $S\times U$ provided 
that  $G$ is close to $F_1-\imath F_2$ there. Assuming that the approximation is close enough,
the function $h/H$ has a well defined $(\Igot,\Igot_0)$-invariant holomorphic square root $H_0$ on $S\times U$, 
and hence also on $\Ncal\times U$ since $S$ is a strong deformation retract of $\Ncal$. 
For $p\in\Ncal$ and $\zeta\in U$ we set
\begin{eqnarray*}
	\wt F_1(p,\zeta) & = & \frac12 H_0(p,\zeta) \big(G(p,\zeta)+\overline G(\Igot(p),\overline \zeta)\big),            \\
	\wt F_2(p,\zeta) & = & \frac{\imath}2 H_0(p,\zeta) \big( G(p,\zeta)- \overline G(\Igot(p),\overline \zeta)\big). \\
\end{eqnarray*}
Clearly,  $(\wt F_1,\wt F_2)$ is an $(\Igot,\Igot_0)$-invariant holomorphic map satisfying 
\[
	\wt F_1(\cdot,\zeta)^2+\wt F_2(\cdot,\zeta)^2=h \ \ \text{on $\Ncal$}.
\] 
Furthermore, for every $\zeta \in U$ the zero set of the pair $\big(\wt F_1(\cdot,\zeta),\wt F_2(\cdot,\zeta) \big)$ coincides 
with the one of $\big(G(\cdot,\zeta),(\overline G\circ \Igot)(\cdot,\overline \zeta)\big)$, namely $D_0$ (cf.\ \eqref{eq:D0}). 
In other words, the  common zeros of $\wt F_1(\cdot,\zeta)$ and $\wt F_2(\cdot,\zeta)$ on $\Ncal$ are exactly those 
of $(f_1,f_2)$ on $S$ for every $\zeta \in U$.  By the construction, the $(\Igot,\Igot_0)$-invariant spray 
\[
	\wt F=(\wt F_1,\wt F_2, F_3,F_4)\colon\Ncal\times U\to\Agot_*\subset\c^4
\]
approximates $F$ uniformly on $S\times U$. Recall that $F$ is period dominating in the first two
components. Hence, if all approximations made in the process are close enough, 
there is a point $\zeta_0 \in U\cap\r^m$ close to $0$ such that the $\Igot$-invariant holomorphic map 
$(\wt f_1,\wt f_2)=(\wt F_1(\cdot,\zeta_0),\wt F_2(\cdot,\zeta_0)) \colon\Ncal\to \C^2$ 
satisfies the conclusion of Lemma \ref{lem:Mergelyan1}.
\end{proof}

Theorem \ref{th:Mergelyan1} will be obtained as an easy application of the following result
whose proof amounts to a recursive application of Lemma \ref{lem:Mergelyan1}.

\begin{theorem} \label{th:Mergelyan2}
Let $(\Ncal,\Igot,\theta)$ be a standard triple \eqref{eq:triple} and $S=K\cup\Gamma\subset \Ncal$ be an 
$\Igot$-admissible subset. Assume that $f=(f_1,f_2)\in\Fscr_\Igot(S,\c^2)$ is such that $f_1^2+f_2^2$ extends 
to a holomorphic function $h\in \Oscr_\Igot(\Ncal)$ which  does vanish identically on $\Ncal$ and has no zeros on 
$bK\cup \Gamma$, and let $\qgot\colon H_1(\Ncal;\z)\to\c^2$ be a group homomorphism satisfying 
\[
	\qgot\circ \Igot_*=\overline\qgot\quad \text{and} \quad 
	\qgot(\gamma)=\int_\gamma f\theta\ \ \text{for every closed curve $\gamma\subset S$}.
\]
Then, $f$ can be approximated uniformly on $S$ by $\Igot$-invariant holomorphic maps 
$\wt f=(\wt f_1,\wt f_2)\colon \Ncal \to  \c^2$ such that  $\wt f_1^2+\wt f_2^2=h$ on $\Ncal$,  
$\int_\gamma \wt f\theta=\qgot(\gamma)$ for every closed curve $\gamma\subset\Ncal$, 
and  the zero locus  of $\wt f$ on $\Ncal$ equals the zero locus  of $f$ on $S$. 
In particular, $\wt f$ does not vanish anywhere on $\Ncal \setminus \mathring K$.
\end{theorem}

\begin{proof} 
Let $S\Subset M_1\Subset M_2\Subset\cdots\Subset \bigcup_{i=1}^\infty M_i =\Ncal$ 
be an exhaustion of $\Ncal$ as in the proof of Theorem \ref{th:Mergelyan0}. 
Arguing as in the proof of Theorem \ref{th:Mergelyan0} but using Lemma \ref{lem:Mergelyan1} 
instead of Proposition \ref{prop:noncritical}, we may construct a sequence of maps
$g_k:=(g_{k,1},g_{k,2})\in \Oscr_\Igot(M_k,\c^2)$ such that for all $k\in \n$ we have
\begin{itemize}
\item $g_{k,1}^2+g_{k,2}^2=h$ on $M_k$,

\vspace{1mm}
\item $\int_\gamma g_k\theta =\qgot(\gamma)$ for all $\gamma \in H_1(M_k;\z)$, and

\vspace{1mm}
\item the zero divisor  of $g_k$ on $M_k$ equals the divisor of $f$ on $S$ 
(in particular, $g_k$ does not vanish anywhere on $M_k \setminus \mathring K$).
\end{itemize}
Indeed, we can choose all $\Igot$-admissible sets $S_j=M_j\cup\alpha\cup\Igot(\alpha)$ involved in the recursive 
construction  (see \eqref{eq:Sj}) such that $h$ does not vanish anywhere on $bM_j \cup \alpha\cup \Igot(\alpha)$; 
this is possible by general position since $h$ is holomorphic and does not vanish identically on $\Ncal$. 
Moreover, for each $k\in \N$ we may choose $g_{k+1}$ close enough to $g_k$ uniformly on $M_k$
to ensure that the sequence $\{g_k\}_{k\in \n}$ converges to a map $\wt f\in \Oscr_\Igot(\Ncal,\c^2)$ 
which approximates  $f$  on $S$ and the zero locus of $\wt f$ on $\Ncal$ equals the zero locus 
of $f$ on $S$.
\end{proof}


\begin{proof}[Proof of Theorem \ref{th:Mergelyan1}]
Let $(\Ncal,\Igot,\theta)$, $S$, $(X,f\theta)$, and $\pgot=(\pgot_j)_{j=1}^n$ be as in the statement of the theorem. 
By the argument in the proof of Theorem \ref{th:Mergelyan}, we may assume that $S$ is connected. 
Set $h=f_1^2+f_2^2\colon S\to\C$. Then, $h=-\sum_{j=3}^n f_j^2$ which by the assumption extends 
to an $\Igot$-invariant  holomorphic function on $\Ncal$ that does vanish identically 
and has no zeros on $\Gamma$. Up to slightly enlarging the compact set $K$ (and hence $S$), we may also assume that 
$h$ does not vanish anywhere on $bS=(bK)\cup\Gamma$. 
Theorem \ref{th:Mergelyan2}, applied to the map $(f_1,f_2)\colon S\to\C^2$ and the group homomorphism 
$\qgot\colon H_1(\Ncal;\z)\to\c^2$ given by $\qgot(\gamma)=\imath(\pgot_1,\pgot_2)(\gamma)$,
provides maps $(\wt f_1,\wt f_2)\in\Oscr_\Igot(\Ncal,\c^2)$ which 
approximate $(f_1,f_2)$ uniformly on $S$ and satisfy the following conditions:
\begin{itemize}
\item $\wt f_1^2+\wt f_2^2=f_1^2+f_2^2=h$ on $\Ncal$, 
\vspace{1mm}
\item $\int_\gamma (\wt f_1,\wt f_2)\, \theta=\qgot(\gamma)$ for every $\gamma\in H_1(\Ncal;\z)$, and 
\vspace{1mm}
\item the zero locus of $(\wt f_1,\wt f_2)$ on $\Ncal$ is exactly the one of $(f_1,f_2)$ on $S$. 
\end{itemize}
In particular,  the map $\wt f:=(\wt f_1,\wt f_2,f_3,\ldots,f_n)\colon\Ncal\to \Agot_*$ belongs to $\Oscr_\Igot(\Ncal,\Agot_*)$, 
where $\Agot_*$ is the punctured null quadric in $\c^n$. 
It follows that $\Re(\wt f\theta)$ is exact on $\Ncal$. Fix a point $p_0\in S$. The $\Igot$-invariant 
conformal minimal immersions $\wt X\colon\Ncal\to\r^3$ given by $\wt X(p)=X(p_0)+\int_{p_0}^p\wt f\theta$
satisfy the conclusion of the theorem. 
\end{proof}

The following result generalizes Lemma \ref{lem:h} to 
surfaces with arbitrary topology. 

\begin{corollary}\label{cor:h}
Let $(\Ncal,\Igot)$ be a standard pair \eqref{eq:pair}.
Every function $h\in\Oscr_\Igot(\Ncal)\setminus\{0\}$
can be written in the form $h=f_1^2+f_2^2$ with $(f_1,f_2)\in\Oscr_\Igot(\Ncal,\c^2\setminus \{0\})$.  
\end{corollary}

\begin{proof}
Choose any $\Igot$-invariant smoothly bounded compact domain $K\subset \Ncal$. By Lemma \ref{lem:h}, we can find 
$(g_1,g_2)\in \Oscr_\Igot(K,\c^2\setminus\{0\})$ with $g_1^2+g_2^2=h$ on $K$. Thus, Theorem \ref{th:Mergelyan2} applied to 
$(g_1,g_2)$ provides $(f_1,f_2)\in\Oscr_\Igot(\Ncal,\c^2\setminus \{0\})$ satisfying the corollary. 
\end{proof}


\chapter{A general position theorem for non-orientable minimal surfaces}
\label{ch:general-position}

In this chapter, we prove Theorem \ref{th:generalposition} stated in the Introduction.

Let $\UnNcal$ be an open or bordered non-orientable surface  endowed with a conformal structure, 
and let $\pi\colon \Ncal\to \UnNcal$ be an oriented $2$-sheeted covering of $\UnNcal$
with the deck transformation $\Igot$. Then, $\Ncal$ is an open (or bordered) Riemann surface
and $\Igot\colon \Ncal\to\Ncal$ is an antiholomorphic fixed-point-free involution,
and a conformal minimal immersion $\UnX\colon \UnNcal\to\R^n$ is uniquely determined 
by a $\Igot$-invariant conformal minimal immersion $X\colon\Ncal\to\R^n$ by $X=\UnX\circ \pi$
(see Section \ref{sec:conformal}).  Hence, Theorem \ref{th:generalposition} follows immediately from the following
more precise result concerning $\Igot$-invariant conformal minimal immersions $\Ncal\to\R^n$.

\begin{theorem} 
\label{th:generalposition2}
Let $(\Ncal,\Igot)$ be a standard pair \eqref{eq:pair}.
If $n\ge 5$ then every $\Igot$-invariant conformal minimal immersion $X \colon \Ncal \to\r^n$ 
can be approximated uniformly on compacts  by $\Igot$-invariant conformal minimal immersions
$\wt X\colon \Ncal \to\r^n$ such that $\Flux_X=\Flux_{\wt X}$ and for all $p,q\in \Ncal$ we have
\begin{equation}\label{eq:pqI}
	\wt X(p) = \wt X(q)\  \Longleftrightarrow \  p=q\ \ \text{or}\ \ \Igot(p)=q.
\end{equation}
If $n=4$ then the approximating maps $\wt X$ can be chosen such that the 
equation $\wt X(p) = \wt X(q)$ may in addition hold for a 
discrete sequence of pairs $(p_j,q_j)\in \Ncal\times\Ncal$ such that $q_j \notin \{p_j, \Igot(p_j)\}$ for all $j$.

The same result holds if $\Ncal$ is a compact bordered Riemann 
surface and $X$ is of class $\Cscr^r(\Ncal)$ for some $r\in\n$; in such case, 
the approximation takes place in the $\Cscr^r(\Ncal)$ topology, and 
for every $n\ge 3$ we can approximate $X$ by $\Igot$-invariant conformal minimal immersions
$\wt X\colon \Ncal \to\r^n$ satisfying Condition \eqref{eq:pqI} for pairs of points $p,q\in b\Ncal$
in the boundary of $\Ncal$.
\end{theorem}

\begin{proof}
Consider the subset of $\Ncal\times \Ncal$ given by
\begin{equation}\label{eq:setD}
	D=\{(p,q)\in \Ncal\times\Ncal : p=q\ \ \text{or}\ \  \Igot(p)=q\}. 
\end{equation}
Since $\Igot$ has no fixed points, $D$ is a disjoint union $D=D_M\cup D_\Igot$ of the diagonal
$D_M:=\{(p,p)\in \Ncal\times\Ncal : p\in \Ncal\}$ and of $D_\Igot:=\{(p,\Igot(p))\in \Ncal\times\Ncal : p\in \Ncal\}$.

To any $\Igot$-invariant conformal minimal immersion $X\colon \Ncal\to\R^n$ we associate 
the difference map $\delta X\colon \Ncal\times\Ncal\to\R^n$ given by
\[
	\delta X (p,q) = X(p)-X(q), \quad p,q\in \Ncal.	
\]
Then, $X$ satisfies Condition \eqref{eq:pqI} if and only if 
\[
	(\delta X)^{-1}(0) = D.
\]
Likewise, $X$ satisfies Condition \eqref{eq:pqI} on the boundary $b\Ncal$ if and only if 
\[
	(\delta X|_{b\Ncal\times b\Ncal})^{-1}(0) = D \cap (b\Ncal\times b\Ncal).
\]

We first consider the case when $\Ncal$ is a compact Riemann surface with nonempty 
boundary and $X$ is of class $\Cscr^r(\Ncal)$ for some $r\in\n$. 
In view of Corollary \ref{cor:local-Mergelyan} and standard results (see e.g.\ Stout \cite{Stout1965}),  
we may assume that $\Ncal$ is a smoothly bounded 
$\Igot$-invariant domain in an open Riemann surface $\Rcal$, the involution $\Igot$ extends to a
fixed-point-free antiholomorphic involution on $\Rcal$, and $X$ is an $\Igot$-invariant conformal 
minimal immersion on an open neighborhood of  $\Ncal$ in $\Rcal$. 
This allows us to replace the integer $r$ by a possibly bigger number such that the transversality arguments 
used in remainder of the proof apply. 

Since $X$ is an immersion, it is locally injective, and hence there is an open  neighborhood 
$U\subset \Ncal\times \Ncal$ of the set $D$ given by \eqref{eq:setD} such that $\delta X$ does not assume the value 
$0\in \r^n$ on $\overline U\setminus D$. To prove the theorem, it suffices to approximate
$X$ sufficiently closely in $\Cscr^1(\Ncal,\R^n)$ by an $\Igot$-invariant conformal minimal immersion 
$\wt X \colon \Ncal\to\r^n$ with $\Flux_{\wt X}= \Flux_X$
such that the difference map $\delta \wt X\colon \Ncal\times\Ncal\to\R^n$ is transverse to the origin 
$0\in \r^n$ on $\Ncal\times \Ncal\setminus U$, and its restriction to $b\Ncal\times b\Ncal$
is transverse to $0$ on $b\Ncal\times b\Ncal\setminus U$.
Indeed, if $n>4=\dim_\r \Ncal\times \Ncal$, this will imply by dimension reasons that $\delta\wt X\ne 0$ 
on $\Ncal\times \Ncal\setminus U$, and hence $\wt X(p)\ne \wt X (q)$ if $(p,q)\in \Ncal\times \Ncal\setminus U$. 
The same argument shows that $\delta \wt X\ne 0$  on $b\Ncal\times b\Ncal\setminus U$ 
for every $n\ge 3$. Finally, if $n=4$ then $\delta\wt X$ may assume
the value $0\in\R^4$ in at most finitely many points $(p_j,q_j)\in \Ncal\times \Ncal\setminus U$,
and we may further arrange that the tangent planes to the immersed surface $\wt X(\Ncal)\subset \R^4$ at $p_j$ and $q_j$
intersect transversely. If on the other hand $(p,q)\in U \setminus D$, then $\wt X(p)\ne \wt X (q)$  provided that $\wt X$ 
is sufficiently close to $X$ in $\Cscr^1(\Ncal)$. Hence, a map $\wt X$ with these properties will satisfy the conclusion of the theorem.

To find such  maps $\wt X$, we construct a neighborhood $\Omega \subset \r^N$ 
of the origin in a high dimensional Euclidean space and a real analytic map $H\colon \Omega \times \Ncal \to \r^n$ 
satisfying the following properties:
\begin{itemize}
\item[\rm (a)] $H(0,\cdotp)=X$,
\vspace{1mm}
\item[\rm (b)]  for every $\xi \in \Omega$, the map $H(\xi,\cdotp)\colon \Ncal\to \r^n$ is an $\Igot$-invariant 
conformal minimal immersion of class $\Cscr^r(\Ncal)$ whose flux homomorphism equals that of $X$, and 
\vspace{1mm}
\item[\rm (c)]  the difference map $\delta H\colon \Omega \times \Ncal\times \Ncal \to \r^n$, defined by 
\[
	\delta H(\xi,p,q) = H(\xi,q) - H(\xi,p), \qquad \xi\in \Omega, \ \ p,q\in \Ncal, 
\]
is such that the partial differential  $\di_\xi|_{\xi=0} \, \delta H(\xi,p,q) \colon \r^N \to \r^n$
is surjective for every $(p,q)\in \Ncal\times \Ncal\setminus U$. 
\end{itemize}

Assume for a moment that a map $H$ with these properties exists.  It follows from (c) and 
compactness of $\Ncal\times \Ncal \setminus U$ that the differential $\di_\xi (\delta H)$ 
is surjective at all points $\xi$ in a neighborhood $\Omega'\subset \Omega$ of the origin in $\r^N$. 
Hence, the map $\delta H \colon \Omega\times (\Ncal\times \Ncal\setminus U) \to\r^n$ 
is transverse to any submanifold of $\r^n$, in particular, to the origin $0\in \r^n$. 
The standard transversality argument (cf.\ Abraham \cite{Abraham1963BAMS} or \cite[Section 7.8]{Forstneric2011-book}) 
shows that for a generic choice of  $\xi \in\Omega'$,  the  map 
$\delta H(\xi,\cdotp,\cdotp)$ is transverse to $0\in\r^n$ on $\Ncal\times \Ncal\setminus U$
and on $b\Ncal\times b\Ncal\setminus U$. 
By choosing $\xi$ sufficiently close to $0\in\r^N$ we obtain an $\Igot$-invariant conformal minimal immersion 
$\wt X=H(\xi,\cdotp)\colon \Ncal \to \r^n$ close to $X$ 
which satisfies the conclusion of the theorem. 

We can find a map $H$ satisfying properties (a)--(c) by following the construction  in 
\cite[Theorem 4.1]{AlarconForstnericLopez2016MZ}, but with a minor difference that we now explain.
The main step is given by the following lemma.

%
%
%
%
\begin{lemma} \label{lem:pq}
Let $D\subset \Ncal\times \Ncal$ be given by \eqref{eq:setD}.
For every $(p,q)\in \Ncal\times \Ncal\setminus D$ there exists a spray 
$H=H^{(p,q)}(\xi,\cdotp) \colon \Ncal\to \r^n$ of $\Igot$-invariant 
conformal minimal immersions of class $\Cscr^r(\Ncal)$, 
depending analytically on the parameter $\xi$ in a neighborhood of the origin in $\r^n$, 
satisfying properties (a) and (b) above, but with (c) replaced by the following condition:
\begin{itemize}
\item[\rm (c')]   the partial differential $\di_\xi|_{\xi=0} \, \delta H(\xi,p,q) \colon \r^n \to \r^n$ 
is an isomorphism. 
\end{itemize}
\end{lemma}

\begin{proof}
This can be seen by following \cite[proof of Lemma 4.3]{AlarconForstnericLopez2016MZ}, but 
performing all steps by $\Igot$-invariant functions. We include a brief outline.

Since $q\notin \{p,\Igot(p)\}$, there exists a smooth embedded arc $\Lambda\subset\Ncal$ with the endpoints
$p$ and $q$ such that $\Lambda \cap \Igot(\Lambda) =\emptyset$. 
Let $\lambda\colon [0,1]\to \Lambda$ be a smooth parameterization.  
Fix a nowhere vanishing $\Igot$-invariant holomorphic $1$-form $\theta$ on $\Ncal$.
Choose a homology basis $\Bcal=\Bcal^+\cup \Bcal^-$ as in \eqref{eq:B0}, \eqref{eq:B-} 
such that 
\begin{equation}\label{eq:LambdaB}
	|\Bcal|\cap \Lambda=\emptyset \quad \text{and} \quad 
	|\Bcal|\cap \Igot(\Lambda)=\emptyset.
\end{equation}
Let $2\di X=f\theta$, so $f\colon \Ncal\to\Agot_*\subset \C^n$
is an $\Igot$-invariant map of class $\Ascr^{r-1}(\Ncal)$ into the punctured null quadric.
By Proposition \ref{prop:period-dominating-spray}, there are a ball $U\subset\C^m$ around 
$0\in \C^m$ for some big $m\in\N$ and an $(\Igot,\Igot_0)$-invariant period dominating spray 
$F\colon \Ncal\times U\to \Agot_*$ of class $\Ascr^{r-1}(\Ncal)$ with the core $F(\cdotp,0)=f$. More precisely, 
if $\Pcal\colon\Ascr^{r-1}(\Ncal,\C^n)\to (\C^n)^l$ denotes the period map \eqref{eq:periodmap} 
associated to the $1$-form $\theta$ and the homology basis $\Bcal^+$, then the partial differential
\[ 
	\di_\zeta\big|_{\zeta=0}  \Pcal(F(\cdotp,\zeta)) \colon\C^m \lra (\C^n)^l 
\] 
maps $\R^m$ (the real part of $\C^m$) surjectively onto $\R^n\times (\C^n)^{l-1}$. 

Consider $(\Igot,\Igot_0)$-invariant sprays $\Psi \colon \Ncal\times \C^n \to \Agot_*$ 
of the form \eqref{eq:flow-spray2}:
\[ 
	\Psi(x,\xi_1,\ldots,\xi_n) =
	\phi^1_{\xi_1h_1(x)}\circ \phi^2_{\xi_2 h_2(x)}\circ\cdots\circ \phi^n_{\xi_n h_n(x)} (f(x)),
	\quad x\in \Ncal,
\] 
where each $\phi^i$ is the flow of some linear holomorphic vector field $V_i$ of type \eqref{eq:Vjk} on $\C^n$,
$\xi=(\xi_1,\ldots,\xi_n)\in \C^n$, and $h_1,\ldots, h_n$ are $\Igot$-invariant holomorphic 
functions on $\Ncal$ to be determined.  We have that
\[
	\frac{\di \Psi(x,\xi)}{\di \xi_{i}}\bigg|_{\xi=0} = h_{i}(x)\, V_i(f(x)), \quad x\in \Ncal, \ \ i=1,\ldots,n.
\]
We choose the coefficient functions $h_i\in \Ascr^{r-1}_\Igot(\Ncal)$ such that they are very
small on $|\Bcal|$, while on the arc $\Lambda$ they are determined so that the real vectors
\begin{equation}\label{eq:spanning-vectors}
	\Re \int_0^1 h_{i}(\lambda(t))\, V_i(f(\lambda(t))) \, \theta(\lambda(t),\dot\lambda(t)) \,dt \in \R^n,
	\quad i=1,\ldots,n
\end{equation}
span $\R^n$ as a real vector space. The construction of such functions $h_i$ is explained in 
\cite[proof of Lemma 4.3]{AlarconForstnericLopez2016MZ}. We begin by finding smooth
$\Igot$-invariant functions $h_i$ with these properties on the system of curves 
$|\Bcal|\cup \Lambda \cup \Igot(\Lambda)$, taking into account \eqref{eq:LambdaB}
and applying $\Igot$-simmetrization. Since the set $|\Bcal|\cup \Lambda \cup \Igot(\Lambda)$ is Runge in $\Ncal$,
we obtain desired holomorphic functions on $\Ncal$ by Mergelyan approximation theorem.

Since the functions $h_i$ are small on $|\Bcal|$, the complex periods of the maps $\Psi(\cdotp,\xi)\colon\Ncal\to\Agot_*$
are close to those of $f=\Psi(\cdotp,0)=F(\cdotp,0)$. Consider the map
\[
	\Phi(x,\xi,\zeta) = \phi^1_{\xi_1h_1(x)}\circ \cdots\circ \phi^n_{\xi_n h_n(x)} (F(x,\zeta))
\]
defined for $x\in\Ncal$, $\zeta\in U\subset\C^m$ and $\xi\in\C^n$. By the period domination property of $F$
and the implicit function theorem, there is a real valued map $\zeta=\zeta(\xi) \in \R^m$ for $\xi\in \R^n$ 
near the origin, with $\zeta(0)=0$, such that, setting 
\[
	G(x,\xi) = \Phi(x,\xi,\zeta(\xi)) \in \Agot_*,\quad x\in \Ncal,\ \xi\in \R^n\ \text{near}\ 0,
\]
the periods of the $\Igot$-invariant holomorphic maps $G(\cdotp,\xi)\colon \Ncal\to \Agot_*$ over the curves in 
the homology basis $\Bcal^+=\{\delta_1,\ldots,\delta_l\}$ are independent of $\xi$,
and hence they equal those of $G(\cdotp,0)=f$.
In particular, their real periods vanish since $f=2\di X/\theta$. Hence, the integrals
\[
	H(\xi,x) = X(x_0) + \Re \int_{x_0}^x G(\cdotp,\xi)\,\theta,\quad x\in\Ncal
\] 
give a well defined family of $\Igot$-invariant conformal minimal immersions $\Ncal\to\R^n$ 
satisfying $H(0,\cdotp)=X$ and $\Flux_{H(\xi,\cdotp)}=\Flux_X$ for all $\xi\in \R^n$ near $0$. 
Furthermore, 
\[
	\delta H(\xi,p,q)= H(\xi,q) - H(\xi,p)= \Re \int_{\Lambda} G(\cdotp,\xi)\,\theta.
\]
By the chain rule, taking into account that the functions $h_i$ are close to
$0$ on $|\Bcal|$, we see that $\di_{\xi_i} \delta H(\xi,p,q)|_{\xi=0}$ 
is close to the expression in \eqref{eq:spanning-vectors} so that these vectors
for $i=1,\ldots,n$ span $\R^n$. This shows that $H$ satisfies the conclusion of the lemma.
For the details we refer to \cite[proof of Lemma 4.3]{AlarconForstnericLopez2016MZ}.
\end{proof}

The construction of a map $H$ satisfying properties (a)--(c) is then completed as in 
\cite[proof of Theorem 4.1]{AlarconForstnericLopez2016MZ} by taking a composition
of finitely many sprays of the kind furnished by Lemma \ref{lem:pq}.
This completes the proof of Theorem \ref{th:generalposition2} if $\Ncal$ is a bordered 
Riemann surface.  

Assume now that $\Ncal$ is an open Riemann surface, $K$ is a compact set in $\Ncal$, and 
$X\colon \Ncal\to \R^n$ is an $\Igot$-invariant conformal minimal immersion for some
$n\ge 5$. Choose an exhaustion $M_1\Subset M_2\Subset \cdots\Subset \cup_{j=1}^\infty M_j =\Ncal$ 
by compact, smoothly bounded, $\Igot$-invariant Runge domains  such that $K\Subset M_1$.
By the already proved case, we can find $X_1\in \CMI_\Igot^n(M_1)$ (see \eqref{eq:CMI})
which satisfies  Condition \eqref{eq:pqI} on $M_1$, it  approximates $X$ as closely as desired
on $M_1\supset K$, and such that $\Flux_{X_1}=\Flux_{X|_{M_1}}$.
Note that any $X'\in \CMI_\Igot^n(M_1)$ which is sufficiently close to $X_1$ in the 
$\Cscr^1(M_1)$ topology also satisfies Condition \eqref{eq:pqI}.
By the Runge-Mergelyan theorem for $\Igot$-invariant conformal minimal immersions
(cf.\ Theorem \ref{th:Mergelyan}) there exists $X_2\in \CMI_\Igot^n(\Ncal)$  
which approximates $X_1$ as closely as desired on $M_1$ and satisfies 
$\Flux_{X_2|_{M_2}}=\Flux_{X|_{M_2}}$. (Note that we applied the cited theorem with respect to the 
group homomorphism $\pgot = \Flux_X \colon H_1(\Ncal;\z)\to \r^n$.) 
By general position, we may also assume that $X_2$ satisfies Condition \eqref{eq:pqI} on $M_2$.

Clearly this process can be continued recursively, yielding a sequence $X_j\in \CMI^n_\Igot(\Ncal)$ such 
that for every $j=1,2,\ldots$,  $X_j$ satisfies  Condition \eqref{eq:pqI}  on $M_j$, 
$\Flux_{X_j|_{M_j}}=\Flux_{X|_{M_j}} \colon H_1(M_j;\Z)\to \R^n$,
and $X_j$ approximates $X_{j-1}$ on $M_{j-1}$ (the last condition being vacuous for $j=1$).
If the approximation is sufficiently close at every step, then the limit 
$\wt X=\lim_{j\to \infty} X_j \colon \Ncal \to \R^n$ satisfies the conclusion of Theorem \ref{th:generalposition2}.
\end{proof}


\chapter{Applications}
\label{ch:applications}

In this chapter, we give a number of applications of the construction methods obtained in the previous chapters, 
thereby showing their power and versatility. The analogous results for orientable minimal surfaces are already known and, 
with the corresponding construction techniques in hand, the proofs in the non-orientable case are similar to 
those for orientable surfaces.

On the one hand, we prove general existence and approximation results for non-orientable minimal surfaces in $\r^n$, $n\geq 3$, with arbitrary conformal structure and global properties such as completeness or properness, among others; see Sections \ref{sec:Iproper} and \ref{sec:Icomplete}. On the other hand, we furnish several existence and approximation results for complete non-orientable minimal surfaces in $\r^n$ which are bounded or proper in certain subdomains of $\r^n$; see Sections \ref{sec:Jordan} and \ref{sec:proper}.
These boundedness conditions impose restrictions on the source Riemann surface, and hence we consider in this 
case only surfaces normalized by bordered Riemann surfaces.


\section{Proper non-orientable minimal surfaces in $\R^n$}
\label{sec:Iproper}

Due to classical results, it was generally believed 
that properness in $\r^3$ strongly influences the conformal properties of minimal surfaces. 
For instance, Sullivan conjectured that there are no properly immersed minimal surfaces in $\r^3$ with finite topology and hyperbolic 
conformal type. (An open Riemann surface is said to be {\em hyperbolic} if it carries non-constant negative subharmonic functions, 
and {\em parabolic} otherwise; cf.\ \cite[p.\ 179]{FarkasKrabook}.) The first counterexample to this conjecture -- a conformal disk -- was given by 
Morales \cite{Morales2003GAFA}; 
the problem was completely solved by Alarc\'on and L\'opez in \cite{AlarconLopez2012JDG} who 
constructed oriented properly immersed minimal surfaces in $\R^3$ with arbitrary topological and conformal structure. 
In the same line, Schoen and Yau \cite[p.\ 18]{SchoenYau1997IP} 
asked in 1985 whether there exist hyperbolic minimal surfaces in $\r^3$ properly projecting into a plane; a general affirmative answer 
to this question was given in \cite{AlarconLopez2012JDG}. Analogous results 
in the non-orientable framework were obtained in \cite{AlarconLopez2015GT}, using the L\'opez-Ros transformation 
\cite{LopezRos1991JDG} and other ad hoc techniques for the three dimensional case. 

On the other hand, it is well known that $\r^3$ contains no properly embedded non-orientable surfaces
due to topological obstructions.  Although the obstruction disappears in dimension $4$, 
no examples of properly embedded non-orientable minimal surfaces  into $\r^4$ seem to be available in the literature. 

In $\r^4=\c^2$ there are properly embedded complex curves, hence orientable minimal surfaces, 
with arbitrary topology \cite{AlarconLopez2013JGA}, whereas it is not known whether every open Riemann surface 
properly embeds into $\c^2$ as a complex curve or a minimal surface \cite{BellNarasimhan1990EMS};
see \cite{ForstnericWold2009JMPA,ForstnericWold2013APDE} for a discussion and recent results 
on this long-standing open problem.

We begin by providing an example of a properly embedded M\"obius strip in $\R^4$; see 
Example \ref{ex:Mobius} and Figure \ref{fig:strip4}. As mentioned above, this seems to be the first known example 
of a properly embedded non-orientable minimal surface in $\R^4$.


\begin{figure}[h]
    \begin{center}
    \scalebox{0.35}{\includegraphics{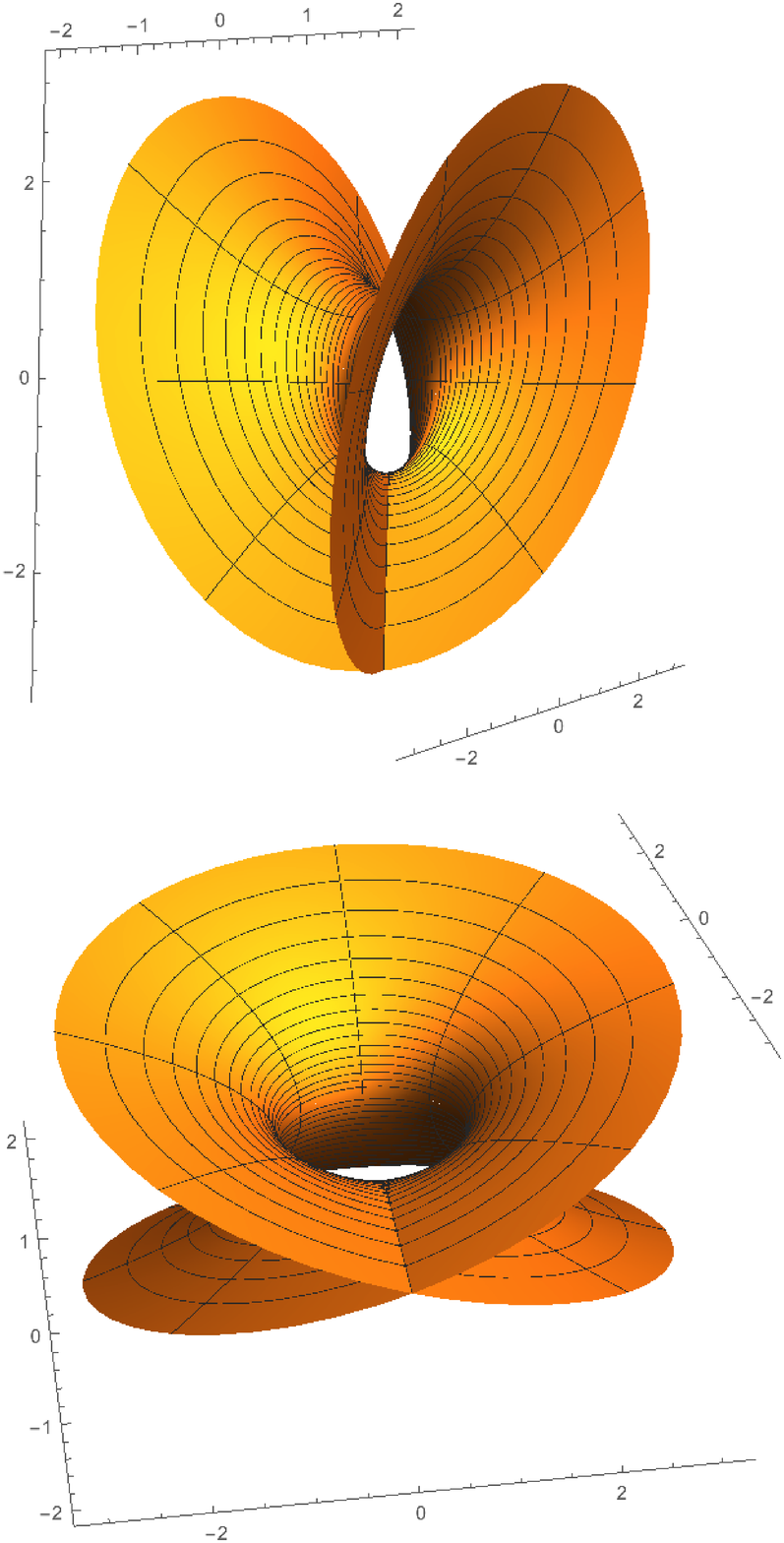}}
        \end{center}
\caption{Two views of the projection into $\r^3=\{0\}\times\r^3\subset\r^4$ of the properly embedded minimal M\"obius strip in $\r^4$ given in Example \ref{ex:Mobius}. This is a harmonically immersed M\"obius strip $\c_*/\Igot\to\r^3$, except at $\zeta=1$ where the surface is not regular. (A harmonic map $u=(u_1,u_2,u_3)\colon M\to\r^3$ from a Riemann surface is an immersion if and only if
$|\sum_{j=1}^3 (\di u_j)^2| < \sum_{j=1}^3 |\di u_j|^2$  everywhere on $M$; cf.\ \cite[Lemma 2.4]{AlarconLopez2013TAMS}.)}
\label{fig:strip4}
\end{figure}

\begin{example}\label{ex:Mobius}
Let $\Igot\colon  \c_*\to\c_*$ be the fixed-point-free antiholomorphic involution 
on the punctured plane $\c_*=\c\setminus\{0\}$ given by 
\[
    	\Igot(\zeta)= -\frac1{\bar\zeta},\quad\  \zeta\in\c_*.
\]
We claim that the harmonic map $X\colon \c_*\to\r^4$ given by
\[
	X(\zeta)=\Re\left( \imath\big(\zeta+\frac1{\zeta}\big) \,,\, \zeta-\frac1{\zeta}  \,,\,
	\frac{\imath}2 \big(\zeta^2-\frac1{\zeta^2}\big)  \,,\, \frac12 \big(\zeta^2+\frac1{\zeta^2}\big) \right) 
\]
is an $\Igot$-invariant proper conformal minimal immersion
such that $X(\zeta_1)=X(\zeta_2)$ if and only if $\zeta_1=\zeta_2$ or $\zeta_1=\Igot(\zeta_2)$, and hence 
its image surface $X(\c_*)$ is a properly embedded minimal M\"obius strip in $\r^4$.

Indeed, a direct computation shows that $X\circ\Igot=\Igot$; i.e.\ $X$ is $\Igot$-invariant. 
Moreover, we have $\di X=f\theta$ where $f$ is the $\Igot$-invariant holomorphic map
\[
    f(\zeta)= \left( \frac{\zeta^2-1}{\zeta} \,,\, -\imath\frac{\zeta^2+1}{\zeta}  \,,\, 
    \frac{\zeta^4+1}{\zeta^2}  \,,\, -\imath\frac{\zeta^4-1}{\zeta^2} \right), \quad \zeta\in\c_*,
\]
and $\theta$ is the $\Igot$-invariant holomorphic $1$-form 
\[
   \theta=\imath \, \frac{d\zeta}{\zeta}
\]
which does not vanish anywhere on $\c_*$.
A calculation shows that $f$ assumes values in the null quadric $\Agot_*\subset\c^4$ \eqref{eq:null-quadric0},
and hence $X$ is a conformal minimal immersion. In polar coordinates $(\rho,t)\in(0,+\infty)\times\r$, we have
\[
    X(\rho e^{\imath t})= \left( \big(\frac1{\rho}-\rho\big)\sin t \,,\, \big(\rho-\frac1{\rho}\big)\cos t  \,,\, 
    -\frac12\big(\rho^2+\frac1{\rho^2}\big)\sin 2t  \,,\, \frac12\big(\rho^2+\frac1{\rho^2}\big)\cos 2t \right).
\]
Since $\|X(\rho e^{\imath t})\|\geq |\rho-1/\rho|$ for all $(\rho,t)\in(0,+\infty)\times\r$ and the function 
$h(\rho)=\rho-1/\rho$ defines a diffeomorphism $h\colon (0,+\infty)\to \r$, we infer that $X$ is a proper map. 
(In fact, the first two components of $X$ provide a proper map $\c_* \to \R^2$.)
Finally, a computation shows that, if $X(\rho_1e^{\imath t_1})= X(\rho_2e^{\imath t_2})$ for 
$(\rho_1, t_1),(\rho_2, t_2)\in (0,+\infty)\times\r$ then either  $\rho_1=\rho_2$ and $t_1=t_2 \mod 2\pi$ 
(and hence $\rho_1e^{\imath t_1}=\rho_2e^{\imath t_2}$), or $\rho_1=\frac1{\rho_2}$ and $t_1=t_2\mod\pi$ 
(and hence $\rho_1e^{\imath t_1}=\Igot(\rho_2e^{\imath t_2})$). 

It follows that $X$ determines a proper minimal embedding 
$\UnX\colon \c_*/\Igot\hra\r^4$ of the M\"obius strip $\c_*/\Igot$ into $\R^4$. See Figure \ref{fig:strip4}.
\qed\end{example}

We now give a general embedding result for non-orientable minimal surfaces in $\r^n$, $n\geq 5$, with arbitrary conformal structure. 
To be more precise, given a standard pair $(\Ncal,\Igot)$ and an integer $n\geq 5$, we provide $\Igot$-invariant conformal minimal 
immersions $Y\colon \Ncal\to\r^n$ such that the image surfaces $Y(\Ncal)\subset\r^n$ are properly embedded non-orientable minimal 
surfaces. This is a consequence of the following general approximation result, which is a more precise version of 
Theorem \ref{th:applications1} {\rm (i)}, also contributing to the Sullivan and Schoen-Yau problems mentioned above.

\begin{theorem}\label{th:Iproper}
Let $(\Ncal,\Igot)$ be a standard pair \eqref{eq:pair}, 
and let $S=K\cup\Gamma\subset\Ncal$ be an $\Igot$-admissible set (see Definition \ref{def:GCMI}). 
Let $(X,f\theta)\in \GCMII^n(S)$ for some $n\ge 3$ (see \eqref{eq:GCMI}), and let $\pgot\colon H_1(\Ncal;\z)\to \r^n$ 
be a group homomorphism such that $\pgot\circ \Igot_*=-\pgot$ and $\pgot|_{H_1(S;\z)}=\Flux_{(X,f \theta)}$ (see \eqref{eq:GFlux}).
Write $X=(X_1,X_2,\ldots,X_n)$ and choose numbers 
\[
	0<\beta < \pi/4, \quad 0<\delta<\min\{(X_2+\tan(\beta)|X_1|)(p) : p\in S\}.
\]
Then, $(X,f \theta)$ may be approximated in the $\Cscr^1(S)$-topology 
(see Definition \ref{def:appGCMI}) 
by $\Igot$-invariant conformal minimal immersions  $Y=(Y_1,Y_2,\ldots,Y_n)\colon \Ncal\to\r^n$ such that 
$\Flux_Y=\pgot$ and $Y_2+\tan(\beta)|Y_1|\colon \Ncal\to\r$ is a proper map bounded below by $\delta$.
In particular, $(Y_1,Y_2)\colon \Ncal\to\r^2$ is also a proper map. 

Furthermore, $Y$ can be chosen such that $\underline{Y}\colon \underline{\Ncal}\to\r^n$ is an embedding
if $n\geq 5$ and an immersion with simple double points if $n=4$. 
\end{theorem}

The case $n=3$ of Theorem \ref{th:Iproper} is proved in \cite[Theorem 6.6]{AlarconLopez2015GT}. The analogue of 
Theorem \ref{th:Iproper} in the orientable case is given in \cite{AlarconLopez2012JDG} for $n=3$ 
and in \cite{AlarconForstnericLopez2016MZ} for arbitrary dimension. In \cite{AlarconForstnericLopez2016MZ} only the properness 
of the pair $(Y_1,Y_2)$ is ensured; we point out that, given $\beta\in(0,\pi/4)$ and an open Riemann surface $M$, 
the arguments in \cite{AlarconLopez2012JDG} can easily be adapted to construct conformal minimal immersions 
$(Y_1,Y_2,\ldots,Y_n)\colon M\to\r^n$, embeddings if $n\geq 5$, such that the map $Y_2+\tan(\beta)|Y_1|\colon \Ncal\to\r$ is proper.

The proof of Theorem \ref{th:Iproper} follows the arguments in \cite[proof of Theorem 6.6]{AlarconLopez2015GT}, but using 
Theorems \ref{th:Mergelyan} and \ref{th:Mergelyan1} instead of \cite[Theorem 5.6]{AlarconLopez2015GT}. We include a 
sketch of the proof  and refer to \cite{AlarconLopez2015GT} for details.

\begin{proof}[Sketch of proof of Theorem \ref{th:Iproper}]
By Theorem \ref{th:Mergelyan0} we may assume that $X$ extends as an $\Igot$-invariant conformal minimal immersion on a 
tubular neighborhood of $S$; hence we may assume without loss of generality that $\Gamma=\emptyset$ and $S$ is a 
smoothly bounded domain.

Let $M_0:=S\Subset M_1\Subset M_2\Subset\cdots\Subset \bigcup_{j=1}^\infty M_j=\Ncal$ be an exhaustion of $\Ncal$ as in the 
proof of Theorem \ref{th:Mergelyan0}. Set $Y^0:=X$. The key of the proof consists of constructing a sequence of $\Igot$-invariant 
conformal minimal immersions 
\[
	Y^j=(Y_1^j,Y_2^j,\ldots,Y_n^j)\colon M_j\to \r^n,\quad j\in\n
\] 
satisfying the following conditions:
\begin{enumerate}[{\rm (i)}]
\item $Y^j$ is as close as desired to $Y^{j-1}$ in the $\Cscr^1(M_{j-1})$ topology.
\item $\Flux_{Y^j}=\pgot|_{H_1(M_j;\z)}$.
\item $Y^j_2+\tan(\beta)|Y^j_1|>\delta+j$ on $bM_j$.
\item $Y^j_2+\tan(\beta)|Y^j_1|>\delta+j-1$ on $M_j\setminus \mathring{M}_{j-1}$.
\item $\underline{Y}^j\colon \underline{\Ncal}\to\r^n$ is an embedding if $n\geq 5$ 
and an immersion with simple double points if $n=4$.
\end{enumerate}
Such a sequence is constructed recursively. For the inductive step we proceed as in 
\cite[proof of Theorem 6.6]{AlarconLopez2015GT} but using Theorems \ref{th:Mergelyan} and \ref{th:Mergelyan1} 
instead of \cite[Theorem 5.6]{AlarconLopez2015GT}, which provides an immersion satisfying conditions {\rm (i)}-{\rm (iv)}. 
The role of the first and third components in \cite[Theorem 6.6]{AlarconLopez2015GT} is played here by the first two ones. 
In particular, crucial steps in the proof of the non-critical case in \cite{AlarconLopez2015GT} require to apply 
Mergelyan approximation to a generalized conformal $\Igot$-invariant minimal immersion into $\r^3$ preserving its component 
in a certain direction (to be precise, in the direction of $(1,0,\tan(\beta))$ or of $(1,0,-\tan(\beta))$; 
cf.\ \cite[Equation $(39)$]{AlarconLopez2015GT}). In our case, we use Theorem \ref{th:Mergelyan1} at these crucial points
to obtain approximation preserving the component in the direction of $(1,\tan(\beta),0,\ldots,0)$ (resp. $(1,-\tan(\beta),0,\ldots,0)$) 
and also any other $n-3$ components among the last $n-2$ ones.
Finally, Property {\rm (v)} is guaranteed by Theorem \ref{th:generalposition}.

If the approximation of $Y^{j-1}$ by $Y^j$ in {\rm (i)} is sufficiently close for all $j\in\n$, the sequence $\{Y^j\}_{j\in\n}$ converges 
uniformly on compacts on $\Ncal$ to an $\Igot$-invariant conformal minimal immersion $Y=(Y_1,Y_2,\ldots,Y_n)\colon\Ncal\to\r^n$ 
satisfying the conclusion of the theorem. Indeed, {\rm (i)} ensures that $Y$ approximates $(X,f\theta)$ as close as desired  in the 
$\Cscr^1(S)$-topology; {\rm (ii)} implies that $\Flux_Y=\pgot$; {\rm (iii)} and {\rm (iv)} guarantee that $Y_2+\tan(\beta)|Y_1|\colon \Ncal\to\r$ 
is a proper map which is bounded below by $\delta$; 
and {\rm (v)} ensures that $Y$ can be chosen such that $\underline{Y}\colon \underline{\Ncal}\to\r^n$ is an embedding
if $n\geq 5$ and  an immersion with simple double points if $n=4$.
\end{proof}

%
%

\section{Complete non-orientable minimal surfaces with fixed components}
\label{sec:Icomplete}

In this section, we prove existence and approximation results for complete non-orientable minimal surfaces in $\r^n$ 
with arbitrary conformal structure and $n-2$ given coordinate functions; see Corollary \ref{co:fixed}. By using this result, 
we then derive the existence of  complete non-orientable minimal surfaces in $\r^n$ with arbitrary conformal structure and 
generalized Gauss map omitting $n$ hyperplanes of $\cp^{n-1}$ in general position; see Corollary \ref{co:gaussmap}.

We begin with the following more precise technical result.

\begin{theorem}\label{th:fixed-derivative}
Let $(\Ncal,\Igot,\theta)$ be a standard triple \eqref{eq:triple} and let $S=K\cup\Gamma\subset \Ncal$ be an $\Igot$-admissible 
subset (see Definition \ref{def:admissible}). 
Let $\kappa\colon\Ncal\to\r_+=[0,+\infty)$ be an $\Igot$-invariant smooth function with isolated zeros. 
Let $f=(f_1,f_2)\in\Fscr_\Igot(S,\c^2)$ (see \eqref{eq:Fscr})  be such that $f_1^2+f_2^2$ extends to a function 
$h\in \Ocal_\Igot(\Ncal)$ which  does not vanish everywhere on $\Ncal$ nor anywhere on $\Gamma$, and let 
$\qgot\colon H_1(\Ncal;\z)\to\c^2$ be a group homomorphism such that $\qgot\circ \Igot_*=\overline\qgot$ and 
$\qgot(\gamma)=\int_\gamma f\theta$ for every closed curve $\gamma\subset S$.

Then, $f$ can be  approximated uniformly on $S$ by $\Igot$-invariant holomorphic maps $\wt f=(\wt f_1,\wt f_2)\colon \Ncal \to  \c^2$ 
such that  $\wt f_1^2+\wt f_2^2=h$ on $\Ncal$,  $\int_\gamma \wt f\theta=\qgot(\gamma)$ for every closed curve $\gamma\subset\Ncal$, 
the zero locus  of $\wt f$ on $\Ncal$ is exactly the one of $f$ on $S$ (in particular, $\wt f$ has no zeros on 
$\Ncal \setminus \mathring K$), and $(|\wt f_1|^2+|\wt f_2|^2+\kappa)|\theta|^2$ is a complete Riemannian metric with 
isolated singularities exactly at the zeros of $|f_1|^2+|f_2|^2+\kappa$.
\end{theorem}

The analogue of  Theorem \ref{th:fixed-derivative} in the orientable case is 
\cite[Theorem 4.4]{AlarconFernandezLopez2013CVPDE}. The proof of Theorem \ref{th:fixed-derivative} follows the same 
arguments as those in \cite[proof of Theorem 4.4]{AlarconFernandezLopez2013CVPDE}, but symmetrizing the construction with respect 
to the involution $\Igot$ and using Theorem \ref{th:Mergelyan2} instead of \cite[Lemma 3.3]{AlarconFernandezLopez2013CVPDE}.
We shall sketch the proof, pointing out the relevant differences and referring to \cite{AlarconFernandezLopez2013CVPDE} for 
complete details. It is also worth mentioning that the bigger generality of Theorem \ref{th:Mergelyan2} when compared to 
\cite[Lemma 3.3]{AlarconFernandezLopez2013CVPDE} (its analogue in the orientable case) enables us to avoid the assumption 
that $f_1$ and $f_2$ are linearly independent (cf.\ \cite[Theorem 4.4 {\rm (A)}]{AlarconFernandezLopez2013CVPDE}). 
Indeed, our proof of Theorem \ref{th:Mergelyan2} can be adapted to the orientable case in order to improve 
\cite[Lemma 3.3]{AlarconFernandezLopez2013CVPDE}, and all the subsequent results in that paper,
by removing the assumption {\rm (A)} from its statement.

\begin{proof}[Sketch of proof of Theorem \ref{th:fixed-derivative}]
By Theorem \ref{th:Mergelyan2} we may assume that $f$ extends as an $\Igot$-invariant holomorphic function on a tubular 
neighborhood of $S$; hence we may assume without loss of generality that $\Gamma=\emptyset$ and $S$ is a smoothly bounded domain.

Let $M_0:=S\Subset M_1\Subset M_2\Subset\cdots\Subset \bigcup_{j=1}^\infty M_j=\Ncal$ be an exhaustion of $\Ncal$ as in 
the proof of Theorem \ref{th:Mergelyan0}. Set $g_0=(g_{0,1},g_{0,2}):=f$ and fix a point $p_0\in\mathring S$. 
The key of the proof is to construct a sequence of $\Igot$-invariant holomorphic maps 
$g_j=(g_{j,1},g_{j,2})\colon M_j\to \c^2$ satisfying the following conditions for every $j\in\n$:
\begin{enumerate}[{\rm (i)}]
\item $g_j$ is as close as desired to $g_{j-1}$ in the $\Cscr^0(M_{j-1})$ topology.
\vspace{1mm}
\item $g_{j,1}^2+g_{j,2}^2=h|_{M_j}$.
\vspace{1mm}
\item $\int_\gamma g_j\theta=\qgot(\gamma)$ for every closed curve $\gamma\subset M_j$.
\vspace{1mm}
\item The zero locus of $g_j$ on $M_j$ is exactly the one of $f$ on $S$. In particular, 
$\sigma(g_j):=(|g_{j,1}|^2+|g_{j,2}|^2+\kappa)|\theta|^2$ is a Riemannian metric on $M_j$ with isolated singularities exactly 
at the zeros of $|f_1|^2+|f_2|^2+\kappa$ on $S$.
\vspace{1mm}
\item $\dist_{(M_j,\sigma(g_j))}(p_0,bM_j)>j$.
\end{enumerate}
To construct such a sequence, we reason as in \cite[proof of Theorem 4.4]{AlarconFernandezLopez2013CVPDE} but symmetrizing 
the argument with respect to the involution $\Igot$ (as in \cite[proof of Theorem 6.8]{AlarconLopez2015GT}) and using 
Theorem \ref{th:Mergelyan2} instead of \cite[Lemma 3.3]{AlarconFernandezLopez2013CVPDE} and 
\cite[Theorem 5.6]{AlarconLopez2015GT}. In particular, for the non-critical case in the recursive construction, 
when $\mathring M_j\setminus M_{j-1}$ consists of finitely many pairwise disjoint annuli, we use $\Igot$-invariant 
Jorge-Xavier's type labyrinths (see \cite{JorgeXavier1980AM}) adapted to the nonnegative $\Igot$-invariant function $\kappa$. 
The latter means that the labyrinth in each component of $\mathring M_j\setminus M_{j-1}$ is contained in an annulus 
where $\kappa$ has no zeros.

If the approximation of each function $g_{j-1}$ by $g_j$ on $M_{j-1}$ in {\rm (i)} is close enough for all $j\in\n$, 
then the sequence $\{g_j\}_{j\in\n}$ converges uniformly on compacts in $\Ncal$ to an $\Igot$-invariant holomorphic map 
$\wt f=(\wt f_1,\wt f_2)\colon\Ncal\to\c^2$ satisfying the conclusion of the theorem. Indeed, {\rm (ii)} ensures that 
$\wt f_1^2+\wt f_2^2=h$; {\rm (iii)} gives that $\int_\gamma \wt f\theta=\qgot(\gamma)$ for every closed curve 
$\gamma\subset\Ncal$; {\rm (iv)} guarantees that the zero locus  of $\wt f$ on $\Ncal$ is exactly the one of $f$ on $S$, 
hence $(|\wt f_1|^2+|\wt f_2|^2+\kappa)|\theta|^2$ is a Riemannian metric on $\Ncal$ with isolated singularities exactly 
at the zeros of $|f_1|^2+|f_2|^2+\kappa$ in $S$; and {\rm (v)} implies that the metric 
$(|\wt f_1|^2+|\wt f_2|^2+\kappa)|\theta|^2$ is complete.
\end{proof}

As a corollary of Theorem \ref{th:fixed-derivative}, we obtain complete non-orientable minimal surfaces 
in $\r^n$ with arbitrary conformal structure and $n-2$ prescribed coordinate functions. We prove the following 
more precise version of Theorem \ref{th:applications1} {\rm (ii)}.

\begin{corollary} \label{co:fixed}
Let $(\Ncal,\Igot)$ be a standard pair \eqref{eq:pair}. Let $n\geq 3$ and let $X=(X_j)_{j=3}^n\colon \Ncal \to \r^{n-2}$ be an
$\Igot$-invariant harmonic map such that $\sum_{j=3}^n \big(\di X_j\big)^2$ does not vanish identically on $\Ncal$. 
Let $\pgot=(\pgot_j)_{j=1}^n\colon H_1(\Ncal;\z)\to\r^n$ be a group homomorphism such that 
$\pgot\circ \Igot_*=-\pgot$ and $\pgot_j(\gamma)=\Im\int_\gamma \di X_j$ for all $j\in\{3,\ldots,n\}$ and all closed curves 
$\gamma\subset \Ncal$. Then the following hold.
\begin{enumerate}[\rm (i)]
\item If $S=K\cup\Gamma$ is an $\Igot$-admissible subset of $\Ncal$ and $(Y,f\theta)\in\GCMI_\Igot^n(S)$ satisfies 
$(Y_j,f_j\theta)=(X_j,2\di X_j)|_S$ for all $j\in\{3,\ldots,n\}$, $\int_\gamma f_j\theta=\pgot_j(\gamma)$ for all $j\in\{1,2\}$ and 
all closed curves $\gamma\subset S$, and $f_1^2+f_2^2\neq 0$ everywhere on $bS$, where $Y=(Y_j)_{j=1}^n$ and $f=(f_j)_{j=1}^n$, 
then $(Y,f\theta)$ can be approximated in the $\Cscr^1(S)$-topology by $\Igot$-invariant complete conformal minimal 
immersions $\wt X=(\wt X_j)_{j=1}^n\colon \Ncal \to \r^n$ such that  $(\wt X_j)_{j=3}^n=X$  and $\Flux_{\wt X}=\pgot$.
\vspace{1mm}
\item In particular, there are $\Igot$-invariant complete conformal minimal immersions $\wt X=(\wt X_j)_{j=1}^n\colon \Ncal \to \r^n$ 
such that  $(\wt X_j)_{j=1}^3=X$  and $\Flux_{\wt X}=\pgot$. 
\end{enumerate}
\end{corollary}

The case $n=3$ of the above result is \cite[Theorem 6.8]{AlarconLopez2015GT}. The analogues of Corollary \ref{co:fixed} in the 
orientable case are \cite[Theorem I]{AlarconFernandezLopez2012CMH} for $n=3$ and 
\cite[Corollary 4.6]{AlarconFernandezLopez2013CVPDE} for arbitrary dimension.

\begin{proof}[Proof of Corollary \ref{co:fixed}]
Let us prove assertion {\rm (i)}. As in the proof of Theorem \ref{th:Mergelyan1}, we may assume that $S$ is connected. 
Set $h:=-\sum_{j=3}^n f_j^2\in \Ocal_\Igot(\Ncal)$ and note that, since $f$ assumes values in the null quadric $\Agot_*\subset\c^n$, 
we have $f_1^2+f_2^2=h|_S$. Set $\qgot:=\imath(\pgot_1,\pgot_2)\colon H_1(\Ncal;\z)\to\c^2$ and 
$\kappa:=\sum_{j=3}^n|f_j|^2\colon \Ncal\to \r_+$.
Obviously the zeros of $\kappa$ are isolated and, since $f$ has no zeros on $S$, the function
\[
	|f_1|^2+|f_2|^2+\kappa =\sum_{j=1}^n |f_j|^2
\]
does not vanish anywhere on $S$. 
Thus, Theorem \ref{th:fixed-derivative} ensures that $(f_1,f_2)$ can be uniformly approximated on $S$ by $\Igot$-invariant 
holomorphic maps $\wt f=(\wt f_1,\wt f_2)\colon \Ncal \to  \c^2$ such that  $\wt f_1^2+\wt f_2^2=-\sum_{j=3}^n f_j^2$ on $\Ncal$,  
$\int_\gamma \wt f\theta=\qgot(\gamma)$ for every closed curve $\gamma\subset\Ncal$, the zero locus  of $\wt f$ on $\Ncal$ 
is exactly the one of $(f_1,f_2)$ on $S$, and $(|\wt f_1|^2+|\wt f_2|^2+\sum_{j=1}^n |f_j|^2)|\theta|^2$ is a complete Riemannian 
metric on $\Ncal$ without singularities. Fix a point $p_0\in\mathring S$. Since we are assuming that $S$ is connected, the 
$\Igot$-invariant conformal minimal immersions $\wt X\colon\Ncal\to\r^n$ given by 
\[
	\wt X(p)=Y(p)+\Re\int_{p_0}^p (\wt f_1,\wt f_2,f_3,\ldots,f_n)\theta
\] 
satisfy the required properties.

In order to prove {\rm (ii)} we take a small compact disk $D\subset\Ncal$ such that $D\cap\Igot(D)=\emptyset$, set $S:=D\cup\Igot(D)$, 
and consider any $\Igot$-invariant holomorphic map $f=(f_1,f_2)\colon S\to\c^2\setminus\{0\}$ such that 
$(f_1^2+f_2^2)\theta=-4(\sum_{j=3}^n \big(\di X_j\big)^2)|_S$; such exists in view of Lemma \ref{lem:h}.
Fix $p_0\in \mathring D$ and, for $j=1,2$, consider the $\Igot$-invariant map $Y_j\colon S\to\r$ defined by 
$Y_j(p)=\Re (\int_{p_0}^p f_j\theta)$ for all $p\in D$. Consider also the $\Igot$-invariant holomorphic map 
$f=\bigl(f_1,f_2, (2\partial X_3)/\theta,\ldots,(2\partial X_n)/\theta\bigr)\colon S \to \Agot_*$. Item {\rm (i)} applied to $S$, 
$\big(Y:=(Y_1,Y_2,X_3,\ldots,X_n),f\theta \big)\in\GCMI_\Igot^n(S)$, and $\pgot$, furnishes an immersion 
$\wt X\colon \Ncal \to \r^n$ with the desired properties.
\end{proof}

We finish this section by proving a more precise version of Theorem \ref{th:applications1} {\rm (iii)}; see 
Corollary \ref{co:gaussmap} below. We first recall the definition and some basics about the generalized Gauss map 
of conformal minimal surfaces in $\r^n$.

Let $n\geq 3$ and denote by $\pi\colon\c^n\to\cp^{n-1}$ the canonical projection. A set of hyperplanes in $\cp^{n-1}$ is said to be in 
{\em general position} if each subset of $k$ hyperplanes, with $k\leq n-1$, has an $(n-1-k)$-dimensional intersection. 

Let $\Ncal$ be an open Riemann surface. A map $\Ncal\to\cp^{n-1}$ is said to be {\em degenerate} if its image lies in a 
proper projective subspace of $\cp^{n-1}$, and {\em nondegenerate} otherwise.

Given a holomorphic $1$-form $\theta$ without zeros on $\Ncal$ and a conformal minimal immersion $X\colon\Ncal\to\r^n$, 
$n\geq 3$, the map
\begin{equation}\label{eq:Gauss}
	G:=\pi\circ\frac{2\partial X}{\theta}\colon\Ncal\to\cp^{n-1} 
\end{equation}
is called the {\em generalized Gauss map} of $X$. Recall  (cf.\ \eqref{eq:Gauss-map}) that this is a Kodaira type map
given in homogeneous coordinates on $\cp^{n-1}$ by the expression
\[
	G(p) = [\di X_1(p):\di X_2(p):\cdots:\di X_n(p)],\quad p\in\Ncal.
\]
Obviously, $G$  assumes values in the projection in $\cp^{n-1}$ of the null quadric $\Agot_*\subset\c^n$. 
Chern and Osserman showed that if $X$ is complete then 
either $X(\Ncal)$ is a plane or  $G(\Ncal)$ intersects a dense set of complex hyperplanes 
(see \cite{Chern1965-book,ChernOsserman1967JAM}). Later, Ru  proved that if $X$ is complete and nonflat then $G$ 
cannot omit more than $n(n + 1)/2$ hyperplanes in general position in $\cp^{n-1}$ (see \cite{Ru1991JDG} and also 
Fujimoto \cite{Fujimoto1983JMSJ,Fujimoto1990JDG} for more information). Moreover, if $G$ is nondegenerate then this 
upper bound is sharp for some values of $n$; for instance for arbitrary odd $n$ and also for any even $n\leq 16$ 
(see \cite{Fujimoto1988SRKU}).  However, the number of exceptional hyperplanes  strongly depends on the underlying 
conformal structure of the surface. For example, Ahlfors \cite{Ahlfors1941ASSF}  proved that every holomorphic map $\c\to\cp^{n-1}$ 
avoiding $n+1$ hyperplanes of $\cp^{n-1}$ in general position is degenerate; see also 
\cite[Chapter 5, \textsection5]{Wu1970-book} and \cite{Fujimoto1972TMJ} for further generalizations. The bound $n+1$ is 
sharp since every open Riemann surface $\Ncal$ admits a complete conformal minimal immersion $\Ncal\to\r^n$ whose 
generalized Gauss map \eqref{eq:Gauss} is nondegenerate and  omits $n$ hyperplanes in general position 
(see \cite[Theorem C]{AlarconFernandezLopez2013CVPDE}). 
This is natural in light of the following recent result of Hanysz 
\cite[Theorem 3.1]{Hanysz2014PAMS}: Let $H_1,\ldots,H_N$ be distinct hyperplanes in $\cp^{n-1}$; 
then $\cp^{n-1}\setminus \bigcup_{j=1}^N H_j$ is an Oka manifold (see \cite{Forstneric2011-book}) if and only if $H_j$ 
are in general position and $N \le n$. On the other hand, the complement of $2n-1$ hyperplanes in general
position in $\cp^{n-1}$ is Kobayashi hyperbolic by the classical result of Green \cite{Green1977PAMS} from 1977.
This implies that any conformal minimal surface in $\R^n$ parametrized by $\C$ or $\C_*$ whose Gauss map omits
$2n-1$ hyperplanes in general position is flat, i.e., its Gauss map is constant.

The following natural question 
remained open in the non-orientable case: do there exist complete 
non-orientable minimal surfaces in $\r^n$ with arbitrary conformal structure whose generalized Gauss map is nondegenerate 
and omits $n$ hyperplanes in general position in $\cp^{n-1}$?  

The following result gives an affirmative answer to this problem. 

\begin{corollary} \label{co:gaussmap}
Let $(\Ncal,\Igot)$ be an open Riemann surface with an antiholomorphic involution without fixed points, let $n\geq 3$, and let
$\pgot\colon H_1(\Ncal;\z)\to\r^n$ be a group homomorphism such that $\pgot\circ \Igot_*=-\pgot$.  Then there exists a full, 
non-decomposable, $\Igot$-invariant complete conformal minimal immersion $X\colon \Ncal \to \r^n$ with $\Flux_X=\pgot$ 
whose generalized Gauss map is nondegenerate and avoids $n$ hyperplanes of $\cp^{n-1}$ in general position. 
\end{corollary}

Recall that a conformal minimal immersion $X\colon\Ncal\to\r^n$ is said to be {\em decomposable} if, with respect to suitable 
orthonormal coordinates $X=(X_j)_{j=1}^n$ in $\r^n$,
one has $\sum_{k=1}^m (\partial X_k)^2=0$ for some $m<n$, and {\em non-decomposable} otherwise. 
Also, $X$ is said to be {\em full} if $X(\Ncal)$ is not contained in any affine hyperplane of $\r^n$. When $n=3$,
the conditions of being non-decomposable, full, and having nondegenerate generalized Gauss map are equivalent.
On the other hand, if $n>3$ then no two of these conditions are equivalent (see \cite{Osserman-book}).

\begin{proof}[Proof of Corollary \ref{co:gaussmap}]
Let $\theta$ be 
an $\Igot$-invariant holomorphic $1$-form without zeros on $\Ncal$. We distinguish cases:

\medskip
\noindent{\em Case 1:} Assume that $n$ is even. Let $D\subset\Ncal$ be a compact disk such that $D\cap\Igot(D)=\emptyset$, 
set $S:=D\cup\Igot(D)$, and fix a point $p_0\in\mathring D$. Choose numbers $\lambda_i\in\r\setminus\{0\}$ and maps
$f_i=(f_{i,1},f_{i,2})\in\Ocal_\Igot(S,\c^2\setminus\{0\})$ such that the following conditions hold:
\begin{enumerate}[\rm (a)]
\item $\sum_{i=1}^{n/2}\lambda_i=0$.
\vspace{1mm}
\item $f_{i,1}^2+f_{i,2}^2=\lambda_j$ everywhere on $S$.
\vspace{1mm}
\item the $\Igot$-invariant conformal minimal immersion $S\to\r^n$ determined by $p\mapsto \Re(\int_{p_0}^p f\theta)$ for all $p\in D$, 
where $f=(f_i)_{i=1}^{n/2}$, is full and non-decomposable and its generalized Gauss map is nondegenerate.
\end{enumerate}
One may for instance define $f_i$ on $D$ to be of the form 
\[
	\left(\frac12 \bigl(g_i+\frac{\lambda_i}{g_i}\bigr), \frac{\imath}2\bigl(g_i-\frac{\lambda_i}{g_i}\bigr)\right),
\] 
where $g_i$ is any holomorphic function $D\to\c\setminus\{0\}$; this guarantees {\rm (b)}. Further, a suitable choice of the numbers 
$\lambda_i$ satisfying {\rm (a)} and of the functions $g_i$ trivially ensures condition {\rm (c)} by general position.

Since $\lambda_i\in\r\setminus \{0\}$, $\lambda_i$ viewed as a constant function $\Ncal\to\c$ lies in $\Ocal_\Igot(\Ncal)$. 
For each $i\in\{1,\ldots,n/2\}$ write $\qgot_i=\imath(\pgot_{2i-1},\pgot_{2i})$, where $\pgot=(\pgot_j)_{j=1}^n$. 
Let $\kappa\colon\Ncal\to\r_+$ be the $\Igot$-invariant constant function $\kappa=|\lambda_1|>0$. By {\rm (b)}, 
Theorem \ref{th:fixed-derivative} ensures that the map $f_i\in\Ocal(\Ncal,\c^2\setminus\{0\})$ may be approximated uniformly 
on $S$ by $\Igot$-invariant holomorphic maps $\wt f_i=(\wt f_{i,1},\wt f_{i,2})\colon\Ncal\to\c^2\setminus \{0\}$ such that 
$\wt f_{i,1}^2+\wt f_{i,2}^2=\lambda_i$ on $\Ncal$, $\int_\gamma\wt f_i\theta=\qgot_i(\gamma)$ for every closed curve 
$\gamma\subset\Ncal$, and $(|\wt f_{i,1}|^2+|\wt f_{i,2}|^2+\kappa)|\theta|^2$ is a complete Riemannian metric on $\Ncal$. 
Taking into account {\rm (a)}, it follows that the $\Igot$-invariant holomorphic map $\wt f:=(\wt f_i)_{i=1}^{n/2}\colon\Ncal\to\c^n$ 
assumes values in $\Agot_*\setminus\{0\}$ and $\Re (\wt f \theta)$ integrates to an $\Igot$-invariant conformal minimal immersion 
$X=(X_j)_{j=1}^n\colon\Ncal\to\r^n$ with $\Flux_X=\pgot$. Furthermore, if the approximation of each $f_i$ by $\wt f_i$ on 
$S$ is close enough then {\rm (c)} implies that $X$ is full and non-decomposable, and its generalized Gauss map $\pi\circ\wt f$ 
is nondegenerate. Moreover, since $\kappa=|\lambda_1|=|\wt f_{1,1}^2+\wt f_{1,2}^2|\leq|\wt f_{1,1}|^2+|\wt f_{1,2}|^2$,
it follows that
\[
	\sum_{j=1}^n|2\di X_j|^2 = \sum_{i=1}^{n/2} \bigl(|\wt f_{i,1}|^2+|\wt f_{i,2}|^2 \bigr)|\theta|^2 \geq 
	\biggl(\kappa+ \sum_{i=2}^{n/2}(|\wt f_{i,1}|^2+|\wt f_{i,2}|^2)\biggr) |\theta|^2
\] 
and hence $\sum_{j=1}^n|2\di X_j|^2$ is a complete Riemannian metric on $\Ncal$ (recall that $n$ is even and so $n\geq 4$); 
thus $X$ is complete. 
Finally, since $(\wt f_{i,1}+\imath \wt f_{i,2})(\wt f_{i,1}-\imath \wt f_{i,2})=\wt f_{i,1}^2+\wt f_{i,2}^2=\lambda_i\neq 0$
everywhere on $\Ncal$, the generalized Gauss map $\pi\circ\wt f\colon\Ncal\to\cp^{n-1}$ of $X$ does not intersect the $n$ hyperplanes 
\[
	H_{i,\delta}:=\big\{\pi(z_1,\ldots,z_n)\in\cp^{n-1}\colon z_{2i-1}+(-1)^\delta \imath z_{2i}=0\big\}
\]
for $i=1,\ldots,n/2$ and $\delta=0,1$ which are in general position. 

\medskip
\noindent{\em Case 2:} Assume that $n$ is odd. In this case we choose $\theta$ as above and satisfying in addition that 
$\int_\gamma\theta=\imath\pgot_n(\gamma)$ for every closed curve $\gamma\subset\Ncal$, where $\pgot=(\pgot_j)_{j=1}^n$; 
such $\Igot$-invariant holomorphic $1$-form without zeros exists by \cite[Theorem 6.4]{AlarconLopez2015GT}. 

Reasoning as in the previous case, we find numbers $\lambda_i\in \r\setminus\{0\}$ and $\Igot$-invariant holomorphic maps 
$\wt f_i=(\wt f_{i,1},\wt f_{i,2})\colon\Ncal\to\c^2\setminus\{0\}$, $i=1,\ldots,(n-1)/2$, such that $\sum_{i=1}^{(n-1)/2}\lambda_i=-1$ 
and $\wt f_{i,1}^2+\wt f_{i,2}^2=\lambda_i$ everywhere on $\Ncal$; hence the map $\wt f:=((\wt f_i)_{i=1}^{(n-1)/2},1)\colon\Ncal\to\c^n$ 
assumes values in the null quadric $\Agot_*\subset\c^n$. Further, we may ensure that $\Re(\wt f\theta)$ integrates to a full, 
non-decomposable, complete $\Igot$-invariant conformal minimal immersion  $X\colon\Ncal\to\r^n$ with $\Flux_X=\pgot$ and 
whose generalized Gauss map $\pi\circ\wt f\colon\Ncal\to\cp^{n-1}$ is nondegenerate.  Finally, observe that $\pi\circ\wt f$ does 
not intersect the $n-1$ hyperplanes
\[
	H_{i,\delta}:=\big\{\pi(z_1,\ldots,z_n)\in\cp^{n-1}\colon z_{2i-1}+(-1)^\delta \imath z_{2i}=0\big\}
\]
for $i=1,\ldots,\frac{n-1}2$ and $\delta=0,1$, nor the hyperplane
\[
	H_0:=\big\{\pi(z_1,\ldots,z_n)\in\cp^{n-1}\colon z_n=0\big\}.
\]
Note that these hyperplanes are in general position. This completes the proof.
\end{proof}

%
%

\section{Complete non-orientable minimal surfaces with Jordan boundaries}
\label{sec:Jordan}

In this section, we provide the first known examples of complete non-orientable minimal surfaces in $\r^n$ with Jordan boundaries. 

Recall that, given a compact bordered Riemann surface $\Ncal$ with a fixed-point-free antiholomorphic involution $\Igot$, 
$\CMI_\Igot^n(\Ncal)$ denotes the set of $\Igot$-invariant conformal minimal immersions 
$\Ncal\to\r^n$ of class $\Cscr^1(\Ncal)$ (see \eqref{eq:CMI}).  Given $X\in\CMI_\Igot^n(\Ncal)$ we denote by $\dist_X$ 
the distance induced on $\Ncal$ by the Euclidean metric of $\r^n$ via the immersion $X$. Recall also that 
$\underline \Ncal=\Ncal/\Igot$ and $\underline X\colon\underline \Ncal\to\r^n$ the induced immersion satisfying $\underline X\circ\pi=X$, 
where $\pi\colon\Ncal\to\underline \Ncal$ is the canonical projection. 

We prove the following more precise version of Theorem \ref{th:intro-Jordan}.

%
%
\begin{theorem}\label{th:Jordan}
Let $(\Ncal,\Igot)$ be a compact bordered Riemann surface with boundary $b\Ncal\ne\emptyset$
and an antiholomorphic involution $\Igot$ without fixed points, and let $n\geq 3$. Every $X\in\CMI_\Igot^n(\Ncal)$ can be approximated 
arbitrarily closely in the $\Cscr^0(\Ncal)$-topology  by $\Igot$-invariant continuous maps 
$Y\colon \Ncal \to \r^n$ such that $Y|_{\mathring \Ncal}\colon\mathring \Ncal \to \r^n$ is a complete 
conformal minimal immersion, $\Flux_Y=\Flux_X$, and $\underline Y|_{b\underline\Ncal}\colon b\underline\Ncal\to \r^n$ 
is a topological embedding. In particular, the boundary $\underline Y(b\underline\Ncal)$ consists of finitely many Jordan curves in $\r^n$.
If $n\ge 5$ then the map $\underline Y\colon \underline\Ncal \to \r^n$ can be chosen a topological embedding.
\end{theorem}

The key tool in the proof of Theorem \ref{th:Jordan} is the Riemann-Hilbert method, given by Theorem \ref{th:RH}, which allows 
one to perturb a given $\Igot$-invariant conformal minimal immersion $\Ncal\to\r^n$ near the boundary, having enough control 
on its placement in $\r^n$ to keep it suitably bounded. This enables us to avoid the typical shrinking of $\Ncal$ in every previous 
construction of complete bounded minimal surfaces in $\r^n$ (see the seminal paper by Nadirashvili \cite{Nadirashvili1996IM},
as well as \cite{AlarconLopez2013MA,AlarconForstneric2015MA} and the references therein for a history of the Calabi-Yau problem),
and hence work with a fixed conformal structure. See also the discussion in Section \ref{sec:intro-Jordan}.

The following technical lemma is crucial in the proof of Theorem \ref{th:Jordan}.

\begin{lemma}\label{lem:C0}
Let $(\Ncal,\Igot)$, $n$, and $X$ be as in Theorem \ref{th:Jordan}, and let $p_0\in\mathring\Ncal$. Given a number $\lambda>0$, 
we may approximate $X$ arbitrarily closely in the $\Cscr^0(\Ncal)$-topology by conformal minimal immersions $Y\in\CMI_\Igot^n(\Ncal)$ 
such that $\Flux_Y=\Flux_X$ and $\dist_Y(p_0,b\Ncal)>\lambda$.
\end{lemma}

\begin{proof}[Sketch of proof of Lemma \ref{lem:C0}]
We choose $\epsilon>0$, take numbers $d_0$ and $\delta_0$ such that $0<d_0<\dist_X(p_0,b\Ncal)$ and $0<\delta_0<\epsilon$, 
and set $c=\frac1{\pi}\sqrt{6(\epsilon^2-\delta_0^2)}>0$. Thus, the sequences $\{d_j\}_{j\in\z_+}$ and $\{\delta_j\}_{j\in\z_+}$ given by
\[
	d_j:=d_{j-1}+\frac{c}j,\quad \delta_j:=\sqrt{\delta_{j-1}^2+\frac{c^2}{j^2}},\quad j\in\n
\]
satisfy
\[
	\{d_j\}_{j\in\z_+}\nearrow +\infty,\quad \{\delta_j\}\nearrow\epsilon.
\]
The proof consists of constructing a sequence of $\Igot$-invariant conformal minimal immersions $X_j\in\CMI_\Igot^n(\Ncal)$ 
satisfying the following conditions:
\begin{enumerate}[\rm (a$_j$)]
\item $\|X_j-X\|_{0,b\Ncal}<\delta_j$.
\vspace{1mm}
\item $\dist_{X_j}(p_0,b\Ncal)>d_j$.
\vspace{1mm}
\item $\Flux_{X_j}=\Flux_X$.
\end{enumerate}
By the Maximum Principle for harmonic maps, the immersion $Y:=X_j$ for sufficiently big $j\in\n$ satisfies the conclusion of the lemma. 

A sequence $X_j$ with these properties is constructed recursively. The basis of the induction is $X_0:=X$, whereas the 
inductive step can be guaranteed by adapting the proof of 
\cite[Lemma 4.2]{AlarconDrinovecForstnericLopez2015PLMS} to the non-orientable framework, 
using Theorems \ref{th:Mergelyan} and \ref{th:RH} 
(the Mergelyan theorem and the Riemann-Hilbert method for $\Igot$-invariant conformal minimal immersions in $\r^n$) instead 
of  \cite[Theorem 5.3]{AlarconForstnericLopez2016MZ} and \cite[Theorem 3.6]{AlarconDrinovecForstnericLopez2015PLMS} 
(their analogues in the orientable framework). In this case, we split the boundary components 
$\alpha_1,\Igot(\alpha_1),\ldots,\alpha_k,\Igot(\alpha_k)$ into compact connected subarcs such that the splittings of $\alpha_i$ 
and $\Igot(\alpha_i)$ are compatible with respect to the involution $\Igot$ 
(cf.\ \cite[Equation {\rm (4.5)}]{AlarconDrinovecForstnericLopez2015PLMS}).
\end{proof}

\begin{proof}[Proof of Theorem \ref{th:Jordan}]
Set $X_0:=X$ and fix a point $p_0\in\mathring \Ncal$. By using Lemma \ref{lem:C0} we may construct a sequence of 
conformal minimal immersions $X_j\in\CMI_\Igot^n(\Ncal)$, $j\in\n$, satisfying the following conditions:
\begin{enumerate}[\rm (a)]
\item $X_j$ approximates $X_{j-1}$ as closely as desired in the $\Cscr^0(\Ncal)$-topology.
\vspace{1mm}
\item $\dist_{X_j}(p_0,b\Ncal)>j$.
\vspace{1mm}
\item $\Flux_{X_j}=\Flux_X$.
\end{enumerate}
Further, by Theorem \ref{th:generalposition} we may also ensure that
\begin{enumerate}[\rm (a)]
\item[\rm (d)] $\underline X_j|_{b\underline\Ncal}\colon b\underline\Ncal\to\r^n$ is a topological embedding, 
and $\underline X_j\colon \underline\Ncal\to\r^n$ is a topological embedding if $n\geq 5$.
\end{enumerate}
If all approximations in {\rm (a)} are close enough then the sequence $\{X_j\}_{j\in\n}$ converges to an $\Igot$-invariant continuous 
map $Y\colon\Ncal\to\r^n$ satisfying the conclusion of the theorem.
See \cite[proof of Theorem 4.5]{AlarconDrinovecForstnericLopez2015PLMS} for further details.
\end{proof}

%
%

\section{Proper non-orientable minimal surfaces in $p$-convex domains}
\label{sec:proper}

In this section, we give several existence and approximation results for complete proper non-orientable minimal surfaces in $(n-2)$-convex 
domains in $\r^n$, $n>3$, or in minimally convex domains in $\r^3$. (See Section \ref{sec:intro-proper} for the definition of a $p$-convex 
domain in $\r^n$, $p\in\{1,\ldots,n-1\}$.) These results contribute to the major problem of  minimal surface theory of determining 
which domains in $\r^3$ admit complete properly immersed minimal surfaces, and how the geometry of the domain influences their 
conformal properties. For background on this topic, see e.g.\ Meeks and P\'erez \cite[Section 3]{MeeksPerez2004SDG} and 
the introduction to \cite{AlarconDrinovecForstnericLopez2015AJM}.

The paper \cite{AlarconDrinovecForstnericLopez2015AJM} contains  general existence and approximation results for proper 
complete minimal surfaces in {\em minimally convex} (i.e.\ $2$-convex domains) of $\r^3$ and $(n-2)$-convex domains of $\r^n$ 
for $n\ge 4$; for the case of convex domains see also \cite{AlarconDrinovecForstnericLopez2015PLMS} and references therein. 
We now provide analogous results for non-orientable minimal surfaces. 
In particular, we may prove the following theorem which includes  Theorem \ref{th:intro3} 
in the introduction as a corollary.

\begin{theorem}\label{th:convex}
Let $n\geq 3$, set 
\[
p=p(n):=\left\{
\begin{array}{ll}
2 & \text{if }n=3,
\\
n-2 & \text{if }n\geq 4,
\end{array}
\right.
\]
and let $(\Ncal,\Igot)$ be a compact bordered Riemann surface with a fixed-point-free antiholomorphic involution. 
Then the following assertions hold.
\begin{enumerate}[\rm (i)]
\item If $D$ is a $p$-convex domain in $\r^n$ then every immersion $X\in\CMI_\Igot^n(\Ncal)$ can be approximated, uniformly on 
compacts in $\mathring\Ncal=\Ncal\setminus b\Ncal$, by proper complete $\Igot$-invariant conformal minimal immersions 
$Y\colon\mathring\Ncal\to D$ with $\Flux_Y= \Flux_X$. 
Furthermore, if $D$ is relatively compact, smoothly bounded and strongly $p$-convex, then $Y$ can be chosen to admit a 
continuous extension to $\Ncal$ satisfying $Y(b\Ncal)\subset b D$ and the estimate
\[
\sup_{z\in \Ncal}\|Y(z)-X(z)\|\leq C\sqrt{\max_{z\in b\Ncal} \dist_{\r^n}(X(z),bD)}
\]
for some constant $C > 0$ depending only on the domain $D$. 
\vspace{1mm}
\item If $\Omega\subset D$ are open sets in $\r^n$ and $\rho\colon\Omega\to(0,+\infty)$ is a smooth strongly $p$-plurisubharmonic 
function such that for any pair of numbers $0 < c_1 < c_2$ the set 
$\Omega_{c_1,c_2} := \{\bx \in\Omega\colon c_1 \le \rho(\bx) \le c_2\}$
is compact, then every $X\in\CMI_\Igot^n(\Ncal)$ satisfying $X(b\Ncal)\subset\Omega$ 
can be approximated, uniformly on compacts in 
$\mathring\Ncal$, by complete $\Igot$-invariant conformal minimal immersions $Y\colon\mathring\Ncal\to D$ such that 
$\Flux_Y= \Flux_X$, $Y(z)\in\Omega$ for every $z\in\mathring\Ncal$ sufficiently close to $b\Ncal$, and
\[
	\lim_{z\to b\Ncal}\rho(Y(z))=+\infty.
\]
\item If $D\subset\r^n$ is a domain having a $\Cscr^2$ strongly $p$-convex boundary point,  then there exist proper, complete, 
$\Igot$-invariant conformal minimal immersions $\mathring\Ncal\to D$ extending continuously to $\Ncal$
and mapping $b\Ncal$ to $bD$.
\vspace{1mm}
\item If $D\subset\r^n$ is a domain having a $p$-convex end in the sense of {\rm (ii)}, or a strongly $p$-convex boundary point, 
then every open non-orientable smooth surface $\underline S$ carries a complete proper minimal immersion $Y\colon \underline S \to D$ 
with arbitrary flux.
\end{enumerate}
Moreover, if $n\ge 5$ then the immersions $Y$ in {\rm (i)}--{\rm (iv)} can be chosen such that the induced
map $\underline Y\colon\UnNcal\to\R^n$ is an embedding.
\end{theorem}

The most general known results in this line up to now are due to Ferrer-Mart\'in-Meeks \cite{FerrerMartinMeeks2012AM} and 
Alarc\'on-L\'opez \cite[Theorem 1.2]{AlarconLopez2014AGMS}, where the conclusion of assertion {\rm (iv)} is proved for convex 
(i.e., $1$-convex) domains in $\r^3$. In particular, Theorem \ref{th:convex} furnishes the first known examples of complete proper 
non-orientable minimal surfaces in a general $p$-convex domain of $\r^n$ for $n=3$ and $p=2$, and for any $n> 3$ and $1<p\le n-2$. 
Recall that the complement of a properly embedded minimal surface in $\r^3$ is a minimally convex domain
(see \cite[Corollary 1.3]{AlarconDrinovecForstnericLopez2015AJM}), 
and hence our results in Theorem \ref{th:convex} apply to this class of domains. 
These are also the first examples of complete 
proper non-orientable minimal surfaces in convex domains in $\r^n$, for any $n\ge 3$, 
with control on the underlying conformal structure.

As usual, we construct the immersions in Theorem \ref{th:convex} by a recursive procedure. 
A main ingredient for the inductive steps is given by the following lemma.

\begin{lemma}
\label{lem:convex}
Let $n$ and $p$ be as in Theorem \ref{th:convex}, let
$D$ be a domain in $\r^n$, and let $\rho\colon D \to \r$ be a $\Cscr^2$ strongly $p$-plurisubharmonic Morse 
function with the (discrete) critical locus $P$. Given a compact set $L \subset D\setminus P$, 
there exist constants $\epsilon_0>0$ and $C_0>0$ such that the following holds.

Let $(\Ncal,\Igot)$ be a compact bordered Riemann surface with a fixed-point-free antiholomorphic involution, and let 
$X\colon \Ncal\to D$ be an $\Igot$-invariant conformal minimal immersion of class $\Cscr^1(\Ncal)$. 
Given an $\Igot$-invariant continuous function $\epsilon \colon b\Ncal\to [0,\epsilon_0]$ supported on the set
$J=\{\zeta\in b\Ncal\colon X(\zeta) \in L\}$, an open set $U\subset \Ncal$ containing $\supp(\epsilon)$ in its relative interior, 
and a constant $\delta>0$, there exists an $\Igot$-invariant conformal minimal immersion 
$Y\colon \Ncal\to D$ satisfying the following conditions:
\begin{enumerate}[\rm (i)]
\item $|\rho(Y(\zeta)) - \rho(X(\zeta)) -\epsilon(\zeta)| <\delta$ for every $\zeta\in b\Ncal$.
\vspace{1mm}
\item $\rho(Y(\zeta))\ge \rho(X(\zeta)) -\delta$ for every $\zeta\in \Ncal$.      
\vspace{1mm}                         
\item $\|Y-X\|_{1,\Ncal\setminus U}<\delta$.  	
\vspace{1mm}			                                         
\item $\|Y-X\|_{0,M} \le C_0 \sqrt{\epsilon_0}$.
\vspace{1mm}							      
\item $\Flux_Y=\Flux_X$.									                               
\end{enumerate} 
\end{lemma}

The proof of Lemma \ref{lem:convex} follows the lines of \cite[proof of Proposition 3.3]{AlarconDrinovecForstnericLopez2015AJM}, 
but symmetrizing the construction with respect to the involution $\Igot$ and using Theorem \ref{th:RH} instead of 
\cite[Theorem 3.2]{AlarconDrinovecForstnericLopez2015AJM} (the Riemann-Hilbert problem for conformal minimal surfaces in $\r^3$). 
In particular, all the sets and maps involved in the proof must be taken to be $\Igot$-invariant. The key of the proof is that for every 
$z\in b\Ncal$ there exists a conformal minimal disk $\alpha_{X(z)}=\alpha_{X(\Igot(z))}\colon \overline\d\to\r^n$, depending smoothly 
on $z$, such that $\alpha_{X(z)}(0)=X(z)$, $\rho(\alpha_{X(z)}(\xi))\geq \rho(X(z))$ for all $\xi\in\overline\d$, and 
$\rho(\alpha_{X(z)}(\xi))$ is suitably larger than $\rho(X(z))$ for $|\xi|=1$ 
(see \cite[proof of Proposition 3.3]{AlarconDrinovecForstnericLopez2015AJM} for more details). 
The existence of such disks when $n=3$ is guaranteed by \cite[Lemma 3.1]{AlarconDrinovecForstnericLopez2015AJM}, 
whereas for $n>3$ one can simply take planar disks in parallel planes (recall that for $n>3$ we have $p=n-2$ and see 
\cite[Remark 3.8]{AlarconDrinovecForstnericLopez2015AJM}). Using such disks as boundary disks of suitable Riemann-Hilbert problems 
in a recursive procedure and applying Theorem \ref{th:RH} yields the result. We refer to 
\cite[proof of Proposition 3.3]{AlarconDrinovecForstnericLopez2015AJM} for further details.

Combining Lemma \ref{lem:C0}, Lemma \ref{lem:convex}, and \cite[Lemmas 3.5 and 3.6]{AlarconDrinovecForstnericLopez2015AJM} 
(the latter ones enable one to avoid the critical points of a Morse exhaustion function $\rho\colon D\to\r$ when applying
Lemma \ref{lem:convex}), we may prove the following result which provides a recursive step in the proof of Theorem \ref{th:convex}.
\begin{lemma}\label{lem:lifting2}
Let $(\Ncal,\Igot)$ be as in Theorem \ref{th:convex}, let $D\subset\r^n$ $(n\ge 3)$ be a domain,
let $\rho$ be a strongly $p$-plurisubharmonic function on $D$,
and let $a<b$ be real numbers such that the set $D_{a,b}=\{x\in D: a<\rho(x) <b\} $ 
is relatively compact in $D$. Given numbers $0<\eta<b-a$ and $\delta>0$, an $\Igot$-invariant conformal minimal immersion 
$X \colon \Ncal\to D$ such that $X(b\Ncal) \subset D_{a,b}$, a point $p_0\in \mathring \Ncal$, a number $d>0$, and a 
compact set $K\subset \mathring \Ncal$, $X$ may be approximated in the $\Cscr^1(K)$-topology by $\Igot$-invariant 
conformal minimal immersions $Y\colon \Ncal \to D$ satisfying the following conditions: 
\begin{enumerate}[(i)]
\item[$\bullet$]   $Y(b\Ncal) \subset D_{b-\eta,b}$.
\vspace{1mm}
\item[$\bullet$]   $\rho(Y(\zeta))\ge \rho(X(\zeta)) -\delta$ for every $\zeta\in \Ncal$.
\vspace{1mm}
\item[$\bullet$]   $\dist_{Y}(p_0,b\Ncal) > d$.
\vspace{1mm}
\item[$\bullet$]   $\Flux_Y=\Flux_X$.
\end{enumerate}
\end{lemma}

The proof of Lemma \ref{lem:lifting2} follows word by word the one of \cite[Lemma 3.7]{AlarconDrinovecForstnericLopez2015AJM}, 
so we omit it.

With Lemmas \ref{lem:convex} and \ref{lem:lifting2} in hand, one may prove Theorem \ref{th:convex} by following the arguments in 
\cite[proof of Theorems 1.1, 1.7, and 1.9]{AlarconDrinovecForstnericLopez2015AJM}, but using Theorem \ref{th:RH} instead of 
\cite[Theorem 3.2]{AlarconDrinovecForstnericLopez2015AJM}, Lemma \ref{lem:C0} instead of 
\cite[Lemma 4.1]{AlarconDrinovecForstnericLopez2015PLMS}, Lemma \ref{lem:convex} instead of 
\cite[Proposition 3.3]{AlarconDrinovecForstnericLopez2015AJM}, and also Theorem \ref{th:generalposition2} 
in order to guarantee the embeddednes of $\underline Y$ if $n\ge 5$.

For instance, the proof of Theorem \ref{th:convex} {\rm (i)} goes as follows. We set $X_0:=X$, choose a strongly 
$p$-plurisubharmonic Morse exhaustion function $\rho\colon D\to\r$, pick suitable increasing sequences  $a_1<a_2<a_3\cdots$ 
and $d_1<d_2<d_3\cdots$ such that $\sup_\Ncal \rho\circ X_0 < a_1$, $\lim_{j\to\infty} a_j=+\infty$, and 
$\lim_{j\to\infty} d_j=+\infty$, and take a decreasing sequence $\delta_j>0$ with 
$\delta=\sum_{j=1}^\infty \delta_j <\infty$. We then fix a point $p_0\in \mathring \Ncal$ and use the cited results to construct a 
sequence of smooth $\Igot$-invariant conformal minimal immersions $X_j\colon \Ncal\to D$, and an increasing sequence 
of $\Igot$-invariant compacts 
$K_0\subset K_1\subset  \cdots\subset \bigcup_{j=1}^\infty K_j=\mathring \Ncal$ with $p_0\in\mathring K_0$, 
such that the following conditions hold for every $j=1,2,\ldots$:
\begin{enumerate}[\rm (i)]
\item   $a_j < \rho\circ X_{j} < a_{j+1}$ on $\Ncal\setminus K_j$. 
\vspace{1mm}
\item  $\rho\circ X_j > \rho\circ X_{j-1}-\delta_j$ on $\Ncal$.           
\vspace{1mm}
\item   $X_j$ is close to $X_{j-1}$ in the $\Cscr^1(K_{j-1})$-topology.
\vspace{1mm}
\item   $\dist_{X_j}(p_0,\Ncal\setminus K_j) > d_j$.                         
\vspace{1mm}
\item    $\Flux_{X_j}=\Flux_{X_{j-1}}$.
\end{enumerate}
If $D$ is relatively compact, smoothly bounded, and strongly $p$-convex then we may also ensure that
\begin{enumerate}[\rm (i)]
\item[\rm (vi)]  $\|X_j-X_{j-1}\|_{0,\Ncal} \le C_0\sqrt{2}^{-j} \sqrt{|a_0|}$ for some constant $C_0>0$ which does not depend on $j$.
\end{enumerate}
If the all the approximations in {\rm (iii)} are close enough then the sequence $X_j$ converges to an $\Igot$-invariant conformal minimal immersion $Y\colon \mathring\Ncal\to D$ satisfying the conclusion of the theorem. 

The proofs of Theorem \ref{th:convex} {\rm (ii)}, {\rm (iii)}, and {\rm (iv)} follow similar basic patterns. We again refer to  \cite[Proof of Theorems 1.1, 1.7, and 1.9]{AlarconDrinovecForstnericLopez2015AJM} and leave the details to the interested reader.


\appendix

\backmatter

\providecommand{\bysame}{\leavevmode\hbox to3em{\hrulefill}\thinspace}
\providecommand{\MR}{\relax\ifhmode\unskip\space\fi MR }
\providecommand{\MRhref}[2]{%
  \href{http://www.ams.org/mathscinet-getitem?mr=#1}{#2}
}
\providecommand{\href}[2]{#2}

\printindex

\end{document}